\documentclass[12pt,british]{amsart}
\usepackage{amsmath}
\usepackage{libertine}
\usepackage[libertine]{newtxmath}
\usepackage{mathrsfs}

\usepackage[T1]{fontenc}
\usepackage[lowtilde]{url}
 \usepackage{hyperref}
 \usepackage{babel} 
\usepackage{amssymb,latexsym}
\usepackage{cite,amsmath,amstext,amsbsy,bm,url} 
\usepackage{paralist}
\usepackage{datetime} 
\usepackage{colortbl}
\usepackage[dvipsnames]{xcolor}
\usepackage{mathtools}
\usepackage{euscript}
\usepackage[all]{xy}
\usepackage{graphicx} 
\usepackage[a4paper,text={170mm,260mm},centering]{geometry} 
\usepackage{comment}  
\usepackage{cite}  
\usepackage{thmtools}
\usepackage{enumitem}
\usepackage{letltxmacro} 
\usepackage{nameref}
\usepackage{cleveref}
\usepackage{calc}
\usepackage{interval}
\usepackage{manfnt}
\usepackage{tikz-cd}  
\usepackage[utf8]{inputenc}
\usepackage{newunicodechar} 
\usepackage{csquotes}

\setdescription{labelindent=\parindent, leftmargin=2\parindent}
\setitemize[1]{labelindent=\parindent, leftmargin=2\parindent}
\setenumerate[1]{labelindent=0cm, leftmargin=*, widest=iiii}

\newunicodechar{α}{\ensuremath{\alpha}}
\newunicodechar{β}{\ensuremath{\beta}}
\newunicodechar{χ}{\ensuremath{\chi}}
\newunicodechar{δ}{\ensuremath{\delta}}
\newunicodechar{ε}{\ensuremath{\varepsilon}}
\newunicodechar{Δ}{\ensuremath{\Delta}}
\newunicodechar{η}{\ensuremath{\eta}}
\newunicodechar{γ}{\ensuremath{\gamma}}
\newunicodechar{Γ}{\ensuremath{\Gamma}}
\newunicodechar{ι}{\ensuremath{\iota}}
\newunicodechar{κ}{\ensuremath{\kappa}}
\newunicodechar{λ}{\ensuremath{\lambda}}
\newunicodechar{Λ}{\ensuremath{\Lambda}}
\newunicodechar{ν}{\ensuremath{\nu}}
\newunicodechar{μ}{\ensuremath{\mu}}
\newunicodechar{ω}{\ensuremath{\omega}}
\newunicodechar{Ω}{\ensuremath{\Omega}}
\newunicodechar{π}{\ensuremath{\pi}}
\newunicodechar{Π}{\ensuremath{\Pi}}
\newunicodechar{φ}{\ensuremath{\phi}}
\newunicodechar{Φ}{\ensuremath{\Phi}}
\newunicodechar{ψ}{\ensuremath{\psi}}
\newunicodechar{Ψ}{\ensuremath{\Psi}}
\newunicodechar{ρ}{\ensuremath{\rho}}
\newunicodechar{σ}{\ensuremath{\sigma}}
\newunicodechar{Σ}{\ensuremath{\Sigma}}
\newunicodechar{τ}{\ensuremath{\tau}}
\newunicodechar{θ}{\ensuremath{\theta}}
\newunicodechar{Θ}{\ensuremath{\Theta}}
\newunicodechar{ξ}{\ensuremath{\xi}}
\newunicodechar{Ξ}{\ensuremath{\Xi}}
\newunicodechar{ζ}{\ensuremath{\zeta}}

\newunicodechar{ℓ}{\ensuremath{\ell}}

\newunicodechar{𝔹}{\ensuremath{\bB}}
\newunicodechar{ℂ}{\ensuremath{\bC}}
\newunicodechar{𝔻}{\ensuremath{\bD}}
\newunicodechar{𝔼}{\ensuremath{\bE}}
\newunicodechar{𝔽}{\ensuremath{\bF}}
\newunicodechar{ℕ}{\ensuremath{\bN}}
\newunicodechar{ℙ}{\ensuremath{\bP}}
\newunicodechar{ℚ}{\ensuremath{\bQ}}
\newunicodechar{ℝ}{\ensuremath{\bR}}
\newunicodechar{𝕏}{\ensuremath{\bX}}
\newunicodechar{ℤ}{\ensuremath{\bZ}}
\newunicodechar{𝒜}{\ensuremath{\sA}}
\newunicodechar{ℬ}{\ensuremath{\sB}}
\newunicodechar{𝒞}{\ensuremath{\sC}}
\newunicodechar{𝒟}{\ensuremath{\sD}}
\newunicodechar{ℰ}{\ensuremath{\sE}}
\newunicodechar{ℱ}{\ensuremath{\sF}}
\newunicodechar{𝒢}{\ensuremath{\sG}}
\newunicodechar{ℋ}{\ensuremath{\sH}}
\newunicodechar{𝒥}{\ensuremath{\sJ}}
\newunicodechar{ℒ}{\ensuremath{\sL}}
\newunicodechar{𝒪}{\ensuremath{\sO}}
\newunicodechar{𝒬}{\ensuremath{\sQ}}
\newunicodechar{𝒯}{\ensuremath{\ts}}
\newunicodechar{𝒲}{\ensuremath{\sW}}

\newunicodechar{∂}{\ensuremath{\partial}}
\newunicodechar{∇}{\ensuremath{\nabla}}

\newunicodechar{↺}{\ensuremath{\circlearrowleft}}
\newunicodechar{∞}{\ensuremath{\infty}}
\newunicodechar{⊕}{\ensuremath{\oplus}}
\newunicodechar{⊗}{\ensuremath{\otimes}}
\newunicodechar{•}{\ensuremath{\bullet}}
\newunicodechar{Λ}{\ensuremath{\wedge}}
\newunicodechar{↪}{\ensuremath{\into}}
\newunicodechar{→}{\ensuremath{\to}}
\newunicodechar{↦}{\ensuremath{\mapsto}}
\newunicodechar{⨯}{\ensuremath{\times}}
\newunicodechar{∪}{\ensuremath{\cup}}
\newunicodechar{∩}{\ensuremath{\cap}}
\newunicodechar{⊋}{\ensuremath{\supsetneq}}
\newunicodechar{⊇}{\ensuremath{\supseteq}}
\newunicodechar{⊃}{\ensuremath{\supset}}
\newunicodechar{⊊}{\ensuremath{\subsetneq}}
\newunicodechar{⊆}{\ensuremath{\subseteq}}
\newunicodechar{⊂}{\ensuremath{\subset}}
\newunicodechar{≥}{\ensuremath{\geq}}
\newunicodechar{≠}{\ensuremath{\neq}}
\newunicodechar{≫}{\ensuremath{\gg}}
\newunicodechar{≪}{\ensuremath{\ll}}

\newunicodechar{≤}{\ensuremath{\leq}}
\newunicodechar{∈}{\ensuremath{\in}}
\newunicodechar{∉}{\ensuremath{\not \in}}
\newunicodechar{∖}{\ensuremath{\setminus}}
\newunicodechar{◦}{\ensuremath{\circ}}
\newunicodechar{°}{\ensuremath{^\circ}}
\newunicodechar{…}{\ifmmode\mathellipsis\else\textellipsis\fi}
\newunicodechar{·}{\ensuremath{\cdot}}
\newunicodechar{⋯}{\ensuremath{\cdots}}
\newunicodechar{∅}{\ensuremath{\emptyset}}
\newunicodechar{⇒}{\ensuremath{\Rightarrow}}

\newunicodechar{⁰}{\ensuremath{^0}}
\newunicodechar{¹}{\ensuremath{^1}}
\newunicodechar{²}{\ensuremath{^2}}
\newunicodechar{³}{\ensuremath{^3}}
\newunicodechar{⁴}{\ensuremath{^4}}
\newunicodechar{⁵}{\ensuremath{^5}}
\newunicodechar{⁶}{\ensuremath{^6}}
\newunicodechar{⁷}{\ensuremath{^7}}
\newunicodechar{⁸}{\ensuremath{^8}}
\newunicodechar{⁹}{\ensuremath{^9}}
\newunicodechar{ⁱ}{\ensuremath{^i}}

\newunicodechar{⌈}{\ensuremath{\lceil}}
\newunicodechar{⌉}{\ensuremath{\rceil}}
\newunicodechar{⌊}{\ensuremath{\lfloor}}
\newunicodechar{⌋}{\ensuremath{\rfloor}}

\newunicodechar{≅}{\ensuremath{\cong}}
\newunicodechar{⇔}{\ensuremath{\Leftrightarrow}}

\definecolor{linkred}{rgb}{0.7,0.2,0.2}
\definecolor{linkblue}{rgb}{0,0.2,0.6}

\usepackage[hyperpageref]{backref}

\DeclareFontFamily{OMS}{rsfs}{\skewchar\font'60}
\DeclareFontShape{OMS}{rsfs}{m}{n}{<-5>rsfs5 <5-7>rsfs7 <7->rsfs10 }{}
\DeclareSymbolFont{rsfs}{OMS}{rsfs}{m}{n}
\DeclareSymbolFontAlphabet{\scr}{rsfs}
\DeclareSymbolFontAlphabet{\scr}{rsfs}

\DeclareFontFamily{U}{mathx}{\hyphenchar\font45}
\DeclareFontShape{U}{mathx}{m}{n}{
	<5> <6> <7> <8> <9> <10>
	<10.95> <12> <14.4> <17.28> <20.74> <24.88>
	mathx10
}{}
 
  \theoremstyle{plain}
  \newtheorem{thmx}{Theorem}
  \renewcommand{\thethmx}{\Alph{thmx}} 
    \newtheorem{thm}{Theorem}[section] 
\newtheorem{proposition}[thm]{Proposition}
\newtheorem{lem}[thm]{Lemma}
\newtheorem{cor}[thm]{Corollary}

  \theoremstyle{definition}
   \newtheorem{dfn}[thm]{Definition}

   \theoremstyle{remark}
\newtheorem{rem}[thm]{Remark}

\numberwithin{equation}{subsection}  

\theoremstyle{plain}
\newlist{thmlist}{enumerate}{1}
\setlist[thmlist]{wide = 0pt, labelwidth = 2em, labelsep*=0em, itemindent = 0pt, leftmargin = \dimexpr\labelwidth + \labelsep\relax, noitemsep,topsep = 1ex, font=\normalfont, label=(\roman*), ref=\thethm.(\roman{thmlisti})}

\addtotheorempostheadhook[thm]{\crefalias{thmlisti}{thm}}

\addtotheorempostheadhook[proposition]{\crefalias{thmlisti}{proposition}}

\addtotheorempostheadhook[dfn]{\crefalias{thmlisti}{dfn}}

\addtotheorempostheadhook[lem]{\crefalias{thmlisti}{lem}}
\addtotheorempostheadhook[main]{\crefalias{thmlisti}{main}}

\newlist{thmenum}{enumerate}{1} 
\setlist[thmenum]{wide = 0pt, labelwidth = 2em, labelsep*=0em, itemindent = 0pt, leftmargin = \dimexpr\labelwidth + \labelsep\relax, noitemsep,topsep = 1ex, font=\normalfont, label=(\roman*), ref=\thethmx.(\roman{thmenumi})}
\crefalias{thmenumi}{thmx}

\crefname{lem}{Lemma}{Lemmas}
\crefname{thm}{Theorem}{Theorems}
\crefname{proposition}{Proposition}{Propositions}
\crefname{dfn}{Definition}{Definitions}
\crefname{rem}{Remark}{Remarks}
\crefname{cor}{Corollary}{Corollaries}
\crefname{corx}{Corollary}{Corollaries}
\crefname{problem}{Problem}{Problems}
\crefname{thmx}{Theorem}{Theorems}
\crefname{claim}{Claim}{Claims}

\crefname{main}{Main Theorem}{Main Theorems}

\def\rank{{\rm rank}}
\def\tsim{\!\!\text{\textasciitilde}}
 
\newcommand{\cS}{\mathcal{S}}
\newcommand{\cT}{\mathcal{T}}

\makeatletter
\newcommand*{\rom}[1]{\expandafter\@slowromancap\romannumeral #1@}
\makeatother

\newcommand{\lowerromannumeral}[1]{\romannumeral#1\relax}

\makeatletter     
\newcommand{\crefnames}[3]{%
	\@for\next:=#1\do{%
		\expandafter\crefname\expandafter{\next}{#2}{#3}%
	}%
}
\makeatother

\crefnames{part,chapter,section}{\S}{\S\S}

\def\oc{\mathscr{O}}  
  
 \def\fc{\mathcal{F}}

\def\lc{\mathcal{L}}
\def\vc{\mathcal{V}}

\def\cb{\mathbb{C}}

\def\vol{\operatorname{vol}}

\def\as{{a^\star}} \def\es{e^\star}

\def\cb{\mathbb{C}}

\def\as{\mathscr{A}}

\def\gs{\mathscr{G}}
\def\fs{\mathscr{F}}
\def\es{\mathscr{E}}
\def\js{\mathscr{J}}

\def\ls{\mathscr{L}}

\def\hs{\mathscr{H}}

\def\rs{\mathscr{R}}
\def\ts{\mathscr{T}}

\def\dbar{\bar{\partial}}

 \def\d{\partial}
\def\db{\mathbb{D}}

\makeatletter 
\let\@wraptoccontribs\wraptoccontribs
\makeatother

\makeatletter
\hypersetup{
	pdfauthor={\authors},
	pdftitle={\@title},
	pdfsubject={\@subjclass},
	pdfkeywords={\@keywords},
	pdfstartview={Fit},
	pdfpagelayout={TwoColumnRight},
	pdfpagemode={UseOutlines},
	bookmarks, 
colorlinks,
linkcolor=Fuchsia,
citecolor=NavyBlue,
urlcolor=CarnationPink}
\makeatother

\begin{document} 

\title[Analytic Shafarevich hyperbolicity conjecture]{On the hyperbolicity of base spaces   for   maximally \\ variational families of smooth projective varieties}

 \author{Ya Deng} 
 \address{Universit\'e de Strasbourg, Institut de Recherche Math\'ematique Avanc\'ee, 	7 Rue Ren\'e-Descartes,
 	67084 Strasbourg, France} 
 
 \curraddr{
 	Institut des Hautes Études Scientifiques,
 	Universit\'e  Paris-Saclay, 35 route de Chartres, 91440, Bures-sur-Yvette, France}
 \email{deng@ihes.fr}
 \urladdr{https://www.ihes.fr/~deng}

\contrib[With an appendix by]{Dan Abramovich}
\address{Department of Mathematics, Box 1917, Brown University, Providence, RI, 02912, 	U.S.A
}

 \email{abrmovic@math.brown.edu}
 \urladdr{http://www.math.brown.edu/~abrmovic/}

	\date{\today} 
	\begin{abstract}
For maximal variational smooth families of projective manifolds  whose general fibers have semi-ample canonical bundle, the  Viehweg hyperbolicity conjecture states that the base spaces of such families are of log general type. This deep conjecture was recently proved by  Campana-P\u{a}un and was later generalized by Popa-Schnell. In this paper we prove that those base spaces are pseudo Kobayashi hyperbolic, as predicted by the Lang conjecture: any complex quasi-projective manifold is pseudo Kobayashi hyperbolic if it is of log general type. As a consequence, we prove the Brody hyperbolicity of moduli spaces of polarized manifolds with semi-ample canonical bundle. This proves a conjecture by Viehweg-Zuo in 2003. We also  prove the Kobayashi hyperbolicity of base spaces for effectively parametrized families of minimal projective manifolds of general type. This generalizes previous work by To-Yeung, in which they further assumed that these families are canonically polarized.   
\end{abstract}
\maketitle
 
 \tableofcontents

\section{Introduction}\label{introduction}
\subsection{Main theorems}
A   complex space $X$ is \emph{Brody hyperbolic} if there is no non-constant holomorphic map $\gamma:\cb\to X$.   The first result in this paper is the affirmative answer to a conjecture by  by Viehweg-Zuo \cite[Question 0.2]{VZ03} on the Brody hyperbolicity of moduli spaces  for polarized manifolds with semi-ample canonical sheaf. 
\begin{thmx}[Brody hyperbolicity of moduli spaces]\label{VZ question}
	Consider the moduli functor $\mathscr{P}_h$ of polarized manifolds with semi-ample canonical sheaf  introduced by Viehweg \cite[\S 7.6]{Vie95}, where $h$ is the  Hilbert polynomial associated to the polarization $\hs$.  Assume that for some quasi-projective manifold $V$ there exists
	a smooth family $(f_U:U\to V,\hs)\in \mathscr{P}_h(V)$ for which the induced moduli map $\varphi_U:V\to P_h$ 
	is quasi-finite over its image, where $P_h$ denotes to be the  quasi-projective\footnote{The quasi-projectivity of $P_h$ was proved by Viehweg in \cite{Vie95}.} coarse moduli scheme for $\mathscr{P}_h$.   Then the base space $V$ is Brody hyperbolic.
\end{thmx} 

A complex space $X$ is called \emph{pseudo Kobayashi hyperbolic}, if $X$ is hyperbolic modulo a  proper Zariski closed subset $\Delta\subsetneq X$, that is, the Kobayashi pseudo distance $d_X:X\times X\to [0,+\infty[$ of
$X$ satisfies that $d_X(p,q)>0$ for every pair of distinct points $p,q\in X$ not both contained in $\Delta$. 
 In particular, $X$ is \emph{pseudo Brody hyperbolic}: any non-constant holomorphic map $\gamma:\cb\to X$ has image   $\gamma(\cb)\subset \Delta$. When such $\Delta$ is an empty set, this definition reduces to the usual definition of \emph{Kobayashi
 hyperbolicity}, and  the Kobayashi pseudo distance $d_{ X}$ is a distance.  

In this paper we indeed prove a stronger result than \cref{VZ question}. 
\begin{thmx}\label{Deng}
	Let $f_U:U\to V$ be a smooth projective morphism  between   complex quasi-projective manifolds  with connected fibers. Assume that the general fiber  of $f_U$  has semi-ample canonical bundle, and $f_U$ is of maximal variation, that is, the general fiber of $f_U$  can only be birational to  at most countably  many other 
	fibers.  Then the base space $V$ is pseudo Kobayashi hyperbolic.
\end{thmx}

As  a byproduct, we reduce  the pseudo Kobayashi hyperbolicity of varieties to the existence of certain negatively curved Higgs bundles (which we call   \emph{Viehweg-Zuo Higgs bundles} in \cref{def:VZ}). This provides a main building block for our recent work \cite{Den19} on the hyperbolicity of bases of log Calabi-Yau pairs.

Another aim of the paper  is to prove affirmatively a
folklore conjecture on the \emph{Kobayashi hyperbolicity} for moduli spaces of minimal projective manifolds of general type, which can be thought of as an analytic refinement of \cref{VZ question} in the case that  fibers have big and nef canonical bundle.  
\begin{thmx}\label{main}
	Let \( f_U:U\to V \) be a smooth projective  family of minimal projective manifolds of general type over a quasi-projective manifold $V$. Assume that   $f_U$ is \emph{effectively parametrized},
	that is, the Kodaira-Spencer map 
	\begin{align}\label{KS}
	\rho_y: \ts_{V,y}\to H^1(U_y, \ts_{U_y})
	\end{align}
	is injective for each point \(y\in V \),
	where  $\ts_{U_y}$ denotes the tangent bundle of the fiber  $U_y:=f_U^{-1}(y)$.  
	Then the base space \(V\) is Kobayashi hyperbolic.
\end{thmx}
\subsection{Previous related results}
\cref{Deng} is closely related to the \emph{Viehweg hyperbolicity conjecture}: let $f_U:U\to V$ be a maximally variational smooth projective family of  projective manifolds with semi-ample canonical bundle over a quasi-projective manifold $V$, then the base $V$ must be of log-general type.  In the series of   works \cite{VZ01,VZ02,VZ03},   Viehweg-Zuo    constructed  in a first step  a big subsheaf of   symmetric log differential forms of the base (so-called \emph{Viehweg-Zuo sheaves}). Built on this result,  Viehweg hyperbolicity conjecture was shown by Kebekus-Kov\'acs \cite{KK08a,KK08b,KK10} when $V$ is a surface or threefold,  by Patakfalvi \cite{Pat12} when $V$ is compact or admits a non-uniruled compactification, and it was    completely solved by Campana-P\u{a}un  \cite{CP15}, in which they  proved a vast generalization of 
the famous generic semipositivity result of Miyaoka (see also \cite{CP15b,CP16,Schn17} for  other different proofs). More recently, using deep theory of Hodge modules, Popa-Schnell \cite{PS17} constructed Viehweg-Zuo
sheaves on the base space $V$ of the smooth family $f_U:U\to V$ of projective manifolds whose geometric generic fiber admits a good minimal model. Combining this with the aforementioned theorem of Campana-P\u{a}un,  they proved that such
base space $V$ is of log general type. Therefore, \cref{Deng} is predicted by  a famous conjecture of Lang (cf. \cite[Chapter VIII. Conjecture 1.4]{Lan91}), which stipulates that a complex quasi-projective manifold
is pseudo Kobayashi hyperbolic if and only if it is of log general type. To our knowledge, Lang's conjecture is by now   known for the trivial case of curves,  for general hypersurface $X$ in the complex projective space $\mathbb{C}P^n$ of high degrees \cite{Bro17,Dem18,Siu15} as well as their complements $\cb P^n\setminus X$ \cite{BD19}, for projective manifolds whose universal cover carries a bounded strictly plurisubharmonic
function \cite{DB18}, for quotients of  bounded  (symmetric) domains   \cite{Rou16,CRT19,CDG19}, and for subvarieties  on abelian varieties  \cite{Yam18}.   \cref{Deng} therefore provides some new evidences for Lang's conjecture. 

\cref{VZ question} was first proved by  Viehweg-Zuo \cite[Theorem 0.1]{VZ03} for moduli spaces of \emph{canonically polarized} manifolds.    Combining the approaches by Viehweg-Zuo \cite{VZ03} with those by Popa-Schnell \cite{PS17}, very recently,    Popa-Taji-Wu \cite[Theorem 1.1]{PTW18}  proved \cref{VZ question}  for  moduli spaces of  \emph{polarized}   manifolds with big and semi-ample canonical bundles.  As we will see below, our work owes a lot to the general strategies and techniques in  their  work \cite{VZ03,PTW18}.    

The Kobayashi hyperbolicity of moduli spaces $\mathcal{M}_g$ of
compact Riemann surfaces of genus   $g\geqslant 2$ has long been known to us by  the work of Royden and Wolpert \cite{Roy74,Wol86}. The first important breakthrough on higher dimensional generalizations was made by  To-Yeung \cite{TY14}, in which they  proved  Kobayashi hyperbolicity of the base $V$ considered in \cref{main}  when  the canonical bundle  \(K_{U_y}\) of each fiber $U_y:=f^{-1}_U(y)$ of $f_U:U\to V$ is further assumed to be ample (see also  \cite{BPW17,Sch17} for alternative  proofs). 
Differently  from the approaches in \cite{VZ03,PTW18},  their strategy is to study the curvature of the generalized Weil-Petersson metric for families of canonically polarized manifolds,  along the approaches initiated by Siu \cite{Siu86} and later developed by Schumacher \cite{Sch12}.   For the smooth family  of Calabi-Yau manifolds (resp. orbifolds),
Berndtsson-P\u{a}un-Wang \cite{BPW17} and Schumacher  \cite{Sch17} (resp. To-Yeung \cite{TY18})  proved the Kobayashi hyperbolicity of the base once this family is assumed to be effectively parametrized. 

Recently, Lu, Sun, Zuo and the author \cite{DLSZ} proved a big Picard type theorem for moduli spaces of polarized manifolds with semi-ample canonical sheaf.  A crucial step of the  proof relies on the \enquote*{generic local Torelli-type theorem} in \cref{mainvz}. \cref{mainvz} also inspired us a lot in our more recent work \cite{Den20} on the big Picard theorem for varieties admitting variation of Hodge structures.



\subsection{Strategy of the proof}
For the smooth  family $f_U:U\to V$  of canonically polarized manifolds  with maximal variation,   Viehweg-Zuo  \cite{VZ03} 
constructed  certain negatively twisted Higgs bundles (which we call   \emph{Viehweg-Zuo Higgs bundles} in \cref{def:VZ})  \((\tilde{\es},\tilde{\theta}):=(\bigoplus_{q=0}^{n}\ls^{-1}\otimes E^{n-q,q},\bigoplus_{q=0}^{n} \vvmathbb{1}\otimes\theta_{n-q,q}) \),  over some smooth projective compactification $Y$ of a certain birational model $\tilde{V}$ of $V$, where \(\ls\) is some big and nef line bundle on \(Y\), and $\big(\bigoplus_{q=0}^{n}E^{n-q,q},\bigoplus_{q=0}^{n}\theta_{n-q,q}\big)$ is a Higgs bundle induced by  a polarized variation of Hodge structure defined over a Zariski open set of $\tilde{V}$. In a recent   paper \cite{PTW18}, Popa-Taji-Wu introduced several new inputs to develop    Viehweg-Zuo's strategy in \cite{VZ03}, which enables them to   construct    those Higgs bundles on base spaces of smooth families whose  geometric generic fiber admits a good minimal model  (see also \cref{thm:existence} for a weaker statement as well as a slightly different proof following the original construction by Viehweg-Zuo). As we will see in the main content, the Viehweg-Zuo Higgs bundles (VZ Higgs bundles for short) are the crucial tools in proving our main results. 

When each  fibers $U_y:=f_U^{-1}(y)$ of the smooth family $f_U:U\to V$ considered in \cref{Deng} have ample or big and nef canonical bundles,  let us briefly recall the  general  strategies in proving the \emph{pseudo Brody hyperbolicity} of $V$    in \cite{VZ03,PTW18}. A certain sub-Higgs bundle $(\fs,\eta)$ of \((\tilde{\es},\tilde{\theta}) \) with log poles contained in the divisor $D:=Y\setminus  \tilde{V}$ gives rise to a morphism 
\begin{eqnarray}\label{eq:iterate}
\tau_{\gamma,k}: \ts_\cb^{\otimes k}\to \gamma^*(\ls^{-1}\otimes E^{n-k,k})
\end{eqnarray}
for any entire curve \(\gamma:\cb\to \tilde{V} \).  If  \(\gamma:\cb\to \tilde{V} \) is Zariski dense, by    the Kodaira-Nakano vanishing  (when $K_{U_y}$ is ample) and  Bogomolov-Sommese vanishing theorems  (when $K_{U_y}$ is big and nef), 
one can verify that  \(\tau_{\gamma,1}(\cb)\not\equiv 0 \). Hence there is some $m>0$ (depending on $\gamma$) so that $\tau_{\gamma,m}$ factors through $\gamma^*(\ls^{-1}\otimes N^{n-m,m})$, where $N^{n-m,m}$ is the kernel of the Higgs field $\theta_m:E^{n-m,m}\to E^{n-m-1,m+1}\otimes \Omega_Y(\log D)$. Applying Zuo's theorem \cite{Zuo00} on the negativity of $N^{n-m,m}$,   a certain positively curved   metric  for $\ls$  can produce a singular hermitian metric on \(\ts_\cb \) with the \emph{Gaussian curvature} bounded from above by a negative constant, which contradicts with the (Demailly's) Ahlfors-Schwarz lemma \cite[Lemma 3.2]{Dem97}.   However,  this approach did not provide enough information for the Kobayashi pseudo distance of the base $V$. Moreover, the use of vanishing theorem cannot show   \(\tau_{\gamma,1}(\cb)\not\equiv 0 \) when fibers of $f_U:U\to V$ is not minimal   manifolds of general type.   

One of the main results in the present paper is to apply  the VZ Higgs bundle to construct a  (possibly degenerate) Finsler metric $F$  on  some birational model $\tilde{V}$ of the base $V$, whose   holomorphic sectional curvature  is bounded above by a negative constant (say   \emph{negatively curved Finsler metric} in \cref{negatively curved}).  A bimeromorphic criteria for pseudo Kobayashi hyperbolicity in \cref{pseudo Kobayashi} states that, 
the base is pseudo Kobayashi hyperbolic if $F$ is \emph{positively definite} over a Zariski dense open set.  Let us now briefly explain our idea of the constructions. By factorizing through  some sub-Higgs sheaf $ (\mathscr{F},\eta) \subseteq (\tilde{\es},\tilde{\theta})$ with logarithmic poles \emph{only} along the boundary divisor $D:=Y\setminus \tilde{V}$, one can define a morphism for any \(k=1,\ldots,n \):
\begin{eqnarray}\label{intro: iterated Kodaira}
\tau_k: {\rm Sym}^k \ts_Y(-\log D) \rightarrow \ls^{-1}\otimes E^{n-k,k},
\end{eqnarray}
where $\ls$ is some big  line bundle over $Y$ equipped with a \emph{positively curved} singular hermitian metric $h_\ls$.  Then for each \(k\), the  hermitian metric $h_k$  on \(\tilde{\es}_k:=\ls^{-1}\otimes  E^{n-k,k}\) induced by  the Hodge metric as well as $h_\ls$ (see \cref{singular metric} for details) will give rise to a Finsler metric \(F_k\) on \(\ts_Y(-\log D)\)  by taking the $k$-th root of the pull-back $\tau_k^*h_k$.  However,    the holomorphic sectional curvature of  \(F_k\)  might not be   negatively curved. 
Inspired by the aforementioned  work of  Schumacher, To-Yeung and Berndtsson-P\u{a}un-Wang \cite{Sch12,Sch17,TY14,BPW17} on the curvature computations of  generalized Weil-Petersson metric for families of canonically polarized manifolds,  we  define  a convex sum of Finsler metrics
\begin{eqnarray} \label{eq:Finsler convex}
F:=(\sum_{k=1}^{n}{\alpha_k}F^2_k)^{1/2}  \quad \mbox{with} \ \alpha_1, \dots,\alpha_n\in \mathbb{R}^+ 
\end{eqnarray}
 on \(\ts_Y(-\log D) \), to offset the unwanted positive terms in the curvature \(\Theta_{\tilde{\es}_k} \) by negative contributions from
the \(\Theta_{\tilde{\es}_{k+1} }\)  (the  last order term was \(\Theta_{\tilde{\es}_{n} }\) is always semi-negative by the Griffiths curvature formula).  We proved in   \cref{uniform}  that for proper \(\alpha_1, \dots,\alpha_n>0\), the holomorphic sectional curvature of \(F\) is negative and bounded away from zero. To summarize, we  establish an \emph{algorithm} for the construction of Finsler metrics via VZ Higgs bundles. 

To prove \cref{Deng}, we first note that the VZ Higgs bundles over some birational model $\tilde{V}$ of the base  space  $V$  were constructed by Popa-Taji-Wu in their elaborate work \cite{PTW18}. Let $Y$ be some  smooth projective compactification   $\tilde{V}$ with simple normal crossing boundary $D:=Y\setminus \tilde{V}$. By our construction of negatively curved Finsler metric $F$ defined in \eqref{eq:Finsler convex} via VZ Higgs bundles, to show that $F$ is \emph{positively definite} over some Zariski open set, it suffices to prove that $\tau_1:\ts_Y(-\log D)\to \ls^{-1}\otimes E^{n-1,1}$  defined in \eqref{intro: iterated Kodaira} is \emph{generically injective}  (which we call \emph{generic local Torelli}   for VZ Higgs bundles in \cref{sec:VZ}). 
This was proved in \cref{mainvz}, by using the degeneration of Hodge metric and the curvature properties of Hodge bundles.   In particular, we  show that the generic injectivity of $\tau_1$ is indeed an intrinsic feature  of all VZ Higgs bundles (not related to the Kodaira dimension of fibers of $f$!). By a standard inductive argument in \cite{VZ03,PTW18}, one can easily show that \cref{Deng} implies \cref{VZ question}.


Now we will explain the strategy to prove \cref{main}.  Note that the VZ Higgs bundles are only constructed over some birational model $\tilde{V}$ of $V$, which is not Kobayashi hyperbolic in general. 
This motivates us first to establish a \emph{bimeromorphic criteria for Kobayashi hyperbolicity} in \cref{bimeromorphic}.    
Based on this criteria, in order to apply the VZ Higgs bundles  
 to prove the Kobayashi hyperbolicity of the base $V$ in \cref{main}, 
 it suffices to show that 
\begin{enumerate}[label=($\spadesuit$),leftmargin=0.6cm]
	\item \label{several blow}for any given point $y$ on the base $V$, there exists a  VZ Higgs  bundle  $(\tilde{\es},\tilde{\theta})$ constructed over some birational model $\nu:\tilde{V}\to V$, such that \(\nu^{-1}:V\dashrightarrow \tilde{V}\) is defined at \(y\).
\end{enumerate}
\begin{enumerate}[label=($\clubsuit$),leftmargin=0.6cm]
	\item \label{condition:Finsler}The negatively curved Finsler metric $F$ on $\tilde{V}$ defined in \eqref{eq:Finsler convex} induced by the above VZ Higgs  bundle $(\tilde{\es},\tilde{\theta})$  is positively definite at the point $\nu^{-1}(y)$.
\end{enumerate}
Roughly speaking,  the idea is to produce  an abundant supply of \emph{fine} VZ Higgs bundles to construct sufficiently many negatively curved Finsler metrics, which are   obstructions to the degeneracy of Kobayashi pseudo distance $d_V$ of $V$. 
This is much more demanding than  the Brody hyperbolicity and Viewheg hyperbolicity of $V$, which can be shown by the existence of \emph{only one} VZ Higgs bundle on an arbitrary birational model of $V$, as mentioned in \cite{VZ02,VZ03,PS17,PTW18}.

Let us briefly explain how we achieve both \ref{several blow} and \ref{condition:Finsler}. 

As far as we see in \cite{VZ03,PTW18}, in their  construction of VZ Higgs bundles,  one has to  blow-up the base  for several times (indeed twice). Recall that the basic setup in  \cite{VZ03,PTW18}   is the  following: after passing to some smooth birational model ${f}_{\tilde{U}}:\tilde{U}=U\times_V{\tilde{V}}\to \tilde{V}$ of $f_U:U\to V$,  one can find a smooth projective compactification  $f:X\to Y$ of  $ \tilde{U}^r\to \tilde{V}$ 
\begin{align}\label{semistable}
\xymatrix{
	U^r \ar[d] &\tilde{U}^r \ar[l]_{\stackrel{{\rm bir}}{\sim}} \ar[d]  \ar[r]^\subseteq &X \ar[d]^{f}      \\
	V   &\tilde{V}\ar[l]^-{\nu}_{\stackrel{{\rm bir}}{\sim}} \ar[r]^\subseteq &Y}
\end{align}
so that there exists (at least) one hypersurface 
\begin{align}\label{eq: smooth hypersurface}
H \in \big\lvert \ell K_{X/Y}-\ell f^*\ls\big\lvert \quad {\rm for\ some} \quad \ell\gg 0 
\end{align}  
which is \emph{transverse} to the general fibers of $f$.  Here \(\ls \) is some big and nef line bundle over \(Y\), and $ {U}^r:= {U}\times_{ {V}}\times \cdots\times_{ {V}} {U}$ (resp. $\tilde{U}^r$)  is the $r$-fold fiber product of $f_{{U}}: {U}\to  {V}$ (resp. $f_{\tilde{U}}:\tilde{U}\to \tilde{V}$).    The VZ Higgs bundle is indeed the logarithmic Higgs  bundles associated
to the Hodge filtration of an auxiliary variation of polarized Hodge structures
constructed by taking the middle dimensional relative de Rham cohomlogy on
the cyclic cover of $X$ ramified along   $H$.

In order to find such $H$ in \eqref{eq: smooth hypersurface}, a crucial step in \cite{VZ03,PTW18} is the use of 
\emph{weakly semi-stable reduction} by Abramovich-Karu \cite{AK00} so that, after changing the birational model   $U\to V$ by performing certain (uncontrollable)  base change $\tilde{U}:=U\times_V\tilde{V}\to \tilde{V}$, one can   find a ``good" compactification $X\to Y$ of $\tilde{U}^r\to \tilde{V}$ and a  finite  dominant morphism $W\to Y$ from a smooth projective manifold $W$ such that the base change $X\times_YW\to W$ is birational to a \emph{mild morphism} $Z\to W$, which is in particular flat with reduced  fibers (even fonctorial under fiber products). 
For our goal \ref{several blow},  we need a more refined  control of the \emph{alteration for the base} in the weakly semistable reduction \cite[Theorem 0.3]{AK00}, which remains unknown at the moment.   {Fortunately, as was suggested to us and proved in \Cref{appendix} by  Abramovich, using moduli of  Alexeev stable maps one can establish a \emph{$\mathbb{Q}$-mild reduction} for the family $U\to V$  in place of the \emph{mild reduction} in \cite{VZ03}, so that we can also find a ``good" compactification $X\to Y$ of $U^r\to V$  without   passing  the birational models $ \tilde{V}\to V$ as in \eqref{semistable}. This is the main theme of \Cref{appendix}.}

Even if we can apply  $\mathbb{Q}$-mild reduction to avoid the first blow-up of the base as in \cite{VZ03,PTW18},   the second blow-up is in general inevitable.  Indeed, the \emph{discriminant} of  the new family $Z_H\to Y\supset V$ obtained  by taking the cyclic cover along $H$ in \eqref{eq: smooth hypersurface} is in general  not  normal crossing.   
One thus has to blow-up this  discriminant locus of $Z_H\to Y$  to make it normal crossing as  in \cite{PTW18}.  Therefore, to  assure \ref{several blow},  it then suffices to show that there exists a   compactification $ {f}: {X}\to  {Y}$ of the smooth family $U^r\to V$ so that for some sufficiently ample line bundle $\as$ over $Y$, 
\begin{equation}\label{eq:enough ample}
f_*(mK_{ {X}/ {Y}})\otimes \as^{-m} \mbox{ is globally generated over $V$ for some $m\gg 0$.}\tag{$\ast$}
\end{equation}
Indeed, for any given point $y\in V$, by \eqref{eq:enough ample} one can find $H$ transverse to the fiber $X_y:=f^{-1}(y)$, and thus  the new family $Z_H\to Y$ will be smooth over an open set containing $y$. 
To the bests of our knowledge,  \eqref{eq:enough ample}  was only known to us when the moduli is canonically polarized   \cite[Proposition 3.4]{VZ02}.   \cref{sec:positivity} is devoted to the proof of  \eqref{eq:enough ample} for the family $U\to V$ in \cref{main} (see \cref{desired global} below).  This  in turn achieves \ref{several blow}.     

To achieve \ref{condition:Finsler}, our  idea is to take \emph{different cyclic coverings}  by  \enquote*{moving}  $H$ in \eqref{eq: smooth hypersurface},     to produce different  \enquote*{fine} VZ Higgs bundles.     For  any given point \(y\in V\),  by  \ref{several blow},  one can take a birational model $\nu:\tilde{V}\to V$ so that $\nu$ is isomorphic at $y$, and there exists a VZ Higgs bundle $(\tilde{\es},\tilde{\theta})$ on  the normal crossing compactification $Y\supset \tilde{V}$.   
To prove that the induced negatively curved Finsler metric   $F$ is positively definite at $\tilde{y}:=\nu^{-1}(y)$, by our definition of   \(F\) in \eqref{eq:Finsler convex}, it suffices to show that 
\(\tau_1 \) defined in \eqref{intro: iterated Kodaira} is   \emph{injective} at \(\tilde{y}\) in the sense of $\cb$-linear map between complex vector spaces
\[
\tau_{1,\tilde{y}}:\ts_{\tilde{V},\tilde{y}}\xrightarrow{\simeq} \ts_{Y}(-\log D)_{\tilde{y}}\xrightarrow{\rho_{\tilde{y}}}H^1(X_{\tilde{y}}, \ts_{X_{\tilde{y}}}) \xrightarrow{\varphi_{\tilde{y}}} \tilde{\es}_{1,\tilde{y}}.
\]  
As we will see in \cref{construction}, when $H$ in \eqref{eq: smooth hypersurface} is properly chosen (indeed being transverse to the fiber $X_y$) which is ensured by \eqref{eq:enough ample},  $\varphi_{\tilde{y}}$ is injective at ${\tilde{y}}$. Hence $\tau_{1,\tilde{y}}$ is injective by our  assumption of \emph{effective parametrization}  (hence $\rho_{\tilde{y}}$ is injective)  in \cref{main}.  This is our strategy to prove \cref{main}.

 \subsection*{Acknowledgments}       
This paper is    the merger of my previous two articles   	 \cite{Den18,Den18b}  with slight improvements. It owes a lot to the celebrated work \cite{VZ02,VZ03,PTW18}, to which I  express my gratitudes. 
I would like to sincerely thank    Professors   Dan Abramovich, S\'ebastien Boucksom, H{\aa}kan Samuelsson  Kalm, Kalle Karu,  Mihai P\u{a}un, Mihnea Popa, Georg Schumacher, J\"org Winkelmann, Chenyang Xu, Xiaokui Yang, Kang Zuo,  and Olivier Benoists,  Junyan Cao, Chen Jiang,  Ruiran Sun,   Lei Wu, Jian Xiao for answering my questions  and very fruitful  discussions. I thank in particular Junyan Cao and Lei Wu for their careful reading of the early draft of the paper and numerous suggestions.  I am particularly grateful to Professor  Dan Abramovich  for suggesting the $\mathbb{Q}$-mild reduction,   and writing \Cref{appendix} which  provides a  crucial step for the present paper.  
  I thank   Professors   Damian Brotbek  and Jean-Pierre Demailly for their  encouragements and supports.   Lastly, I thank the referee for his careful reading of the paper and his suggestions on rewriting the paper completely.
\section*{Notations and conventions.} 
Throughout this article we will work over the complex number field $\cb$.
\begin{itemize}[leftmargin=0.4cm]
	\item An  algebraic fiber space\footnote{Here we follow the definition in \cite{Mor85}.}  (or fibration for short)  $f:X\to Y$ is a surjective projective morphism between   projective manifolds with connected geometric fibers. 
	Any $\mathbb{Q}$-divisor $E$ in $X$ is said to be $f$-exceptional if $f(E)$ is an algebraic variety of codimension at least two in $Y$. 
	\item We say that  a morphism $f_U:U\to V$ is a \emph{smooth family} if $f_U$ is a surjective smooth projective morphism with connected fibers between quasi-projective varieties.  
	\item For any surjective morphism $Y'\to Y$, and the algebraic fiber space $f:X\to Y$, we denote by $(X\times_YY')^{\tsim}$ the  (unique) irreducible component (say the \emph{main component}) of $X\times_Y Y'$ 	which dominates $Y'$. 
	\item  Let $\mu:X'\to X$ be a birational morphism from a   projective manifold $X'$ to a singular variety $X$.  $\mu$ is called a \emph{strong desingularization} if $\mu^{-1}(X^{\rm reg})\to X^{\rm reg}$ is an isomorphism. Here $X^{\rm reg}$ denotes to be the smooth locus of $X$.
	\item For any birational morphism $\mu:X'\to X$, the \emph{exceptional locus}  is the inverse image of the smallest closed set of $X$ outside of which $\mu$ is an isomorphism, and denoted by ${\rm Ex}(\mu)$.
	\item  	Denote by $X^r:=X\times_{Y} \cdots\times_{Y} X$   the $r$-fold fiber product of the fibration \(f:X\to Y\),  $(X^r)^{\tsim}$   the \emph{main component} of $X^r$ dominating $Y$,  and $X^{(r)}$  a \emph{strong desingularization}  of  $(X^r)^{\tsim}$.
	\item For any quasi-projective manifold $Y$,  a Zariski open subset $Y_0\subset Y$ is called a \emph{big open set} of $Y$ if and only if ${\rm codim}_{Y\setminus Y_0}(Y)\geqslant 2$.	
	\item A singular hermitian metric $h$ on the line bundle  $L$ is said to be \emph{positively curved} if the curvature current $\Theta_h(L)\geqslant 0$.
\end{itemize}
\section{Brody hyperbolicity of the base}\label{Higgs}
To begin with, let us introduce the definition of \emph{Viehweg-Zuo Higgs bundles} over  quasi-projective manifolds in an abstract way  following \cite{VZ03,PTW18}.  Then we prove a generic local Torelli   for VZ Higgs bundles. We will show that based on the previous work by Viehweg-Zuo and Popa-Taji-Wu, this generic local Torelli theorem suffices to prove \cref{VZ question}.  

\subsection{Abstract Viehweg-Zuo Higgs bundles}\label{sec:VZ}
The definition we present below follows from the formulation in \cite{VZ02,VZ03} and \cite[Proposition 2.7]{PTW18}. 
\begin{dfn}[Abstract Viehweg-Zuo Higgs bundles]\label{def:VZ}
	Let $V$ be a quasi-projective manifold, and let $Y\supset V$ be a projective compactification of $V$ with the boundary $D:=Y\setminus V$ simple normal crossing. A \emph{Viehweg-Zuo Higgs bundle on $V$} is a logarithmic  Higgs bundle $(\tilde{\es},\tilde{\theta})$  over $Y$  consisting of the following data:
	\begin{thmlist}
		\item  a divisor $S$ on $Y$ so that $D+S$ is simple normal crossing,
		\item \label{VZ big}  a big and nef line bundle $\ls$ over $Y$ with $\mathbf{B}_+(\ls)\subset D\cup S $, 
		\item   a Higgs bundle  $( {\es}, {\theta}):=\big(\bigoplus_{q=0}^{n}E^{n-q,q},\bigoplus_{q=0}^{n}\theta_{n-q,q}\big)$ induced by the lower canonical extension of a polarized VHS defined over $Y\setminus (D\cup S)$, 
		\item a sub-Higgs sheaf $(\fs,\eta)\subset (\tilde{\es},\tilde{\theta})$,
	\end{thmlist}
	which	satisfy the following properties.
	\begin{enumerate}[leftmargin=0.5cm]
		\item The Higgs bundle $(\tilde{\es},\tilde{\theta}):=(\ls^{-1}\otimes {\es},\vvmathbb{1}\otimes {\theta})$. In  particular, $\tilde{\theta}:\tilde{\es}\to \tilde{\es}\otimes \Omega_{Y}\big(\log (D+S)\big)$, and $\tilde{\theta}\wedge\tilde{\theta}=0$.
		\item The sub-Higgs sheaf $(\fs,\eta)$ has log poles only on the boundary $D$, that is, $\eta:\fs\to \fs\otimes\Omega_{Y}(\log D)$.
		\item \label{contain trivial}Write $\tilde{\es}_k:=\ls^{-1}\otimes E^{n-k,k}$, and denote by $\fs_k:=\tilde{\es}_k\cap \fs$. Then the first stage $\fs_0$ of $\fs$ is an \emph{effective line bundle}. In other words, there exists a non-trivial morphism $\oc_Y\to \fs_0$.
	\end{enumerate}
\end{dfn}
As shown in \cite{VZ02}, by iterating $\eta$ for $k$-times, we obtain
$$
\fs_0\xrightarrow{\overbrace{\eta\circ\cdots\circ \eta}^{k\, \text{times}}} \fs_k\otimes \big(\Omega_Y(\log D)\big)^{\otimes k}.
$$
Since $\eta\wedge\eta =0$, the above morphism factors through $ \fs_k\otimes {\rm Sym}^k\Omega_Y(\log D)$, and by \eqref{contain trivial} one thus obtains 
$$
\oc_Y\to \fs_0\to \fs_k\otimes {\rm Sym}^k\Omega_Y(\log D)\to \ls^{-1}\otimes E^{n-k,k}\otimes {\rm Sym}^k\Omega_Y(\log D).
$$
Equivalently, we have a morphism
\begin{align}\label{iterated Kodaira2}
\tau_k:  {\rm Sym}^k \ts_Y(-\log D)\rightarrow \ls^{-1}\otimes E^{n-k,k}.
\end{align}
It was proven in \cite[Corollary 4.5]{VZ02} that $\tau_1$ is always non-trivial.
We say that a VZ Higgs bundle  satisfies the \emph{generic local Torelli} if $\tau_1:\ts_Y(-\log D)\to \ls^{-1}\otimes E^{n-1,1}$ in \eqref{iterated Kodaira2}  is  generically injective. As we will see in \cref{constructionvz}, in \cref{mainvz} we prove that  the generic  local  Torelli  holds for any VZ Higgs bundles. 

\subsection{A quick tour on Viehweg-Zuo's construction}\label{sec:existence}
For   the smooth family $U\to V$ in \cref{Deng,VZ question}, it was shown in  \cite{VZ02} and \cite[Proposition 2.7]{PTW18} that  there is a VZ Higgs bundle over some  birational model $\tilde{V}$ of $V$. Indeed, using the deep theory of mixed Hodge modules, Popa-Taji-Wu \cite{PTW18} can even construct VZ Higgs bundles over the bases of maximal variational smooth families whose geometric generic fiber admits a good minimal model. Since we need to study the precise loci where $\tau_1$ is injective in the proof of \cref{main}, in this subsection we  recollect Viehweg-Zuo's construction on VZ Higgs bundles over the base space $V$ (up to a birational model and a projective compactification) in \cref{Deng}. We refer the readers to see \cite{VZ02} and \cite{PTW18} for more details. In \cref{construction}, we show how to refine this construction   to prove \cref{main}.  Let us mention that we do not clarify any originality for this subsection. 
\begin{thm}\label{thm:existence}
	Let  $U\to V$ be the smooth family in \cref{Deng}.  Then after replacing $V$ by a birational model $\tilde{V}$,  there is a smooth compactification $Y\supset \tilde{V}$ and a VZ Higgs bundle over $\tilde{V}$.
\end{thm}
\begin{proof}
	By \cite{VZ03,PTW18}, one can take a birational morphism $\nu:\tilde{V}\to V$ and a smooth compactification $f:X\to Y$ of $U^r\times_V\tilde{V}\to \tilde{V}$ 
	so that there exists a hypersurface
	\begin{align}\label{eq:cyclic}
	H\in |\ell\Omega_{X/Y}^n(\log \Delta)-\ell f^*\ls+\ell E|,\quad n:=\dim X-\dim Y
	\end{align} 
	with $\ls$ a big and nef line bundle over $Y$ 
	satisfying  that
	\begin{enumerate}[leftmargin=0.6cm] 
		\item the complement $D:=Y\setminus \tilde{V}$ is simple normal crossing.
		\item \label{enu:smooth}The hypersurface $H$ is smooth over some Zariski open set $V_0\subset \tilde{V}$ with $D+S:=Y\setminus V_0$ simple normal crossing.
		\item The divisor $E$ is  effective and $f$-exceptional divisor with $f(E)\cap V_0=\varnothing$. 
		\item The augmented base locus $\mathbf{B}_+(\ls)\cap V_0=\varnothing$.
	\end{enumerate} 
	Here we denote by $\Delta:=f^{-1}(D)$ so that $(X,\Delta)\to (Y,D)$ is a log morphism.   Within this basic setup, let us first recall      two Higgs bundles in the theorem following \cite[\S 4]{VZ02}. 
	Leaving out a codimension two subvariety of
	${Y}$ supported on $D+S$, we  assume that 
	\begin{itemize}[leftmargin=0.6cm] 	
		\item  the morphism $f$ is flat, and $E$ in \eqref{eq:cyclic}  disappears. 
		\item The divisor $D+S$ is smooth. Moreover,  both $\Delta $ and $\Sigma=f^{-1}S$ are relative normal crossing. 
	\end{itemize}  
	Set $\lc:=\Omega_{X/Y}^n(\log \Delta)$.  
	Let $\delta:W \to X$  be a blow-up of $X$ with centers in $\Delta+\Sigma$ such that  $\delta^*(H+\Delta+\Sigma)$ is a normal crossing divisor. One thus obtains a
	cyclic covering of $\delta^*H$,  by taking the $\ell$-th root out of
	$\delta^*H$. Let $Z$ 
	to be
	a strong desingularization of this covering, which is smooth over $V_0$ by \eqref{enu:smooth}.     We denote the compositions by $h:W\to Y$ and $g:Z\to Y$, whose restrictions   to
	$V_0$ are both smooth.  Write $\Pi:=g^{-1}(S\cup D)$ which can be assumed to be normal crossing.    Leaving out   codimension two subvariety supported $D+S$ further, we assume that $h$ and $g$ are also flat, and both $\delta^*(H+\Delta+\Sigma)$ and   $\Pi$ are relative normal crossing.     
	Set
	$$
	F^{n-q,q}:= R^qh_*\Big(\delta^*\big(\Omega_{X/Y}^{n-q}(\log \Delta)\big)\otimes \delta^*\lc^{-1}\otimes \oc_W\big(\lfloor \frac{\delta^*H}{\ell} \rfloor\big)\Big) / {\rm torsion}.
	$$ 
	It was shown in \cite[\S 4]{VZ02} that there exists a natural  edge
	morphism
	\begin{align}\label{eq:edge}
	\tau_{n-q,q}: F^{n-q,q}\to F^{n-q-1,q+1}\otimes \Omega_Y(\log D),
	\end{align}
	which gives rise to  the first Higgs bundle   $\big(\bigoplus_{q=0}^{n}F^{n-q,q},\bigoplus_{q=0}^{n}\tau_{n-q,q}\big)$ defined over a big open set of $Y$ containing $V_0$.
	
	%
	Write $Z_0:=Z\setminus \Pi$.  
	Then the local system
	$ R^n g_*\cb_{\upharpoonright Z_0}$  
	extends to a locally
	free sheaf $\vc$ 
	on 
	$Y$ (here $Y$ is projective rather than the big open set!) equipped with the logarithmic connection
	$$
	\nabla:\vc\to \vc\otimes \Omega_Y\big(\log (D+S)\big),  
	$$ 
	whose  eigenvalues of the residues   lie  in $[0,1)$ (the so-called \emph{lower canonical extension}). 
	By Schmid's \emph{nilpotent orbit theorem} \cite{Sch73}, the Hodge filtration of $ R^n g_*\cb_{\upharpoonright Z_0}$   extends to a filtration $\vc:=\fc^0\supset \fc^1\supset \cdots\supset \fc^{n}$ of \emph{subbundles} so that their graded sheaves $E^{n-q,q}:=\fc^{n-q}/\fc^{n-q+1}$ 
	are also  {locally free}, and  there exists
	$$
	\theta_{n-q,q}:E^{n-q,q}\to E^{n-q-1,q+1}\otimes \Omega_{Y}(\log D+S). 
	$$  
	This defines the second Higgs bundle 
	$
	\big(\bigoplus_{q=0}^{n}E^{n-q,q},\theta_{n-q,q}\big)
	$. 
	As observed in \cite{VZ02,VZ03},   $E^{n-q,q}=R^qg_*\Omega^{n-q}_{Z/Y}(\log \Pi)
	$ over a big open set of $Y$ by the   theorem of Steenbrink \cite{Ste77,Zuc84}. By the construction of the cyclic cover $Z$,  this in turn implies  the following  commutative diagram over a  big open set of $Y$:
	\begin{align}\label{dia:two Higgs relation}
	\xymatrixcolsep{4.3pc}\xymatrix{
		\ls^{-1}\otimes	E^{n-q,q}  \ar[r]^-{\vvmathbb{1}\otimes \theta_{n-q,q}}   &		\ls^{-1}\otimes E^{n-q-1,q+1}\otimes \Omega_Y\big(\log (D+S)\big)  \\
		F^{n-q,q}  \ar[u]^{\rho_{n-q,q}}  \ar[r]^-{  \tau_{n-q,q}}      &     F^{n-q-1,q+1}\otimes \Omega_Y(\log D)   \ar[u]_{\rho_{n-q-1,q+1}\otimes \iota}}
	\end{align} 
	as shown in  \cite[Lemma 6.2]{VZ03} (cf. also \cite[Lemma 4.4]{VZ02}).   
	
	Note that all the objects are defined on a big open set of  $Y$ except for $
	\big(\bigoplus_{q=0}^{n}E^{n-q,q},\theta_{n-q,q}\big)
	$, which are defined on the whole $Y$.  Following \cite[\S 6]{VZ03}, for every $q=0,\ldots, n$, we define $F^{n-q,q}$ to be 
	the reflexive hull, and  the morphisms  
	$\tau_{n-q,q}$ and  $\rho_{n-q,q}$ extend  naturally.  

	To conclude that $\big(\bigoplus_{q=0}^{n}\ls^{-1}\otimes E^{n-q,q},\bigoplus_{q=0}^{n}\vvmathbb{1}\otimes  {\theta}_{n-q,q}\big)$ is a VZ Higgs bundle as in \cref{def:VZ},  we have to introduce a sub-Higgs sheaf with log poles supported on $D$. Write $\tilde{\theta}_{n-q,q}:=\vvmathbb{1}\otimes\theta_{n-q,q}$ for short. Following \cite[Corollary 4.5]{VZ02} (cf. also \cite{PTW18}), for each $q=0,\ldots,n$, we define a coherent torsion-free sheaf $\fs_q:=\rho_{n-q,q}(F^{n-q,q})\subset E^{n-q,q}$ . 
	By  $F^{n,0}\supset \oc_Y$,  $\fs_0\supset \oc_Y$.  By \eqref{eq:edge} and \eqref{dia:two Higgs relation}, one has
	$$
	\tilde{\theta}_{n-q,q}:\fs_q\to \fs_{q+1}\otimes \Omega_{Y}(\log D),
	$$
	and let us by $\eta_{q}$ the restriction of  $\tilde{\theta}_{n-q,q}$ to $\fs_q$.  Then $(\fs,\eta):=\big(\bigoplus_{q=0}^{n}\fs_q,\bigoplus_{q=0}^{n}\eta_{q}\big)$ is a sub-Higgs bundle of $(\tilde{\es},\tilde{\theta}):=\big(\bigoplus_{q=0}^{n}\ls^{-1}\otimes E^{n-q,q},\bigoplus_{q=0}^{n}\tilde{\theta}_{n-q,q}\big)$. 
\end{proof}

\subsection{Proper metrics for logarithmic Higgs bundles} 
We adopt the same notations as \cref{def:VZ} in the rest of \cref{Higgs}. As is well-known, $\es$  can be endowed with the Hodge metric \(h\)  induced by the polarization, which may blow-up around the simple normal crossing boundary \(D+S\).  However,  according to the work of Schmid,  Cattani-Schmid-Kaplan and Kashiwara \cite{Sch73,CKS86,Kas85}, $h$ has  \emph{mild singularities} (at most logarithmic singularities), and as proved in  \cite[\S 7]{VZ03} (for unipotent monodromies) and \cite[\S 3]{PTW18} (for quasi-unipotent monodromies), one  can  take a proper  singular metric \(g_\alpha\) on \(\ls\) such that the induced singular hermitian metric \(g_\alpha^{-1}\otimes h \) on \(\tilde{\es}:=\ls^{-1}\otimes \es\) is locally bounded from above. Before we summarize the above-mentioned results in \cite[\S 3]{PTW18}, we introduce some notations in \emph{loc. cit.}

Write the simple normal crossing divisor \(D=D_1+\cdots+D_k \) and \(S=S_1+\cdots+S_\ell \). Let \(f_{D_i}\in H^0\big(Y,\oc_Y(D_i)\big)\) and \(f_{S_i}\in H^0\big(Y,\oc_Y(S_i)\big)\) be the canonical section defining \(D_i\) and \(S_i\). We fix  smooth hermitian metrics  \(g_{D_i}\) and \(g_{S_i}\) on \(\oc_Y(D_i)\) and \(\oc_Y(S_i)\). Set
\[
r_{D_i}:=- \log \lVert f_{D_i}\rVert^2_{g_{D_i}}, \quad r_{S_i}:=- \log \lVert f_{S_i}\rVert^2_{g_{S_i}},
\]
and define
\[
r_D:=\prod_{i=1}^{k}r_{D_i}, \quad r_S:=\prod_{i=1}^{\ell}r_{S_i}.
\]
Let $g$ be a singular hermitian metric with analytic singularities of the  big and nef  line bundle \(\ls\) such that \(g\) is smooth on $Y\setminus \mathbf{B}_+(\ls)\supset Y\setminus D\cup S$,  and the curvature current 
$
\sqrt{-1}\Theta_{g}(\ls)\geqslant \omega
$
for some smooth K\"ahler form \(\omega \) on $Y$. For \(\alpha\in \mathbb{N}\),  define
\[
g_\alpha:=g\cdot (r_D\cdot r_S)^\alpha
\]
The following proposition is a slight variant of \cite[Lemma 3.1, Corollary 3.4]{PTW18}.\noindent
\begin{proposition}[\!\protect{\cite{PTW18}}]\label{singular metric}
	When \(\alpha\gg 0 \), after rescaling \(f_{D_i} \) and \(f_{S_i}\), there exists a continuous, positively definite hermitian form \(\omega_\alpha\) on \(\ts_{Y}(-\log {D}) \) such that
	\begin{thmlist} 
		\item\label{estimate}  over $V_0:=Y\setminus D\cup S$,  the curvature form
		\[
		\sqrt{-1}\Theta_{g_\alpha}(\ls)_{\upharpoonright V_0} \geqslant r_D^{-2}\cdot \omega_{\alpha\upharpoonright V_0}.
		\]
		\item \label{bounded} The singular hermitian metric \(h_g^\alpha:=g_\alpha^{-1}\otimes h \) on \(\ls^{-1}\otimes \es\) is locally bounded  on \(Y\), and smooth outside \((D+S)  \). Moreover, \(h_g^\alpha\) is degenerate  on \(D+S \). 
		\item \label{new bound} The singular hermitian metric  \(r_D^2h_g^\alpha \) on  \(\ls^{-1}\otimes \es\) is also locally bounded  on \(Y\).  \qed
	\end{thmlist}
\end{proposition} 
\begin{rem}
	It follows from \cref{singular metric} that both \(h_g^\alpha\) and \(r_D^2h_g^\alpha \) can be seen as   Finsler metrics  on \(\ls^{-1}\otimes \es\) which are degenerate  on \({\rm Supp}(D+S) \), and positively definite on \(V_0\).
\end{rem}
Although the last statement of \cref{bounded} is not explicitly stated in \cite{PTW18}, it can be easily seen from the proof of \cite[Corollary 3.4]{PTW18}. \cref{singular metric} mainly relies on the  asymptotic  behavior of the Hodge metric for \emph{lower canonical extension} of a variation of Hodge structure (cf. \cref{Hodge metric} below) when the monodromy around the boundaries are only quasi-unipotent. 
\begin{thm}[\!\protect{\cite[Lemma 3.2]{PTW18}}]\label{Hodge metric}
	Let $\mathcal{H}=\fc^0\supset \fc^1\supset \cdots\supset \fc^N\supset 0$  be a variation of Hodge structures  defined over  $(\Delta^*)^p\times \Delta^q$, where $\Delta$ (resp. $\Delta^*$) is the (resp. punctured) unit disk. Consider the lower canonical extension $^l\fc^\bullet$ over $\Delta^{p+q}\supset (\Delta^*)^p\times \Delta^q$, and denote by $(\es,\theta)$   the associated Higgs bundle. Then for any holomorphic section $s\in \Gamma(U,\es)$, where $U\subsetneq  \Delta^{p+q}$ is a relatively compact open set containing the origin, one has the following norm estimate 
	\begin{align}\label{eq:norm}
	|s|_{\rm hod}\leqslant C \big((-\log |t_1|)\cdot (-\log |t_2|)\cdots (-\log |t_p|)\big)^\alpha,
	\end{align}
	where $\alpha$ is some positive constant independent of  $s$, and $t=(t_1,\ldots,t_{p+q})$ denotes to be  the coordinates of $\Delta^{p+q}$.
\end{thm}
Let us mention that the estimates of  Hodge metric for   \emph{upper canonical extension} were obtained by Peters \cite{Pet84} in one variable, and by Catanese-Kawamata \cite{CK17} in several variables, based on the work   \cite{Sch73,CKS86}. We provide a slightly different proof of \cref{Hodge metric}  for completeness sake, following  closely the  approaches in \cite{Pet84,CK17}. 
\begin{proof}[Proof of \cref{Hodge metric}]
	The fundamental group $\pi_1\big((\Delta^*)^p\times \Delta^q\big)$ is generated by elements
	$\gamma_1,\ldots,\gamma_p$, where $\gamma_j$ may be identified with the counter-clockwise generator of the
	fundamental group of the $j$-th copy of $\Delta^*$
	in $(\Delta^*)^p$. Set $T_j$ to be the 
	monodromy transformation with respect to $\gamma_j$, which  pairwise commute and are known to be
	quasi-unipotent; that is, for any multivalued section $\underline{v}(t_1,\ldots,t_{p+q})$ of $\mathcal{H}$, one has
	$$
	\underline{v}(t_1,\ldots,e^{2\pi i}t_j,\ldots,t_{p+q})=T_j\cdot \underline{v}(t_1,\ldots,t_{p+q})
	$$
	and $[T_j,T_k]=0$ for any $j,k=1,\ldots,p$. Set 
	$
	T_j=D_j\cdot U_j
	$ 
	to be the (unique) Jordan-Chevally decomposition, so that $D_j$ diagonalizable and $U_j$ is unipotent with $[D_j,U_j]=0$.   Since $T_j$ is quasi-unipotent by the theorem of Borel,   all the eigenvalues of $D_j$ are thus the roots of unity. 
	Set  $N_j:=\frac{1}{2\pi i}\sum_{k>0}(I-U_j)^k/k$. If $D_j={\rm diag.}(d_{j\ell})$ then we set $S_j={\rm diag.}(\lambda_{j\ell})$ with $\lambda_{j\ell}\in (-2\pi i,0]$ and $\exp(\lambda_{j\ell})=d_{j\ell}$. Since $[T_j,T_k]=0$,   Jordan-Chevally decomposition implies that
	\begin{align}\label{commutative}
	[S_j,S_k]=[S_j,N_k]=[N_j,N_k]=0.
	\end{align}
	Fix a point $t_0\in (\Delta^*)^p\times \Delta^q$, and take a basis $v_1,\ldots,v_r\in V_{t_0}$ so that $S_1,\ldots,S_p$ are simultaneously diagonal, that is, one has
	\begin{align}\label{eigenvalue}
	S_j(v_\ell)=\lambda_{j\ell}.
	\end{align}
	Let us define $\underline{v}_1(t),\ldots,\underline{v}_r(t)$ to be the induced multivalued flat sections. Then 
	$$
	e_j(t):=\exp\big(-\frac{1}{2\pi i}\sum_{i=1}^{p}(S_i+N_i)\cdot \log t_i  \big)\underline{v}_j(t)
	$$
	is single-valued and  forms a basis of holomorphic sections for the lower canonical extension $^l \mathcal{H}$.
	
	Recall that $d_{j\ell}$ are all roots of unity. One thus can take  a positive integer $m$ so that $m_{j\ell}:=-{m\lambda_{j\ell}}/{2\pi i}$ are all \emph{non-negative integers}. Equivalently, each $T_j^m$ is unipotent. Define a ramified cover 
	\begin{align*}
	\pi:   \Delta^{p+q}&\to  \Delta^{p+q}\\
	(w_1,\ldots,w_{p+q}) &\mapsto   (w_1^m,\ldots,w^m_{p},w_{p+1},\ldots,w_{p+q})
	\end{align*}
	and set $\pi'$ to be the restriction of $\pi$ to $(\Delta^*)^p\times \Delta^q$. Then $\pi'^*\fc^{\bullet}$ is a variation of Hodge structure on $(\Delta^*)^p\times \Delta^q$ with unipotent monodromy, and we define $^c{\pi'^*\mathcal{H}}$ the canonical extension of $\pi'^*\mathcal{H}$. Set $\underline{u}_j(w)=\pi'^*\underline{v}_j$ which are multivalued sections for the local system $\pi'^*\mathcal{H}$. Then
	$$
	\underline{u}_j(w_1,\ldots,e^{2\pi i}w_j,\ldots,w_{p+q})=T_j^m\cdot \underline{u}_j(w_1,\ldots, w_{p+q}).
	$$
	Define 
	\begin{align}\label{basis}
	\tilde{e}_j(w):=\exp\big(-\frac{1}{2\pi i}\sum_{i=1}^{p}mN_i\cdot \log w_i  \big)\underline{u}_j(w)
	\end{align}
	which forms  a basis  of   $^c\pi'^*\mathcal{H}$. Based on the work of \cite{Sch73,CKS86}, it was shown in \cite[Claim 7.8]{VZ03} that one has the upper bound  of norms
	\begin{align}\label{norm}
	|\tilde{e}_j(w)|_{\rm hod}\leqslant C_0 \big((-\log |w_1|)\cdot (-\log |w_2|)\cdots (-\log |w_p|)\big)^\alpha
	\end{align}
	for some positive constants $C_0$ and $\alpha$.
	One the other hand, we have
	\begin{align*}
	\pi'^*e_j(w)&=\exp\big(-\frac{1}{2\pi i}\sum_{i=1}^{p}(S_i+N_i)\cdot \log w^m_i  \big)\pi'^*\underline{v}_j (w) \\
	&\stackrel{\eqref{commutative}}{=}\exp\big(-\frac{1}{2\pi i}\sum_{i=1}^{p}mN_i\cdot \log w_i  \big)\cdot \exp\big(-\frac{1}{2\pi i}\sum_{i=1}^{p}mS_i\log w_i \big) \pi'^*\underline{v}_j (w)\\
	&\stackrel{\eqref{eigenvalue}}{=}\exp\big(-\frac{1}{2\pi i}\sum_{i=1}^{p}mN_i\cdot \log w_i  \big)\cdot \exp\big(-\frac{1}{2\pi i}\sum_{i=1}^{p}m\lambda_{ij}\log w_i \big) \pi'^*\underline{v}_j (w)\\
	&=\prod_{i=1}^{p}w_i^{m_{ij}}\cdot \exp\big(-\frac{1}{2\pi i}\sum_{i=1}^{p}mN_i\cdot \log w_i  \big)\cdot  \underline{u}_j (w)\\
	&\stackrel{\eqref{basis}}{=}\prod_{i=1}^{p}w_i^{m_{ij}}\cdot \tilde{e}_j(w).
	\end{align*}
	By the definition of lower canonical extension, $m_{ij}$ are all non-negative integers, and thus  
	\begin{align*}
	\pi'^*|e_j|_{\rm hod}(w)&=|\pi'^*e_j(w)|_{\rm hod}= \prod_{i=1}^{p}|w_i|^{m_{ij}}|\tilde{e}_j(w)|_{\rm hod}\\
	&\stackrel{\eqref{norm}}{\leqslant}  C_0 \big((-\log |w_1|)\cdot (-\log |w_2|)\cdots (-\log |w_p|)\big)^\alpha.
	\end{align*}
	Hence
	$$
	|e_j|_{\rm hod}(t)\leqslant   \frac{C_0}{m^p} \big((-\log |t_1|)\cdot (-\log |t_2|)\cdots (-\log |t_p|)\big)^\alpha.
	$$
	Note that $^l\mathcal{H}\stackrel{\mathscr{C}^\infty}{\simeq}\es$. Therefore, for any holomorphic section $s\in \Gamma(U,\es)$, there exist smooth functions $f_1,\ldots,f_r\in \oc(U)$ so that $s=\sum_{j=1}^{r}f_je_j$. This shows the estimate \eqref{eq:norm}.
\end{proof}
\begin{rem}
	For the Hodge metric of upper canonical extension, one makes the choice that $\lambda_{j\ell}\in [0,2\pi i)$ instead of $\lambda_{j\ell}\in (-2\pi i,0]$ in the proof of \cref{Hodge metric}. Then the same computation as above can easily show that
	$$
	|e_j|_{\rm hod}(t)\leqslant \prod_{i=1}^{p}|t_i|^{-\frac{\lambda_{ij}}{2\pi i}} \frac{C}{m^p} \big((-\log |t_1|)\cdot (-\log |t_2|)\cdots (-\log |t_p|)\big)^\alpha,
	$$
	which were obtained in \cite{CK17}.
\end{rem}

\subsection{A generic local Torelli for VZ Higgs bundle}\label{constructionvz}
In this section we  prove that the generic  local  Torelli  holds for any VZ Higgs bundle, which is a crucial step in the proofs of \cref{Deng,VZ question}.

\begin{thmx}[Generic  local  Torelli]\label{mainvz}
	For the abstract Viehweg-Zuo Higgs bundles defined in \cref{def:VZ}, the morphism $\tau_1:\ts_Y(-\log D)\to \ls^{-1}\otimes E^{n-1,1}$ defined in \eqref{iterated Kodaira2} is generically injective. 
\end{thmx} 

\begin{proof}[Proof of \cref{mainvz}]
	By \cref{def:VZ}, the non-zero morphism $\oc_Y\to \fs_0\to \ls^{-1}\otimes E^{n,0}$   induces a global section  $s\in H^0(Y, \ls^{-1}\otimes E^{n,0})$, 
	which is \emph{generically} non-vanishing over $V_0:=Y\setminus D\cup S$. Set 
	\begin{align}\label{set}
	V_1:=\{y\in V_0 \mid s(y)\neq 0 \}
	\end{align}
	which is a non-empty Zariski open set of $V_0$.  For the first stage of VZ Higgs bundle $\ls^{-1}\otimes E^{n,0}$, we equip it with a singular metric $h_g^\alpha:=g_\alpha^{-1}\otimes h$ as in \cref{singular metric}, so that \cref{estimate,bounded} are satisfied.  Note that $h_g^\alpha$ is smooth over $V_0$. Let us denote  $D'$ to be the $(1,0)$-part of its Chern connection over $V_0$, and $\Theta_0$ to be its curvature form. Then by the Griffiths curvature formula of Hodge bundles (see   \cite{GT84}), 	  over $V_0$ we   have
	\begin{align}\nonumber
	\Theta_0 &=- \Theta_{\ls,g_{\alpha}}\otimes \vvmathbb{1}+\vvmathbb{1}\otimes \Theta_{h}(E^{n,0})\\\nonumber
	&=- \Theta_{\ls,g_{\alpha}}\otimes \vvmathbb{1}-\vvmathbb{1}\otimes   ({\theta}_{n,0}^*\wedge  {\theta}_{n,0})\\\label{eq:Hodge bundle2}
	&=- \Theta_{\ls,g_{\alpha}}\otimes \vvmathbb{1}- \tilde{\theta}_{n,0}^*\wedge \tilde{\theta}_{n,0},
	\end{align}
	where we set \(\tilde{\theta}_{n-k,k}:=\vvmathbb{1}\otimes \theta_{n-k,k} : \ls^{-1}\otimes E^{n-k,k}\rightarrow \ls^{-1}\otimes  E^{n-k-1,k+1}\otimes \Omega_{Y}\big(\log (D+S)\big) \), and define $\tilde{\theta}_{n,0}^*$ to be the adjoint of $\tilde{\theta}_{n,0}$ with respect to the metric $h_g^\alpha$. Hence over $V_1$ one has
	\begin{align} \nonumber
	-\sqrt{-1}\d \dbar\log |s|_{h_g^\alpha}^2&=  \frac{\big\{ \sqrt{-1}\Theta_0(s),s\big\}_{h_g^\alpha}}{| s|^2_{h_g^\alpha}}+\frac{\sqrt{-1}\{D's,s \}_{h_g^\alpha}\wedge \{s,D's \}_{h_g^\alpha}}{|s|_{h_g^\alpha}^4}-\frac{\sqrt{-1}\{D's,D's \}_{h_g^\alpha}}{|s|_{h_g^\alpha}^2} \\\label{eq:crucial}
	& \leqslant   \frac{\big\{ \sqrt{-1}\Theta_0(s),s\big\}_{h_g^\alpha}}{| s|^2_{h_g^\alpha}}
	\end{align}	
	thanks to the Lagrange's inequality 
	$$\sqrt{-1}|s|^2_{h_g^\alpha}\cdot \{D's,D's \}_{h_g^\alpha}\geqslant \sqrt{-1}\{D's,s\}_{h_g^\alpha}\wedge \{s,D's \}_{h_g^\alpha}.$$
	Putting \eqref{eq:Hodge bundle2} to \eqref{eq:crucial}, over $V_1$ one has 
	\begin{align} \label{eq:final}
	\sqrt{-1}\Theta_{\ls,g_\alpha}-\sqrt{-1}\d \dbar\log |s|_{h_g^\alpha}^2 \leqslant  -\frac{\big\{ \sqrt{-1}\tilde{\theta}_{n,0}^*\wedge \tilde{\theta}_{n,0}(s),s\big\}_{h_g^\alpha}}{| s|^2_{h_g^\alpha}}= \frac{\sqrt{-1}\big\{ \tilde{\theta}_{n,0}(s),\tilde{\theta}_{n,0}(s)\big\}_{h_g^\alpha}}{| s|^2_{h_g^\alpha}} 
	\end{align}
	where $\tilde{\theta}_{n,0}(s)\in H^0\Big(Y,\ls^{-1}\otimes  E^{n-1,1}\otimes \Omega_{Y}\big(\log (D+S)\big)\Big)$. 
	By \cref{bounded}, for any $y\in D\cup S$, one has
	$$\lim\limits_{y'\in V_0,y'\to y}|s|^2_{h_g^\alpha}(y')=0.$$
	Therefore, it follows from the compactness of $Y$  that there exists   $y_0\in V_0$ so that $|s|^2_{h_g^\alpha}(y_0)\geqslant |s|^2_{h_g^\alpha}(y)$ for any $y \in V_0$. Hence $|s|^2_{h_g^\alpha}(y_0)>0$, and by \eqref{set}, $y_0\in V_1$. Since $|s|^2_{h_g^\alpha}$ is smooth over $V_0$,  
	$
	\sqrt{-1}\d \dbar\log |s|_{h_g^\alpha}^2(y_0)
	$ 
	is semi-negative. By \cref{estimate}, $	\sqrt{-1}\Theta_{\ls,g_\alpha}$ is strictly positive at $y_0$. By \eqref{eq:final} and $|s|_h^2(y_0)>0$, we conclude  that $\sqrt{-1}\big\{  \tilde{\theta}_{n,0}(s),\tilde{\theta}_{n,0}(s)\big\}_{h_g^\alpha}$ is strictly positive at $y_0$. In particular, for any non-zero $\xi \in \ts_{Y,y_0}$, $\tilde{\theta}_{n,0}(s)(\xi)\neq 0$. For
	$$
	\tau_1:\ts_Y (-\log D  )\to \ls^{-1}\otimes E^{n-1,1}
	$$
	in \eqref{iterated Kodaira2}, over $V_0$ it is defined by $\tau_1(\xi):=\tilde{\theta}_{n,0}(s)(\xi)$, which 
	is thus \emph{injective at $y_0\in V_1$}. Hence $\tau_1$ is \emph{generically injective}. The theorem is thus proved.
\end{proof}
\begin{rem}
	Viehweg-Zuo \cite{VZ02} showed that $\tau_1:\ts_Y(-\log D)\to \ls^{-1}\otimes E^{n-1,1}$ defined in \eqref{iterated Kodaira2}  does not vanishing on $V$
	using a global argument relying on the Griffiths curvature computation for Hodge metric and the bigness of
	direct image sheaves due to Kawamata and Viehweg. Moreover, by the work of Viehweg-Zuo \cite{VZ03} and Popa-Taji-Wu \cite{PTW18}, it has already been known to us that, when fibers in \cref{VZ question} have big and semi-ample canonical bundle,  the VZ Higgs bundles constructed in \cref{thm:existence} over the base always satisfy \cref{mainvz}.
\end{rem}
	Though \cref{VZ question} follows from our more general result in \cref{Deng}, we are able to prove \cref{VZ question} by directly applying the  results  by Viehweg-Zuo \cite{VZ03} and Popa-Taji-Wu \cite{PTW18}. Since we need some efforts to prove \cref{Deng}, let us quickly show how to combine their   work with \cref{mainvz} to prove \cref{VZ question}. 
\begin{proof}[Proof of \cref{VZ question}]
By the stratified arguments of Viehweg-Zuo \cite{VZ03}, it suffices to prove that there cannot exists a Zariski dense entire curve. Assume by contradiction that there exists such $\gamma:\cb\to V$. The existence of VZ Higgs bundle on some birational model $\tilde{V}$ of $V$ is known to us by  \cref{thm:existence}.  Let $\tilde{\gamma}:\cb\to \tilde{V}$ is the lift of $\gamma$ which is also Zariski dense.  In \cite{VZ03,PTW18}, the authors proved that the restriction of $\tau_1$ defined in    \eqref{iterated Kodaira2} on $\cb$, say $\tau_1|_{\cb}:\ts_{\cb}\to  \tilde{\gamma}^*(\ls^{-1}\otimes E^{n-1,1})$, has to vanish identically, or else,  they can construct  a pseudo hermitian metric on $\cb$ with strictly negative Gaussian curvature, which violates the Ahlfors-Schwarz lemma. By \cref{mainvz}, this cannot happen since $\tilde{\gamma}:\cb\to \tilde{V}$ is Zariski dense. The theorem is proved.
\end{proof} 
  \section{Pseudo Kobayashi hyperbolicity of the base}
 In this section we first establish an algorithm to construct Finsler metrics whose holomorphic sectional curvatures are bounded above by a negative constant via VZ Higgs bundles. By our construction and  generic local Torelli \cref{mainvz},   those Finsler metrics are positively definite over a Zariski open set, and by the Ahlfors-Schwarz lemma, we prove that a quasi-projective manifold is pseudo Kobayashi hyperbolic once it is equipped with a VZ Higgs bundle, and thus prove \cref{Deng}.

\subsection{Finsler metric and (pseudo) Kobayashi hyperbolicity}\label{sec:Finsler}
Throughout this subsection  \(X\) will denote to be  a complex manifold of dimension $n$.

\begin{dfn}[Finsler metric]\label{def:Finsler}
	Let \(\es\) be a  holomorphic vector bundle on $X$. A \emph{Finsler metric}\footnote{This definition is  a bit different from the definition in \cite{Kob98}, which requires \emph{convexity} or \emph{triangle inequality}, and the Finsler metric there  can be upper-semi continuous.}   on \(\es\) is a real non-negative  continuous  function \(F:\es\to  \mathclose[ 0,+\infty\mathopen[ \) such
	that 
	\[F(av) = |a|F(v)\]
	for any \(a\in\cb  \) and \(v\in  \es\).  The Finsler metric\(F\) is \emph{positively definite} at some subset \(S\subset X \) if for any \(x\in S\) and any non-zero vector \(v\in \es_x \),  \(F(v)> 0 \).   
\end{dfn}

When \(F\) is  a  Finsler metric on \(\ts_X\), we also say that $F$ is  a  \emph{Finsler metric on  $X$}.

Let \(\es\) and \(\gs\) be two locally free sheaves on \(X\), and suppose that there is a morphism 
\[
\varphi: {\rm Sym}^m\es \to\gs
\]
Then for any Finsler metric \(F\) on  \(\gs\), \(\varphi \) induces a pseudo metric \((\varphi^*F)^{\frac{1}{m}} \) on \(\es\) defined by
\begin{eqnarray}\label{induced metric}
(\varphi^*F)^{\frac{1}{m}}(e):=F\big(\varphi(e^{\otimes m})\big)^{\frac{1}{m}}
\end{eqnarray}
for any \(e\in \es \). It is easy to verify that \( (\varphi^*F)^{\frac{1}{m}}\) is also a Finsler metric on $\es$. Moreover, if over some open set  \(U\), \(\varphi \) is an injection as a morphism between vector bundles, and \(F \) is positively definite over \(U\), then \((\varphi^*F)^{\frac{1}{m}} \) is also positively definite over \(U\).
\begin{dfn}\label{def:hyperbolicity}
	\begin{thmlist}
		\item
		The \emph{Kobayashi-Royden  infinitesimal  pseudo-metric} of $X$ 
		is a length function $\kappa_X:\ts_X\to [0,+\infty[$,   defined by
		\begin{eqnarray}\label{KR}
		\kappa_{X}(\xi)=\inf_{\gamma} \big\{ \lambda>0 \mid \exists \gamma:\db\rightarrow X,  \gamma(0)=x, \lambda\cdot \gamma'(0)= \xi  \big\}
		\end{eqnarray}
		for any \(x\in X\) and \(\xi\in \ts_X \),	where \(\db \) denotes the unit disk in \(\cb \). 
		\item \label{def:distance} The \emph{Kobayashi pseudo distance} of $X$, denoted by $d_X:X\times X\to [0,+\infty[$, is 
		$$
		d_X(p,q)=\inf_{\ell}  \int_{0}^{1}\kappa_X\big(\ell'(\tau)\big)d\tau
		$$
		for every pair of points $p,q\in X$, where the infimum is taken over all differentiable curves $\ell:[0,1]\to X$ joining $p$
		to $q$. 
		\item \label{def:pseudo}Let $\Delta\subsetneq X$ be a closed subset. A complex manifold $X$ is \emph{Kobayashi hyperbolic modulo $\Delta$} if $d_X(p,q)>0$ for every pair of distinct points $p,q\in X$ not both contained in $\Delta$. When $\Delta$ is an empty set, the  manifold $X$ is   \emph{Kobayashi
			hyperbolic}; when $\Delta$ is proper and Zariski closed, the   manifold $X$ is \emph{pseudo Kobayashi
			hyperbolic}.
	\end{thmlist}
\end{dfn}
By definition it is easy to show that if \(X\) is   Kobayashi hyperbolic (resp. pseudo Kobayashi hyperbolic), then \(X\) is Brody hyperbolic (resp. algebraically degenerate). Brody's theorem says that when $X$ is compact,  $X$  is Kobayashi hyperbolic if  it is Brody hyperbolic.   However unlike the
case of Kobayashi hyperbolicity, no criteria is known for pseudo Kobayashi hyperbolicity of a compact complex space in terms of entire curves. Moreover, there are many examples of complex (quasi-projective) manifolds  which are Brody hyperbolic but not Kobayashi hyperbolic.

For any holomorphic map \(\gamma:\db\rightarrow X \), the Finsler metric \(F\) induces a continuous Hermitian pseudo-metric on \(\db \)
\[
\gamma^*F^2= \sqrt{-1}\lambda(t)  dt\wedge d\bar{t},
\]
where \(\lambda(t)\) is a non-negative continuous function on \(\db \). The \emph{Gaussian curvature} \(K_{\gamma^*F^2}\) of the pseudo-metric \(\gamma^*F^2\) is defined to be 
\begin{eqnarray}\label{Gauss}
K_{\gamma^*F^2}:=-\frac{1}{\lambda}\frac{\d^2\log \lambda}{\d t {\d}\bar{t}}.
\end{eqnarray}
\begin{dfn}\label{HSC}	
	Let \(X\) be a complex manifold endowed with  a Finsler metric \(F\). 
	\begin{thmlist}
		\item \label{def:sectional}For any \(x\in X \), and \(v\in  \ts_{X,x} \), let \([v]\) denote the complex line spanned by \(v\). We define the holomorphic sectional curvature \(K_{F,[v]}\) in the direction of \([v] \) by 
		\[
		K_{F,[v]}:= \sup K_{\gamma^*F^2}(0)
		\]
		where the supremum is taken over all \(\gamma:\db\rightarrow X \) such that \(\gamma(0)=x\) and \([v]\)
		is tangent to \(\gamma'(0)\).
		\item \label{negatively curved}	We say that   $F$  is \emph{negatively curved} if  there is a
		positive constant \(c\) such that \(K_{F,[v]}\leqslant -c\) for all \(v\in \ts_{X,x} \) for which \(F(v)>0\).
		\item \label{degenerate}	A point \(x\in X \) is   a \emph{degeneracy point} of \(F\) if \(F(v)=0\) for some
		nonzero \(v\in \ts_{X,x} \), and the set of such points  is denoted by  \(\Delta_{F} \).
	\end{thmlist}
\end{dfn}
As mentioned in \cref{introduction},  our negatively curved  Finsler metrics are only constructed on birational models of the base spaces in \cref{Deng,main}, we thus have to establish     bimeromorphic criteria for (pseudo) Kobayashi hyperbolicity  to prove the main theorems.
\begin{lem}[Bimeromorphic criteria for pseudo Kobayashi hyperbolicity]  \label{pseudo Kobayashi}
	Let $\mu:{X}\to Y$   be a bimeromorphic morphism between complex manifolds.   If there exists a Finsler metric \(F\) on $X$ which is negatively curved in the sense of \cref{negatively curved}, then   \(X\) is Kobayashi hyperbolic modulo  \(\Delta_F\), and $Y$ is Kobayashi hyperbolic modulo $\mu\big({\rm Ex}(\mu)\cup \Delta_F\big)$, where ${\rm Ex}(\mu)$ is the exceptional locus of $\mu$. In particular, when $\Delta_F$ is a proper analytic subvariety of $X$, both $X$ and $Y$ are pseudo Kobayashi hyperbolic. 
\end{lem}
\begin{proof}
	The first statement  is a slight variant of   \cite[Theorem 3.7.4]{Kob98}. By normalizing $F$ we may assume that $K_F\leqslant -1$.  By the Ahlfors-Schwarz lemma, one has 
	$
	F\leqslant \kappa_X.
	$ 
	Let $\delta_F:X\times X\to [0,+\infty[$ be the distance
	function on $X$ defined by $F$ in a similar way as $d_X$:
	$$
	\delta_F(p,q):=\inf_{\ell}  \int_{0}^{1}F\big(\ell'(\tau)\big)d\tau
	$$
	for every pair of points $p,q\in X$, where the infimum is taken over all differentiable curves $\ell:[0,1]\to X$ joining $p$
	to $q$.   Since $F$ is continuous and positively definite over $X\setminus \Delta_F$,    
	for any $p\in X\setminus \Delta_F$, one has $d_X(p,q)\geqslant \delta_F(p,q)>0$ for any $q\neq p$, which proves the first statement.
	
	Let us denote by \({\rm Hol}(Y,y)  \) to be the set of holomorphic maps \(\gamma:\db\to Y \) with \(\gamma(0)=y \). Pick any point $y\in U:=Y\setminus \mu\big({\rm Ex}(\mu)\big)$, then there is a unique point \(x \in {X} \) with \(\mu({x})=y \). Hence \(\mu \) induces a bijection between the sets
	\[
	{\rm Hol}({X},{x})\stackrel{\simeq}{\to} {\rm Hol}(Y,y)
	\]
	defined by \(\tilde{\gamma}\mapsto \mu\circ \tilde{\gamma} \). Indeed, observe that \(\mu^{-1}:Y\dashrightarrow X\) is a meromorphic map, so is \(\mu^{-1}\circ {\gamma} \) for any \({\gamma}\in  {\rm Hol}(Y,y)\). Since \(\dim \db=1 \),  the map \(\mu^{-1}\circ {\gamma} \) is moreover holomorphic. It follows from \eqref{KR} that 
	\begin{eqnarray*} 
		\kappa_{{X}}({\xi})=\kappa_{{Y}}\big( {\mu_*({\xi})}\big)
	\end{eqnarray*}
	for any \(\xi\in \ts_{{X},{x}} \). Hence one has 
	$$
	\mu^*\kappa_Y|_{\mu^{-1}(U)}=\kappa_X|_{\mu^{-1}(U)}\geqslant F|_{{\mu^{-1}(U)}}.
	$$
	Let   $G:\ts_U\to [0,+\infty[$ be the Finsler metric on $U$ so that $\mu^{*}G=F|_{{\mu^{-1}(U)}}$. Then $G$ is continuous and positively definite over $U\setminus \mu(\Delta_F)$, and one has 
	$$
	\kappa_Y|_{U}\geqslant G.
	$$
	Therefore,   for any $y\in Y\setminus \mu\big(\Delta_F\cup {\rm Ex}(\mu)\big)$, one has $d_Y(y,z)>0$ for any $z\neq y$, which proves the second statement. 	\end{proof}

The above criteria can be refined further to show the Kobayashi hyperbolicity of the complex manifold.  


\begin{lem}[Bimeromorphic criteria for Kobayashi hyperbolicity]\label{bimeromorphic}
	Let \(X\) be a complex manifold. Assume that for each  point  \(p\in X \), there is  a bimeromorphic morphism \(\mu:\tilde{X}\to X \) with  \(\tilde{X}\) equipped with a negatively  curved Finsler metric \(F \)  such that \(p\notin \mu\big(\Delta_F  \cup {\rm Ex}(\mu) \big) \). Then \(X\) is Kobayashi hyperbolic.	 
\end{lem}
\begin{proof}
	It suffices to show that $d_X(p,q)>0$ for every pair of distinct points $p,q\in X$. We take the bimeromorphic morphism \(\mu:\tilde{X}\to X \) in the lemma with respect to $p$. By \cref{pseudo Kobayashi}, $X$ is  Kobayashi hyperbolic modulo $\mu\big(\Delta_F  \cup {\rm Ex}(\mu) \big)$, which shows that $d_X(p,q)>0$ for any $q\neq p$. The lemma follows.
\end{proof}

\subsection{Curvature formula}\label{sec:curvature}
Let $(\tilde{\es},\tilde{\theta})$ be the VZ Higgs bundles on a quasi-projective manifold $V$ defined in \cref{sec:VZ}.   
In the next two subsections, we will  construct a negatively curved Finsler metric on \(V\) via $(\tilde{\es},\tilde{\theta})$. Our main result is the following.
\begin{thm}[Existence of negatively curved Finsler metrics]\label{construction of Finsler}
	Same notations as \cref{def:VZ}. Assume that $\tau_{1}$ is injective over a non-empty Zariski open set $V_1\subseteq Y\setminus D\cup S$. Then there exists a  Finsler metric \(F\) (see \eqref{Finsler} below) on \(\ts_{Y}(-\log D) \) such that
	\begin{thmlist}
		\item \label{choice of open} it is  positively definite over   \(V_1\).
		\item When $F$ is seen as a Finsler metric on \(V=Y\setminus D\), it is negatively curved in the sense of \cref{negatively curved}. 
	\end{thmlist} 
\end{thm}

Let us first construct  the desired Finsler metric $F$, and we then proved the curvature property. 
By \eqref{iterated Kodaira2}, for each \(k=1,\ldots,n \), there exists
\begin{eqnarray}\label{contain}
\tau_k: {\rm Sym}^k \ts_{Y}(-\log D) \rightarrow \ls^{-1}\otimes E^{n-k,k}.
\end{eqnarray}
Then it follows from \cref{bounded} that the Finsler metric \(h_g^\alpha \) on \( \ls^{-1}\otimes E^{n-k,k} \) induces  a Finsler metric \(F_k\) on \(\ts_Y(-\log D) \) defined as follows: for any \(e\in  \ts_{Y}(-\log D)_y \),
\begin{eqnarray}\label{k-Finsler}
F_k(e):= (\tau_k^*h_g^\alpha)^{\frac{1}{k}}(e) =h_g^\alpha\big(\tau_k(e^{\otimes k})\big)^{\frac{1}{k}}
\end{eqnarray}
For any \(\gamma:\db\rightarrow  V \), one has
\begin{align*} 
d\gamma:\ts_{\db}\to \gamma^*\ts_V\hookrightarrow \gamma^*\ts_{Y}(-\log D)
\end{align*}
and thus the Finsler metric \(F_k\) induces a continuous Hermitian pseudo-metric on \(\db \), denoted by
\begin{eqnarray}\label{seminorm}
\gamma^*F^2_k:=\sqrt{-1}G_k(t) dt\wedge d\bar{t}.
\end{eqnarray}
In general, \(G_k(t)\) may be identically equal to zero for all $k$. However, 
if we further assume that    \(\gamma(\db)\cap V_1\neq \varnothing\), by  the assumption in \cref{construction of Finsler} that the restriction of $\tau_1$ to $V_1$ is injective, one has \(G_1(t)\not\equiv 0 \). 
Denote by \(\d_t:=\frac{\d}{\d t} \)  the canonical vector fields in \( \db\), and \(\bar{\d}_t:=\frac{\d}{\d \bar{t}}\) its conjugate.  Set \(C:=\gamma^{-1}(V_1) \), and note that \(\db\setminus C \) is a discrete set in \(\db \).     
\begin{lem}\label{curvatureest}
	Assume that \(G_k(t)\not\equiv 0 \) for some \( k>1\). Then   the \emph{Gaussian curvature} \(K_k\) of the continuous pseudo-hermitian metric \(\gamma^*F_k^2 \) on \(C\) satisfies that
	\begin{eqnarray}\label{sectional curvature}
	K_k:=-\frac{	\d^2\log G_k}{\d t\d \bar{t}}/G_k\leqslant \frac{1}{k}\Big(-\big(\frac{G_k}{G_{k-1}} \big)^{k-1}+\big(\frac{G_{k+1}}{G_k} \big)^{k+1}  \Big)
	\end{eqnarray}
	over \(C\subset \db \).
\end{lem}
\begin{proof}
	For $i=1,\ldots,n$, let us write \(e_i:=\tau_{i}\big(d\gamma (\d_t)^{\otimes i}\big)\), which can be seen as a section of $\gamma^*(\ls^{-1}\otimes E^{n-i,i})$. Then by \eqref{k-Finsler} one observes that 
	\begin{eqnarray}\label{G_i}
	G_i(t)=\lVert e_i\rVert_{h^\alpha_g}^{2/i}.
	\end{eqnarray}  
	Let    \(\mathscr{R}_k=\Theta_{h_g^\alpha}(\ls^{-1}\otimes E^{n-k,k})\) be the  curvature form of \(\ls^{-1}\otimes E^{n-k,k}\) on $V_0:=Y\setminus D\cup S$ induced by the metric \(h_g^\alpha=g_\alpha^{-1}\cdot h\) defined in \cref{bounded}, and let $D'$ be the $(1,0)$-part of the Chern connection  $D$ of $(\ls^{-1}\otimes E^{n-k,k},h_g^\alpha)$. Then  for $k=1,\ldots,n$, one has
	\begin{align*} 
	-\sqrt{-1}\d \dbar\log G_k&= \frac{1}{k}\Big(\frac{\big\{  \sqrt{-1}\rs_k(e_k),e_k\big\} _{h_g^\alpha}}{\lVert e_k\rVert^2_{h_g^\alpha}}+\frac{\sqrt{-1}\{D'e_k,e_k \} _{h_g^\alpha}\wedge \{e_k,D'e_k \} _{h_g^\alpha}}{\lVert e_k\rVert_{h_g^\alpha}^4}-\frac{\sqrt{-1}\{D'e_k,D'e_k \} _{h_g^\alpha}}{\lVert e_k\rVert_{h_g^\alpha}^2}\Big)\\
	& \leqslant  \frac{1}{k}\frac{\big\{  \sqrt{-1}\rs_k(e_k),e_k\big\} _{h_g^\alpha}}{\lVert e_k\rVert^2_{h_g^\alpha}}
	\end{align*}	
	thanks to the Lagrange's inequality 
	$$\sqrt{-1}\lVert e_k\rVert^2_{h_g^\alpha}\cdot \{D'e_k,D'e_k \} _{h_g^\alpha}\geqslant \sqrt{-1}\{D'e_k,e_k \} _{h_g^\alpha}\wedge \{e_k,D'e_k \} _{h_g^\alpha}.$$
	Hence
	\begin{eqnarray}\label{curvature formula}
	-\frac{	\d^2\log G_k}{\d t\d \bar{t}}\leqslant \frac{1}{k}\cdot\frac{\big\langle \rs_k(e_k)(\d_t,\bar{\d}_t),e_k\big\rangle_{h_g^\alpha}}{\lVert e_k\rVert^2_{h_g^\alpha}}.
	\end{eqnarray}	
	Recall that for the logarithmic Higgs bundle \((\bigoplus_{k=0}^{n}E^{n-k,k}, \bigoplus_{k=0}^{n}\theta_{n-k,k}) \), 
	the curvature \(\Theta_k \) on \(E^{n-k,k}_{\upharpoonright V_0} \) induced by the Hodge metric \(h\) is given by
	\[
	\Theta_k=-\theta_{n-k,k}^*\wedge \theta_{n-k,k}-\theta_{n-k+1,k-1}\wedge \theta^*_{n-k+1,k-1},
	\]
	where we recall that  $\theta_{n-k,k}:E^{n-k,k}\to E^{n-k-1,k+1}\otimes \Omega_{Y}\big(\log (D+S)\big)$. 
	Set \(\tilde{\theta}_{n-k,k}:=\vvmathbb{1}\otimes \theta_{n-k,k} : \ls^{-1}\otimes E^{n-k,k}\rightarrow \ls^{-1}\otimes  E^{n-k-1,k+1}\otimes \Omega_{Y}\big(\log (D+S)\big) \), and  one has
	\[
	\xymatrixcolsep{5pc}\xymatrix{\ls^{-1}\otimes E^{n-k+1,k-1} \ar@/^1pc/[r]^{{\tilde{\theta}_{n-k+1,k-1}(\d_t)}} &\ls^{-1}\otimes E^{n-k,k} \ar@/^1pc/[l]^{\tilde{\theta}^*_{n-k+1,k-1}(\bar{\d}_t)} \ar@/^1pc/[r]^{{\tilde{\theta}_{n-k,k}(\d_t)}} & \ls^{-1}\otimes E^{n-k-1,k+1} \ar@/^1pc/[l]^{\tilde{\theta}^*_{n-k,k}(\bar{\d}_t)} }
	\]
	where $\tilde{\theta}^*_{n-k,k}$ is the adjoint of $\tilde{\theta}_{n-k,k}$ with respect to the metric $h_g^\alpha$ over $Y\setminus D\cup S$. Here we also write $\d_t$ (resp. $\bar{\d}_t$) for $d\gamma(\d_t)$ (resp. ${d\gamma(\bar{\d}_t)}$ ) abusively. Then over $V_0$, we   have
	\begin{align}\label{eq:Hodge bundle}
	\rs_k=-\Theta_{\ls,g_\alpha}\otimes \vvmathbb{1}+\vvmathbb{1}\otimes \Theta_k=-\Theta_{\ls,g_\alpha}\otimes \vvmathbb{1}-\tilde{\theta}_{n-k,k}^*\wedge \tilde{\theta}_{n-k,k}-\tilde{\theta}_{n-k+1,k-1}\wedge \tilde{\theta}^*_{n-k+1,k-1}.
	\end{align}
	By 
	the definition of $\tau_k$ in \eqref{iterated Kodaira2}, for any $k=2,\ldots,n$ one has
	\begin{eqnarray}\label{relation}
	e_{k}=\tilde{\theta}_{n-k+1,k-1}(\d_t)(e_{k-1}),
	\end{eqnarray}
	and we can derive   the following  curvature formula
	\begin{align*} 
	\langle \rs_k(e_k)(\d_t,\bar{\d}_t),e_k\big\rangle_{h_g^\alpha}&=-\Theta_{\ls,g_\alpha}(\d_t,\bar{\d}_t)\lVert e_k\rVert_{h_g^\alpha}^2+\\
	&\ \ \ \ \big\langle\tilde{\theta}_{n-k,k}^*(\bar{\d}_t)\circ \tilde{\theta}_{n-k,k}(\d_t)(e_k)-\tilde{\theta}_{n-k+1,k-1}(\d_t)\circ \tilde{\theta}^*_{n-k+1,k-1}(\bar{\d}_t)(e_k),e_k\big\rangle_{h_g^\alpha}\\
	&\leqslant\big\langle\tilde{\theta}_{n-k,k}^*(\bar{\d}_t)\circ \tilde{\theta}_{n-k,k}(\d_t)(e_k),e_k\big\rangle_{h_g^\alpha}\\
	&\ \ \ -\big\langle\tilde{\theta}_{n-k+1,k-1}(\d_t)\circ \tilde{\theta}^*_{n-k+1,k-1}(\bar{\d}_t)(e_k),e_k\big\rangle_{h_g^\alpha}\\
	&\stackrel{\eqref{relation}}{=}\lVert e_{k+1}\rVert_{h_g^\alpha}^2-\lVert \tilde{\theta}^*_{n-k+1,k-1}(\bar{\d}_t)(e_k)\rVert^2_{h_g^\alpha}\\
	&\leqslant\lVert e_{k+1}\rVert_{h_g^\alpha}^2-\frac{|\big\langle \tilde{\theta}^*_{n-k+1,k-1}(\bar{\d}_t)(e_k),e_{k-1}\big\rangle_{h_g^\alpha}|^2}{\lVert e_{k-1}\rVert_{h_g^\alpha}^2}\quad {(\mbox{Cauchy-Schwarz inequality})}  \\
	&=\lVert e_{k+1}\rVert_{h_g^\alpha}^2-\frac{|\big\langle e_k,\tilde{\theta}_{n-k+1,k-1}({\d}_t)(e_{k-1})\big\rangle_{h_g^\alpha}|^2}{\lVert e_{k-1}\rVert_{h_g^\alpha}^2}  \\
	&\stackrel{\eqref{relation}}{=}\lVert e_{k+1}\rVert_{h_g^\alpha}^2-\frac{\lVert e_{k}\rVert_{h_g^\alpha}^4}{\lVert e_{k-1}\rVert_{h_g^\alpha}^2} \\
	&\stackrel{\eqref{G_i}}{=}G_{k+1}^{k+1}-\frac{G_k^{2k}}{G_{k-1}^{k-1}}
	\end{align*}
	Putting this into \eqref{curvature formula}, we  obtain   \eqref{sectional curvature}.
\end{proof}
\begin{rem}\label{Griffiths}
	For the final stage $E^{0,n}$ of the Higgs bundle \((\bigoplus_{q=0}^{n} E^{n-q,q},\bigoplus_{q=0}^{n}\theta_{n-q,q}) \). We make the convention that \(G_{n+1}\equiv 0 \). Then the Gaussian curvature for \(G_n\) in \eqref{curvature formula} is always semi-negative, which is similar as the Griffiths curvature formula for   Hodge bundles in \cite{GT84}.
\end{rem}

When \(k=1 \), by \eqref{curvature formula} one has
\begin{align*}
-\frac{\d^2\log G_1}{\d t\d \bar{t}}/G_1
&\leqslant  \frac{\big\langle \rs_1(e_1)(\d_t,\bar{\d}_t),e_1\big\rangle_{h_g^\alpha}}{\lVert e_1\rVert_{h_g^\alpha}^4}\\
&\stackrel{\eqref{eq:Hodge bundle}}{= }   \frac{-\Theta_{\ls,g_\alpha}(\d_t,\bar{\d}_t)}{\lVert e_1\rVert_{h_g^\alpha}^2}+\\
&\ \ \ \ \frac{\big\langle\tilde{\theta}_{n-1,1}^*(\bar{\d}_t)\circ \tilde{\theta}_{n-1,1}(\d_t)(e_1)-\tilde{\theta}_{n,0}(\d_t)\circ \tilde{\theta}^*_{n,0}(\bar{\d}_t)(e_1),e_1\big\rangle_{h_g^\alpha}}{\lVert e_1\rVert_{h_g^\alpha}^4}\\
&\stackrel{\eqref{relation}}{\leqslant}   \frac{-\Theta_{\ls,g_\alpha}(\d_t,\bar{\d}_t)\lVert e_1\rVert_{h_g^\alpha}^2+\lVert e_{2}\rVert_{h_g^\alpha}^2}{\lVert e_1\rVert_{h_g^\alpha}^4}\\
&=   \frac{-\Theta_{\ls,g_\alpha}(\d_t,\bar{\d}_t) }{\lVert e_1\rVert_{h_g^\alpha}^2}+\big(\frac{G_2}{G_1}\big)^2
\end{align*}
We need the following lemma to control the negative term in the above inequality. 
\begin{lem}
	When \(\alpha\gg 0 \), there exists a \emph{universal} constant \(c>0\), such that for any \(\gamma:\db\rightarrow V \) with  \(\gamma(\db)\cap V_0\neq \varnothing\), one has
	\[
	\frac{\Theta_{\ls,g_\alpha}(\d_t,\bar{\d}_t) }{\lVert e_1\rVert_{h_g^\alpha}^2}\geqslant c.
	\]
	In particular,
	\[
	-\frac{\d^2\log G_1}{\d t\d \bar{t}}/G_1\leqslant -c+ \big(\frac{G_2}{G_1}\big)^2
	\]
\end{lem}
\begin{proof}
	By \cref{bounded}, it suffices to prove that 
	\begin{eqnarray}\label{strict positive}
	\frac{\gamma^*\big( r_D^{-2}\cdot \omega_\alpha \big)(\d_t,\bar{\d}_t)}{\lVert e_1\rVert_{h_g^\alpha}^2}\geqslant c.
	\end{eqnarray}
	Note that
	\[
	\frac{\gamma^*\big( r_D^{-2}\cdot \omega_\alpha \big)(\d_t,\bar{\d}_t)}{\lVert e_1\rVert_{h_g^\alpha}^2}=\frac{\gamma^*\big(  \omega_\alpha \big)(\d_t,\bar{\d}_t)}{\gamma^*(r_D^{2})\cdot \lVert e_1\rVert_{h_g^\alpha}^2}=\frac{\gamma^*\omega_\alpha(\d_t,\bar{\d}_t)}{\gamma^*\tau_1^*( r_D^{2}\cdot h_g^\alpha)(\d_t,\bar{\d}_t)},
	\]
	where \(\tau_1^*( r_D^{2}\cdot h_g^\alpha) \) is the Finsler metric on \(\ts_{Y}(-\log D )\) defined by \eqref{induced metric}.
	By \cref{new bound},  \(\omega_\alpha \) is a positively definite Hermitian metric on \(\ts_{Y}(-\log D )\).  Since \(Y\) is compact,   there exists a \emph{uniform constant} \(c>0\)  such that
	\[
	\omega_\alpha\geqslant c\tau_1^*( r_D^{2}\cdot h_g^\alpha).
	\] 
	We thus obtained the desired  inequality \eqref{strict positive}.
\end{proof}
In summary, we have the following curvature estimate for the Finsler metrics $F_1,\ldots,F_n$ defined in  \eqref{k-Finsler}, which is  similar as \cite[Lemma 9]{Sch17} for the Weil-Petersson metric.
\begin{proposition}\label{summary}
	For any \(\gamma:\db\rightarrow V \) such that \(\gamma(\db)\cap V_1\neq \varnothing \). Assume that \(G_k\not\equiv 0 \) for \(k=1,\ldots,q\), and \(G_{q+1}\equiv 0\) (thus \(G_j\equiv 0\) for all \(j>q+1\)). Then \(q\geqslant 1\), and over \(C:=\gamma^{-1}(V_1) \), which is a complement of a discrete set in \(\db \), one has
	\begin{align}\label{eq:curvature1}
	-\frac{\d^2\log G_1}{\d t\d \bar{t}}/G_1&\leqslant  -c+ \big(\frac{G_2}{G_1}\big)^2\\   \label{eq:curvature2}
	-\frac{	\d^2\log G_k}{\d t\d \bar{t}}/G_k&\leqslant \frac{1}{k}\Big(-\big(\frac{G_k}{G_{k-1}} \big)^{k-1}+\big(\frac{G_{k+1}}{G_k} \big)^{k+1}  \Big) \quad \forall 1<k \leqslant  q.
	\end{align}
	Here the constant \(c>0\) does not depend on the choice of \(\gamma \).
\end{proposition}

\subsection{Construction of the Finsler metric}\label{sec:construction}
By \cref{summary},  we observe that none of the Finsler metrics $F_1,\ldots,F_n$ defined in  \eqref{k-Finsler} is negatively curved. 
Following the similar strategies in \cite{TY14,Sch17,BPW17}, we construct a new Finsler metric $F$ (see \eqref{Finsler} below) by defining a convex sum of all   \(F_1,\ldots,F_n\),  to cancel the   positive terms in \eqref{eq:curvature1} and \eqref{eq:curvature2} by negative terms in the next stage.  By \cref{Griffiths}, we observe that the highest last order term is always semi-negative.  We mainly follow the computations in \cite{Sch17}, and try to make this subsection as self-contained as possible.  Let us first recall the following basic inequalities by Schumacher.\noindent 
\begin{lem}[\!\protect{\cite[Lemma 8]{Sch12}}]
	Let \(V\) be a complex manifold, and let 
	\(G_1,\ldots,G_n\)  be non-negative \(\mathscr{C}^2\) functions on \(V\). Then
	\begin{eqnarray}\label{calculus}
	\sqrt{-1}\d\bar{\d}\log (\sum_{i=1}^{n}G_i )\geqslant \frac{\sum_{j=1}^{n}G_j\sqrt{-1}\d\bar{\d}\log G_j}{\sum_{i=1}^{n}G_i}
	\end{eqnarray}
\end{lem}

\begin{lem}[\!\protect{\cite[Lemma 17]{Sch17}}]
	Let \(\alpha_j>0 \) for \(j=1,\ldots,n \). Then for all \(x_j\geqslant 0 \)
	\begin{align}\nonumber
	&\sum_{j=2}^{n}(\alpha_{j}x_j^{j+1}-\alpha_{j-1}x_j^j )x_{j-1}^2\cdot\ldots\cdot x_1^2\\\label{complicate}
	&\geqslant \frac{1}{2}\Bigg(-\frac{\alpha_1^3}{\alpha_2^2}x_1^2+\frac{\alpha_{n-1}^{n-1}}{\alpha_n^{n-2}}x_n^2\cdot\ldots\cdot x_1^2+\sum_{j=2}^{n-1}\bigg(\frac{\alpha_{j-1}^{j-1}}{\alpha_j^{j-2}} -\frac{\alpha_{j}^{j+2}}{\alpha_{j+1}^{j+1}}\bigg)x_j^2\cdot\ldots\cdot x_1^2 \Bigg)
	\end{align}
\end{lem}
Set \(x_j=\frac{G_j}{G_{j-1}} \) for \(j=2,\ldots,n \) and \(x_1:=G_1\) where \(G_j\geqslant 0 \) for \(j=1,\ldots,n \). Put them into \eqref{complicate} and we obtain
\begin{align}\nonumber
&\sum_{j=2}^{n}\Big(\alpha_{j}\frac{G_j^{j+1}}{G_{j-1}^{j-1}}-\alpha_{j-1}\frac{G_j^{j}}{G_{j-1}^{j-2}} \Big)\\\label{final}
&\geqslant \frac{1}{2}\Bigg(-\frac{\alpha_1^3}{\alpha_2^2}G_1^2+\frac{\alpha_{n-1}^{n-1}}{\alpha_n^{n-2}}G_n^2+\sum_{j=2}^{n-1}\bigg(\frac{\alpha_{j-1}^{j-1}}{\alpha_j^{j-2}} -\frac{\alpha_{j}^{j+2}}{\alpha_{j+1}^{j+1}}\bigg)G_j^2 \Bigg)
\end{align} 

The following technical lemma is crucial in constructing our negatively curved Finsler metric $F$.
\begin{lem}[\!\protect{ \cite[Lemma 10]{Sch17}}]\label{convex sum}
	Let \(F_1,\ldots,F_n\) be Finsler metrics on a complex space \(X\), with the holomorphic sectional curvatures denoted by \(K_1,\ldots,K_n\). Then for the Finsler metric \(F:=(F^2_1+\ldots+F^2_n)^{1/2}\), its holomorphic sectional curvature
	\begin{eqnarray}\label{sectional}
	K_F\leqslant \frac{\sum_{j=1}^{n}K_jF_j^4}{F^4}.
	\end{eqnarray}
\end{lem}
\begin{proof}
	For any holomorphic map \(\gamma:\db\rightarrow X \), we denote by \(G_1,\ldots,G_n\) the semi-positive functions on \(\db \) such that
	\[
	\gamma^*F^2_i=\sqrt{-1}G_i dt\wedge d\bar{t}
	\]
	for \(i=1,\ldots,n\). 	Then  \[\gamma^*F^2=\sqrt{-1}(\sum_{i=1}^{n} {G_i})dt\wedge d\bar{t},\] 
	and it follows from \eqref{Gauss} that the Gaussian curvature of \(\gamma^*F^2 \)
	\begin{eqnarray*}
		K_{\gamma^*F^2}&=&-\frac{1}{ \sum_{i=1}^{n} {G_i}  }\frac{\d^2\log (\sum_{i=1}^{n} {G_i}) }{\d t {\d}\bar{t}} \\
		&\stackrel{\eqref{calculus}}{\leqslant}&-\frac{1}{ (\sum_{i=1}^{n} {G_i})^2 }\sum_{j=1}^{n} {G_j}\frac{\d^2\log  {G_j} }{\d t {\d}\bar{t}}\\
		&\leqslant&\frac{\sum_{j=1}^{n}K_jG^{2}_j}{(\sum_{i=1}^{n} {G_i})^2}.
	\end{eqnarray*}
	The lemma  follows from \cref{def:sectional}.
\end{proof}

For any \(\gamma:\db\rightarrow V \) with \(C:=\gamma^{-1}(V_1)\neq \varnothing \), we define a   Hermitian pseudo-metric \(\sigma:=\sqrt{-1}H(t)dt\wedge d\bar{t}\) on \(\db\) by taking convex sum in the following form
\[H(t):=\sum_{k=1}^{n}  {{k\alpha_k}G_k}(t), \]  where \(G_k\) is defined in \eqref{seminorm}, and \(\alpha_1,\ldots,\alpha_n\in \mathbb{R}^+ \) are some \emph{universal constants}  which will be fixed later. Following the similar estimate in \cite[Proposition 11]{Sch17}, one can choose those constants properly such that the Gaussian curvature \(K_\sigma\) of \(\sigma\) is uniformly bounded.
\begin{proposition}\label{uniform}
	There exists \emph{universal} constants \(0<\alpha_1\leqslant\ldots\leqslant\alpha_n\) and \(K>0\) (independent of \(\gamma:\db\rightarrow V \)) such that  the Gaussian curvature 
	\[
	K_{\sigma}\leqslant -K.
	\]
	on \(C\).
\end{proposition}
\begin{proof}
	It follows from \eqref{sectional} that 
	\[
	K_{\sigma}\leqslant \frac{1}{H^2}  \sum_{j=1}^{n}j \alpha_j K_jG^{2}_j 
	\]
	and
	\[
	K_j:=-\frac{	\d^2\log G_j}{\d t\d \bar{t}}/G_j.
	\]
	By \cref{summary}, one has 
	\begin{eqnarray*}
		K_\sigma&\leqslant& \frac{\alpha_1G^2_1}{H^2}\bigg(-c+ \Big(\frac{G_2}{G_1}\Big)^2\bigg)+\frac{1}{H^2}\sum_{j=2}^{n} \alpha_jG^2_j  \bigg(-\Big(\frac{G_j}{G_{j-1}} \Big)^{j-1}+\Big(\frac{G_{j+1}}{G_j} \Big)^{j+1}  \bigg)\\
		&\leqslant&\frac{1}{H^2}\bigg( -c\alpha_1G_1^2-\sum_{j=2}^{n}\Big(\alpha_j\frac{G_j^{j+1}}{G_{j-1}^{j-1}}-\alpha_{j-1}\frac{G_j^{j}}{G_{j-1}^{j-2}} \Big) \bigg)\\
		&\stackrel{\eqref{final}}{\leqslant} & \frac{1}{H^2}\bigg(\Big(-c+\frac{1}{2}\frac{\alpha_1^2}{\alpha_2^2}\Big)\alpha_1G_1^2+\frac{1}{2}\sum_{j=2}^{n-1}\Big(\frac{\alpha_j^{j+2}}{\alpha_{j+1}^{j+1}}-\frac{\alpha_{j-1}^{j-1}}{\alpha_{j}^{j-2}} \Big)G_j^2-\frac{1}{2}\frac{\alpha_{n-1}^{n-1}}{\alpha_n^{n-2}}G_n^2  \bigg)\\
		&=:&-\frac{1}{H^2}\sum_{j=1}^{n}\beta_jG_j^2
	\end{eqnarray*}
	One can take \(\alpha_1=1 \), and choose the further \(\alpha_j>\alpha_{j-1} \) inductively such that \(\min_j \beta_j>0 \). Set \(\beta_0:=\min_j \frac{\beta_j}{(j\alpha_j)^2} \). Then
	\begin{eqnarray*}
		K_\sigma&\leqslant& -\frac{1}{H^2}\beta_0\sum_{j=1}^{n}(j\alpha_jG_j)^2\\
		&\leqslant & -\frac{\beta_0}{nH^2}(\sum_{j=1}^{n}{j\alpha_j}G_j)^2\\
		&=&-\frac{\beta_0}{n}=:-K.
	\end{eqnarray*}
	Note that \(\alpha_1,\ldots,\alpha_n \) and \(K\) is universal. The lemma is thus proved.
\end{proof}
It follows from \cref{uniform} and \eqref{Gauss} that one has the following estimate
\begin{eqnarray}\label{subharmonic}
\frac{\d ^2 \log H(t)}{\d t\d \bar{t}}\geqslant KH(t)\geqslant 0
\end{eqnarray}
over  the Zariski dense open set  \(  C\subseteq  \db \), and in particular \(\log H(t) \) is a subharmonic function over \(C\). Since \(H(t)\in  \mathclose[0,+\infty \mathclose[ \) is continuous   (in particular locally bounded from above) over \(\db\),   \(\log H(t) \) is a subharmonic function over \(\db\), and the estimate \eqref{subharmonic} holds over the whole \(\db \). 

In summary, we construct a \emph{negatively curved     Finsler metric} \(F\) on \(Y\setminus D \), defined by
\begin{eqnarray}\label{Finsler}
F:=(\sum_{k=1}^{n}{k\alpha_k}F^2_k)^{1/2},
\end{eqnarray}
where \(F_k\) is defined in \eqref{k-Finsler}, such that  \(\gamma^*F^2=\sqrt{-1}H(t)dt\wedge d\bar{t} \) for any \(\gamma:\db\to V \). Since we assume that $\tau_{1}$ is injective over $V_0$,  the Finsler metric $F_1$ is positively definite on  \(V_0\), and \emph{a fortiori} $F$. 
Therefore, we finish the proof of \cref{construction of Finsler}.

\subsection{Proof of \cref{Deng}}
\begin{proof}[Proof of \cref{Deng}]
	By \cref{thm:existence},   there is a VZ Higgs bundle over some  birational model $\tilde{V}$ of $V$.  By   \cref{mainvz} and    \cref{construction of Finsler},  we can associate this VZ Higgs bundle with a negatively curved Finsler metric which is positively definite over some Zariski dense open set of $\tilde{V}$.   
	The theorem follows directly from the bimeromorphic criteria for pseudo Kobayashi hyperbolicity in \cref{pseudo Kobayashi}.	
\end{proof}
\begin{rem}
Let me mention that Sun and Zuo also have the similar idea in constructing Finsler metric over the base using Viehweg-Zuo Higgs bundles combining with To-Yeung's method \cite{TY14}. 
\end{rem}
\section{Kobayashi hyperbolicity of the base}
In this section we will prove \cref{main}. We first  refine Viehweg-Zuo's result on the positivity of direct images. We then apply this result to take different branch covering in the construction of VZ Higgs bundles to prove the Kobayashi hyperbolicity of the base in \cref{main}.
\subsection{Preliminary for positivity of  direct images}\label{sec:positivity of direct images}
  
We first recall a \emph{pluricanonical extension theorem} due to Cao \cite[Theorem 2.10]{Cao16}.  Its proof  is a combination of the Ohsawa-Takegoshi-Manivel \(L^2\)-extension theorem, with    the semi-positivity of  \(m\)-relative Bergman metric studied by Berndtsson-P\u{a}un \cite{BP08,BP10} and P\u{a}un-Takayama \cite{PT14}. 
\begin{thm}[Pluricanonical $L^2$-extension]\label{extension}
	Let \(f:X\rightarrow Y\) be an algebraic fiber space so that the Kodaira dimension  of the   general fiber is non-negative. Assume that $f$ is smooth over  a dense Zariski open set of \(Y_0\subset Y\) so that both  $B:=Y\setminus Y_0$ and $f^*B$ are normal crossing.   Let \(L\) be any  pseudo-effective line bundle $L$  on $X$ equipped with a positively curved singular metric  $h_L$ with algebraic singularities satisfying the following property
	\begin{thmlist}
		\item  	 \label{2}	There exists some regular value \(z\in Y \) of \(f\), such that  for some \(m\in \mathbb{N} \), all the sections \(H^0\big(X_{z}, (m K_X+L)_{\upharpoonright X_z}\big) \)  extends locally near \(z\).
		\item \(H^0\big(X_{z}, (m K_{X_z}+L_{\upharpoonright X_z})\otimes \js(h_{L\upharpoonright X_z}^{\frac{1}{m}}) \big)\neq\varnothing \).
	\end{thmlist}  
	Then for any regular value $y$ of $f$ satisfying that
	\begin{thmlist}
		\item all sections \(H^0\big(X_{y}, m K_{X_y}+L_{\upharpoonright X_y}\big) \) extends locally near \(y\),
		\item the metric \(h_{L\upharpoonright X_y} \) is not identically equal to \(+\infty \),
	\end{thmlist} 
 the following restriction map 
	\begin{equation*} 
	\begin{tikzcd}
	H^0 (X, m K_{X/Y} -m\Delta_f+ L +f^* A_Y)  \ar[r,twoheadrightarrow] &  H^0 \big(X_y ,  (m K_{X_y} + L_{\upharpoonright X_y})\otimes \js(h_{L\upharpoonright X_y}^{\frac{1}{m}})\big)  
	\end{tikzcd}
	\end{equation*}
is surjective. Here \(A_Y\)  is a \emph{universal} ample line bundle  on \(Y\) which does not depend on \(L\), \(f\) and \(m\),    and 
	 	\begin{eqnarray}\label{eq:multiplicity}
 	\Delta_{f}:=\sum_{j}(a_j-1) V_j.
	\end{eqnarray} 
	where the sum is taken over all prime divisors $V_j$ of $f^*B$ with multiplicity $a_j$ and its image $f(V_j)$    a divisor in $Y$.
\end{thm}


We will apply a technical lemma  in \cite[Claim 3.5]{CP17}  to prove  \cref{global generated}.  Let us   first recall some definitions of singularities of divisors in \cite[Chapter 5.3]{Vie95} in a slightly different language.\begin{dfn}
	Let \(X\) be a smooth projective variety, and let \(\ls \) be a line bundle such that \(H^0(X,\ls)\neq \varnothing \).  One defines
	\begin{align}\label{el}
	e(\ls)= {\rm sup}\big\{\frac{1}{c(D)} \mid   D\in |\ls| {{\rm\ is\ an\ effective\ divisor}}  \big\}
	\end{align}
	where
	\[
	c(D):= {\rm sup}\{c>0 \mid (X,c\cdot D) {{\rm\ is \ a \ klt\ divisor}}\}
	\]
	is the \emph{log canonical threshold} of \(D\).
\end{dfn}
Viehweg showed that one can control the lower bound of \(e(\ls) \).
\begin{lem}[\!\protect{\cite[Corollary 5.11]{Vie95}}]\label{control}
	Let \(X\) be a smooth projective variety equipped with a very ample line bundle \(\hs \), and let \(\ls \) be a line bundle such that \(H^0(X,\ls)\neq \varnothing \). 
	\begin{thmlist}
		\item Then there is a uniform estimate
		\begin{align}\label{lct}
		e(\ls)\leqslant c_1(\hs)^{\dim X-1 }\cdot c_1(\ls)+1.
		\end{align}\label{equal}
		\item Let \(Z:=X\times\cdots\times X \) be the \(r\)-fold product. Then for \(\mathscr{M}:=\bigotimes_{i=1}^{r}{\rm pr}_i^*\ls \), one has \(e(\mathscr{M})=e(\ls) \).
	\end{thmlist}
\end{lem}
Let us recall the following result by Cao-P\u{a}un \cite{CP17}.
\begin{lem}[Cao-P\u{a}un]\label{CP}
	Let \(f:X\rightarrow Y\) be an algebraic fiber space so that the Kodaira dimension  of the   general fiber is non-negative. Assume that $f$ is smooth over  a dense Zariski open set of \(Y_0\subset Y\) so that both  $B:=Y\setminus Y_0$ and $f^*B$ are normal crossing.   
	Then there exists some positive integer \(C\geqslant 2\) so that for any  \(m\geqslant m_0\) and \(a\in \mathbb{N} \),   any \(y\in Y_0 \) and    any section 
	\[
	\sigma \in H^0(X_y,  amC K_{X_y} ),
	\]
	there exists a  section
	\begin{align}\label{eq:Sigma}
	\Sigma\in H^0\big(X,  f^*A_Y-   af^*\det f_*(mK_{X/Y}) +amr_mCK_{X/Y}+ a(P_m+F_m)  \big)
	\end{align}
	whose restriction to the fiber \(X_y\) is equal to \(\sigma^{\otimes r_m} \).  Here   \(F_m\) and \(P_m \) are effective divisors on \(X\) (independent of \(a\)) such that $F_m$ is $f$-exceptional with \(f(F_m)\subset {\rm Supp}(B) \), \( {\rm Supp}(P_m)\subset  {\rm Supp}(\Delta_f)\), \(r_m:=\rank f_*(mK_{X/Y}) \),  and \(A_Y\) is the universal ample line bundle on \(Y\) defined in \cref{extension}.  
\end{lem}

We     recall the definition of   \emph{Koll\'ar family of varieties with semi-log canonical singularities} (\emph{slc family} for short).
\begin{dfn}[slc family]
	An \emph{slc family} is a flat proper morphism $f:X\to B$ such that:
	\begin{thmlist}
		\item 
		each fiber $X_b:=f^{-1}(b)$ is a projective variety with slc singularities.
		\item  $\omega_{X/B}^{[m]}$
		is flat.
		\item \label{kollar condition} The family $f:X\to B$ satisfies the \emph{Koll\'ar condition}, which means that, for any $m\in \mathbb{N}$, the reflexive power $\omega_{X/B}^{[m]}$ commutes with
		arbitrary base change.  
	\end{thmlist}
\end{dfn}
To make \cref{kollar condition} precise, for every base change $\tau:B'\to B$, given the induced
morphism $\rho:X'=X\times_{B}B'\to X$ we have that the natural homomorphism
$\rho^*\omega_{X/B}^{[m]}\to \omega_{X'/B'}^{[m]}$
is an isomorphism. 
Let us collect the  basic properties of slc families, as is well-known to the experts.
\begin{lem}
	Let    $g:Z\to W$ be a surjective morphism between quasi-projective manifolds with connected fibers, which is  birational to an slc family $g':Z'\to W$  whose generic fiber has at most Gorenstein canonical singularities. Then 
	\begin{thmlist}		
		\item  \label{normal canonical} the total space $Z'$ is normal and has only canonical singularities at worst.
		\item \label{KSBA birational}If  $\nu:W'\to W$  	is a dominant morphism with $W'$ smooth quasi-projective, then  	$Z'\times_{W}W'\to W'$ is still an slc family whose generic fiber has at most Gorenstein canonical singularities, and is birational to  
		$(Z\times_{W}W')^{\tsim}\to W'$.
		\item \label{fiber kollar}  Denote by $Z'^r$ the $r$-fold fiber  product $Z'\times_{W}\cdots\times_{W}Z'$.   	Then $g'^r:Z'^r\to W$ is also  an slc family    whose generic fiber has at most Gorenstein canonical singularities.  Moreover, $Z'^r$ is birational to the main component $(Z^r)^{\tsim}$ of $Z^r$ dominating $W$.
		\item  \label{KSBA fiber} Let  $Z^{(r)}$
		be a desingularization of   $(Z^r)^{\tsim}$. Then  $(g^{(r)})_*(\ell K_{Z^{(r)}/W})\simeq (g'^r)_*(\ell K_{Z'^{r}/W})$ is reflexive for every sufficiently divisible  $\ell>0$.
	\end{thmlist}  
\end{lem} 
\subsection{Positivity of direct images} \label{sec:positivity}
This section is devoted to prove \cref{thm:globalgen} on positivity of direct images, which refines results by Viehweg-Zuo \cite[Proposition 3.4]{VZ02} and \cite[Proposition 4.3]{VZ03}. It will be crucially used to proved \cref{main}.
\begin{thm} \label{thm:globalgen}
	Let \(f_0:X_0\rightarrow Y_0\) be a smooth family of projective manifolds of general type. 
	Assume that   for any \(y\in Y_0\),   the set of \(z\in Y_0\) with \( X_z \stackrel{{\rm bir}}{\sim} X_y \) is finite. 
	\begin{thmlist}
		\item \label{global generated} For any   smooth projective compactification $f:X\to Y$ of \(f_0:X_0\rightarrow Y_0\)  and any sufficiently  ample line bundle \(A_Y\) over \(Y\),    \( f_*(\ell K_{X/Y})^{\star\star}  \otimes A_Y^{-1}\) is globally generated over \( Y_0 \) for any \(\ell\gg 0\).  In particular, \(f_*(\ell K_{X/Y})  \) is ample with respect to \(Y_0\).
		\item \label{enough ample}   In the same setting as (\lowerromannumeral{1}),  \(\det f_*(\ell K_{X/Y})\otimes A_Y^{-r_\ell} \) is also globally generated over \(Y_0\) for any \(\ell\gg 0\), where \(r_{\ell}=\rank f_*(\ell K_{X/Y})\). In particular,    \(\mathbf{B}_+\big(\det f_*(\ell K_{X/Y}) \big)\subset Y\setminus Y_0\).
		\item  \label{desired global} 
		For some $r\gg 0$, there exists an algebraic fiber space   $f:X\to Y$ compactifying $X_0^r\to Y_0$,  so that 
		$f_*(\ell K_{{X}/Y})\otimes A_Y^{-\ell} $ is globally generated over $Y_0$ for $\ell$ large and divisible enough.   Here $X_0^r$ denotes to be the $r$-fold fiber product of $X_0\to Y_0$,  and  $A_Y$ is some sufficiently ample line bundle over $Y$. 
	\end{thmlist}   
\end{thm}
\begin{proof}
	Let us first show that, to prove Claims (\lowerromannumeral{1}) and (\lowerromannumeral{2}), one can assume that both $B:=Y\setminus Y_0$ and $f^*B$ are normal crossing.  
	
	For the arbitrary smooth projective compactification $f':X'\to Y'$ of $f_0:X_0\to Y_0$, we take a log resolution $\nu:Y\to Y'$ with centers supported on $Y'\setminus Y_0$ so that $B:=\nu^{-1}(Y'\setminus Y_0)$ is a simple normal crossing divisor.  Define   $X$ to be strong desingularization  of the main component $(X'\times_{Y'}Y)^{\tsim}$ dominant over $Y$  
	\begin{align} 
	\xymatrix{
		X  \ar[r]\ar[rd]_-{f}&	X'\times_{Y'}Y\ar[d]   \ar[r] & X'\ar[d]^-{f'}  \\
		&	Y  \ar[r]^{\nu} & Y'
	}
	\end{align}
	so that $f^*B$ is normal crossing. By \cite[Lemma 2.5.a]{Vie90}, there is the inclusion
	\begin{align}\label{eq:inclusion Vie}
	\nu_{*}f_*(mK_{X/Y})\hookrightarrow f'_*(mK_{X'/Y'})
	\end{align} 
	which is an isomorphism over $Y_0$ for each $m\in \mathbb{N}$. Hence for any    ample line bundle $A$ over $Y'$, once  \( f_*(m K_{X/Y})^{\star\star}  \otimes (\nu^*A)^{-1}\) is globally generated over \( \nu^{-1}(Y_0)\simeq Y_0 \) for some $m\geqslant 0$,    $f'_*(mK_{X'/Y'})^{\star\star}  \otimes A^{-1}$ will be also  globally generated over \(Y_0 \). As we will see, Claim (\lowerromannumeral{2}) is a direct consequence of Claim (\lowerromannumeral{1}). This  proves the above statement.

	\begin{enumerate}[leftmargin=0cm,itemindent=0.7cm,labelwidth=\itemindent,labelsep=0cm, align=left,label= {\rm (\roman*)},itemsep=0.07cm]
		\item Let us fix a sufficiently ample line bundle $A_Y$ on $Y$. Assume that  both $B:=Y\setminus Y_0$ and $f^*B$ are normal crossing.  	It follows from \cite[Theorem 5.2]{Vie90} that one can take some \(b\gg a\gg 0 \),  \(\mu\gg m\gg 0 \) and \(s\gg 0 \) such that \(\ls:=   \det f_*(\mu mK_{X/Y})^{\otimes a}\otimes \det f_*(mK_{X/Y})^{\otimes b} \) is ample over \(Y_0\).  In other words, \(\mathbf{B}_+(\ls)\subset {\rm Supp}(B) \). By the definition of   augmented base locus,  one can even arrange \(a, b\gg 0 \) such that  there exists a singular hermitian metric \( h_1\) of \(\ls-4A_Y \)  which is smooth over \(Y_0\), and the curvature current \(\sqrt{-1}\Theta_{h_\ls}(\ls)\geqslant \omega \) for some K\"ahler form \(\omega \) in \(Y\).   
		Denote by \(r_1:=\rank f_*(\mu mK_{X/Y})\)	and  \(r_2:=\rank f_*(mK_{X/Y})\). It follows from \cref{CP} that  for any sections
		\[
		\sigma_1 \in H^0(X_y,  a\mu mC K_{X_y} ), \quad \sigma_2 \in H^0(X_y, b mC K_{X_y} ),
		\]
		there exists  effective divisors  \(\Sigma_1 \) and \(\Sigma_2 \) such that
		\begin{eqnarray*}
			\Sigma_1+a f^*\det f_*(m\mu K_{X/Y})- f^*A_Y \stackrel{{\rm linear}}{\sim} am\mu r_1CK_{X/Y}+ P_1+F_1 \\
			\Sigma_2+b f^*\det f_*(m K_{X/Y})- f^*A_Y \stackrel{{\rm linear}}{\sim} bmr_2C K_{X/Y}+ P_2+F_2 
		\end{eqnarray*}
		and 
		\[
		\Sigma_{1\upharpoonright X_y}=\sigma_1^{\otimes r_1}, \quad \Sigma_{2\upharpoonright X_y}=\sigma_2^{\otimes r_2}.
		\]
		Here $F_i$ is $f$-exceptional with \(f(F_i)\subset {\rm Supp}(B)\), \( {\rm Supp}(P_i)\subset {\rm Supp}(\Delta_f)\)  for \(i=1,2\).

		Write \(N:=am\mu r_1C+bmr_2C \), \(P:=P_1+P_2\) and \(F:= F_1+F_2\). Fix any $y\in Y_0$.  Then  the effective divisor \(\Sigma_1 +\Sigma_2 \) induces a singular hermitian metric  
		\(h_2\)  for the line bundle \(L_2:=NK_{X/Y}-f^*\ls+2f^*A_Y+P+F\) such that \(h|_{X_y} \) is not identically equal to \(+\infty \), and so is the singular hermitian metric \(h:=f^*h_1\cdot h_2 \) over \(L_0:=L_2+f^*\ls-4f^*A_Y=NK_{X/Y} -2f^*A_Y+P+F \).  
		In particular, when \(\ell \) sufficiently large,  the multiplier ideal sheaf  \( \js (h_{\upharpoonright X_y}^{\frac{1}{\ell}})=\oc_{X_y}\). By Siu's invariance of plurigenera, all the global sections \(H^0\big(X_y,(\ell K_X+L_0)_{\upharpoonright X_y}\big)\simeq H^0\big(X_y, (\ell+N)K_{X_y}\big) \) extends locally, and we thus can apply \cref{extension} to obtain the desired surjectivity 
		\begin{align}\label{surjectivity}
		H^0  \big(X, \ell K_{X/Y} + L_0 -\ell\Delta_f +f^* A_Y\big) \twoheadrightarrow  H^0 \big(X_y ,  (\ell+N) K_{X_y}  \big), 
		\end{align}
		Recall that \({\rm Supp}(P)\subset {\rm Supp}(\Delta_f)\). Then $\ell f^*B\geqslant P$ for $\ell\gg 0$, and
		one has the inclusion of sheaves 
		\[
		\ell K_{X/Y} + L_0 -\ell \Delta_f +f^* A_Y\hookrightarrow (N+\ell)   K_{X/Y}   -f^* A_Y +F.        
		\] 
		which is an isomorphism over $X_0$  . 	By \eqref{surjectivity} this implies that  the direct image sheaves \( f_*(\ell K_{X/Y}  -f^* A_Y+F)\) are globally generated over some Zariski open set \(U_y\subset Y_0\) containing \(y\) for $\ell \gg 0$.  
		Since \(y\) is an arbitrary point in \(Y_0\), the direct image  \( f_*(\ell K_{X/Y}+F ) \otimes A_Y^{-1}\) is globally generated   over \(Y_0\) for $\ell \gg 0$  by noetherianity.    Recall that $F$ is $f$-exceptional with $f(F)\subset {\rm Supp}(B)$. Then there is an injection
		\[
		f_*(\ell K_{X/Y}+F ) \otimes A_Y^{-1}\hookrightarrow  f_*(\ell K_{X/Y})^{\star\star} \otimes A_Y^{-1}
		\]  
		which is an isomorphism over $Y_0$. Hence \( f_*(\ell K_{X/Y} )^{\star\star} \otimes A_Y^{-1}\) is also globally generated over  \(Y_0\).  
	Hence \(f_*(\ell K_{X/Y} ) \) is ample with respect to \(Y_0\) for $\ell\gg 0$. 	The first claim follows.
		
		\item The trick to prove the second claim has already appeared  in \cite{Den17} in proving a conjecture by Demailly-Peternell-Schneider. We first recall that \(f_*(\ell K_{X/Y} ) \) is locally free outside a codimension 2 analytic subset of $Y$. 	By the proof of \cref{global generated},  for \(\ell\) sufficiently large and divisible,  \(f_*(\ell K_{X/Y} +F)\otimes A_Y^{-1}\) is  locally free and  generated by global sections over \(Y_0\), where $F$ is some $f$-exceptional effective divisor. Therefore,   its determinant \(\det f_*(\ell K_{X/Y}+F )\otimes A_Y^{-r_\ell}\) is also globally generated  over \(Y_0\), where \(r_\ell:=\rank  f_*(\ell K_{X/Y}) \).   Since $F$ is $f$-exceptional and effective, one has
		$$
		\det f_*(\ell K_{X/Y}+F )\otimes A_Y^{-r_\ell}=\det f_*(\ell K_{X/Y} )\otimes A_Y^{-r_\ell},
		$$
		and therefore, \(\det f_*(\ell K_{X/Y})\otimes A_Y^{-r_\ell}\) is also globally generated  over \(Y_0\).	By the very definition of the augmented base locus $\mathbf{B}_+(\bullet)$ we conclude that
		\[\mathbf{B}_+\big(\det  f_*(\ell K_{X/Y})\big)\subset {\rm Supp}(B) .\] 
		The second claim is proved.

		%
		
		\item   	 
		We combine the ideas in \cite[Proposition 4.1]{VZ03} as well as   the pluricanonical extension techniques in  \cref{extension} to prove the result.  
		By \cref{def: good compactification},   there exists a  smooth projective compactification $Y$ of $Y_0$ with  $B:=Y\setminus Y_0$ simple normal crossing,  a    non-singular finite covering $\psi:W\to Y$, and an slc family $g':Z'\to W$,  which extends the family $X_0\times_{Y_0}W$. 
		By \cref{fiber kollar}  for any $r\in \mathbb{Z}_{>0}$,  the $r$-fold fiber  product $g'^r:Z'^r \to W$ is still an slc family,  which compactifies the smooth family $X_0^r\times_{Y_0}W\to W_0$, where $W_0:=\psi^{-1}(Y_0)$.  Note that $Z'^r$ has canonical singularities. 
		
		Take a smooth projective compactification $f:{X}\to Y$ of $X_0^r\to {Y_0}$ so that $f^*B$ is normal crossing. Let  ${Z}\to Z'^r$ be a strong desingularization of $Z'^r$, which also resolves this birational map $Z'^r\dashrightarrow (X\times_{Y} W)^{\tsim}$.   Then $ g:{Z}\to W$ is smooth over $W_0:=\psi^{-1}(Y_0)$.    
		\begin{equation*}
		\begin{tikzcd}
		{Z} \arrow[r]\arrow[rrr, bend left=20]\arrow[d, " {g}"]  &	Z'^r \arrow[r,dashrightarrow]   \arrow[d, "g'^r"]  &    X \arrow[d, "f"] &  (X\times_{Y} W)^{\tsim} \arrow[l]\arrow[d]\\
		W   \arrow[r, equal]  &	W  \arrow[r, "\psi"]     & Y &  W \arrow[l, "\psi"']   
		\end{tikzcd}
		\end{equation*}

		Let  $\tilde{Z}$ be a strong desingularization of $Z'$, which is thus smooth over $W_0:=\psi^{-1}(Y_0)$.      
		For the new family $\tilde{g}:\tilde{Z}\to W$, we denote by     $\tilde{Z}_0:=\tilde{g}^{-1}(W_0)$.  Then $\tilde{Z}_0\to W_0$ is also a smooth family, and any fiber of $Z_w$ with $w\in W_0$ is a projective manifold of general type.   By our assumption in the theorem,  for any \(w\in W_0\),    the set of \(w'\in W_0 \) with \( \tilde{Z}_{w'} \stackrel{{\rm bir}}{\sim} \tilde{Z}_{w} \) is finite as $\psi:W\to Y$ is a finite morphism.  We thus can apply \cref{global generated,enough ample} to our new family $\tilde{g}:\tilde{Z}\to W$. 

		From now on, we will always assume that $\ell\gg 0$ is sufficiently divisible so that $\ell K_{Z'}$ is Cartier.   
		Let $A_Y$ be a sufficiently ample line bundle over $Y$, so that $A_W:=\psi^*A_Y$ is also \emph{sufficiently  ample}. Since $Z'$ has canonical singularity,   $\tilde{g}_*(\ell K_{\tilde{Z}/W})=g'_*(\ell K_{Z'/W})$. It follows from  \cref{enough ample} that,  for any $\ell \gg 0$,  
		the line bundle  
		\begin{eqnarray} \label{det ample}
		\det  \tilde{g}_*(\ell K_{\tilde{Z}/W} )\otimes A_W^{-r}=\det g'_*(\ell K_{Z'/W})\otimes A_W^{-r}
		\end{eqnarray}
		is   globally generated over $W_0$, where $r:=\rank\   g'_*(\ell K_{Z'/W})$ depending  on $\ell$. Then  there exists a positively-curved singular hermitian metric $h_{\det}$ on the line bundle $\det  g'_*(\ell K_{Z'/W})\otimes A_W^{-r} $ such that $h_{\det}$ is smooth over $W_0$. 
		
		By the base change properties of slc families (see    \cite[Proposition 2.12]{BHPS13} and \cite[Lemma 2.6]{KP17}), one has
		\begin{align*}
		\omega^{[\ell]}_{Z'^r/W} \simeq \bigotimes_{i=1}^r{\rm pr}_i^*\omega^{[\ell]}_{Z'/W},\quad 
		g'^r_*(\ell K_{Z'^r/W}) \simeq \bigotimes^r g'_*(\ell K_{Z'/W}), 
		\end{align*}
		where ${\rm pr}_i:Z'^r\to Z'$ is the $i$-th directional projection map. Hence $\ell K_{Z'^r}$ is Cartier as well, and  we have 
		$$
		\bigotimes^r g'_*(\ell K_{Z'/W})\simeq g'^r_*(\ell K_{Z'^r/W})=  g_*(\ell K_{Z/W}).
		$$
		By \cref{KSBA fiber}, $g_* (\ell K_{Z /W})$ is reflexive, and we thus have
		\begin{align*} 
		\det g'_*(\ell K_{Z'/W}) \to  \bigotimes^r g'_*(\ell K_{Z'/W})   \simeq g_* (\ell K_{Z/W}),
		\end{align*}
		which induces a natural effective divisor 
		\begin{align*} 
		\Gamma   \in  \lvert \ell K_{Z/W}  - g^* \det g'_*(\ell K_{Z'/W}) \rvert
		\end{align*}  
		such that $\Gamma_{ \upharpoonright Z_{w}}\neq 0$ for any  (smooth) fiber $Z_w$ with  \(w\in W_0\).   	By \cref{control}   for any $w\in W_0$ the log canonical threshold
		\begin{align}\label{eq:lct control}
		c(\Gamma_{\upharpoonright Z_w})\geqslant \frac{1}{e(\ell K_{Z_w} )}=\frac{1}{e\big(  \bigotimes_{i=1}^{r}{\rm pr}^*_i(K_{Z'_w}^{\otimes \ell} )\big)}=\frac{1}{e\big(  \ell K_{Z'_w} \big)}\geqslant \frac{1}{\ell\cdot c_1(\as)^{d-1}\cdot c_1(K_{Z'_w})+1}\geqslant \frac{2 }{(C-1)\ell}
		\end{align}
		for some  \(C\in \mathbb{N} \) which does not depend on \(\ell\) and $w\in W_0$.     Denote by  $h$ the singular hermitian metric on 
		\[\ell K_{Z/W}  - g^* \det g'_*(\ell K_{Z'/W})  \] induced by $\Gamma $. By the definition of log canonical threshold, the multiplier ideal sheaf  $\js\big(h^{\frac{1}{(C-1)\ell}}_{ \upharpoonright Z_w} \big)=\oc_{Z_w}$ for  any  fiber $Z_w$ with  \(w\in W_0\).   	Let us define a positively-curved singular metric   $h_{	\mathscr{F}}$ for    the line bundle
		$  
		\mathscr{F}:= \ell   K_{Z/W}   - r g^* A_W 
		$
		by setting $h_{	\mathscr{F}}:=h\cdot g^*h_{\det}$.   Then  $\js\big(h^{\frac{1}{(C-1)\ell}}_{\fs \upharpoonright Z_w} \big)=\oc_{Z_w}$ for any $w\in W_0$.

		For any $n\in \mathbb{N}^*$, applying \cref{extension} to $n  \fs$ we obtain the surjectivity
		\begin{align}\label{3:direct image}
		H^0\big(Z, (C-1)  n\ell K_{Z/W}   +n \fs+ g^*  A_W\big)\twoheadrightarrow  H^0\big(Z_w,   Cn\ell  K_{Z_w}   \big)
		\end{align}
		for all $w\in W_0$. In other words,
		\begin{eqnarray*} 
			g_*\big(C\ell n K_{Z/W}    )\otimes   A_W^{-(n  r-1)} 
		\end{eqnarray*}
		is globally generated over $W_0$ for any $\ell \gg 0$ and any $n\geqslant 1$.  
		
		Since \(K_{X_y}\) is big, one thus has
		\[
		r=r_\ell \sim \ell^d\ \  {\rm as\ } \ \ell \to +\infty
		\]
		where \(d:=\dim Z_w\geqslant 2 \) (if the fibers of $f$ are curves, one can take a fiber product to replace the original family).   Recall that $C$ is a constant which does not depend on $\ell$.  One thus can take  an \emph{a priori}  $\ell\gg 0 $ so that  $r\gg C\ell$. 	   In conclusion,  for sufficiently large and divisible $m$, 
		\begin{eqnarray*} 
			g_*\big(m K_{Z/W}    )\otimes   A_W^{-2m}=	g_*\big(m K_{Z/W}    )\otimes   \psi^*A_Y^{-2m} 
		\end{eqnarray*}
		is globally generated over $W_0$.   
		Therefore, we have a    morphism
		\begin{align}\label{eq: global}
		\bigoplus_{i=1}^{N}\psi^*A_Y^{m}\to  g_*\big(m K_{Z/W} \big)\otimes   \psi^*A_Y^{-m},
		\end{align}
		which is surjective over $W_0$. On the other hand, by \cite[Lemma 2.5.b]{Vie90}, one has the inclusion
		$$
		g_*\big(m K_{Z/W} \big)     \hookrightarrow   \psi^*f_*(mK_{X/Y}), 
		$$
		which is an isomorphism over $W_0$. 
		\eqref{eq: global} thus induces a morphism 
		\begin{align} \label{eq:surjective}
		\bigoplus_{i=1}^{N}\psi_*\oc_W\otimes A_Y^{m}\to 	  \psi_*g_*\big(m K_{Z/W} \big)\otimes    A_Y^{-m}\to   
		\psi_* \psi^*\big(f_*(mK_{X/Y})\big)\otimes   A_Y^{-m},
		\end{align}
		which is surjective over $Y_0$. Note that 
		that even if $f_*(mK_{X/Y})$  is merely a coherent sheaf, the projection formula $\psi_* \psi^*\big(f_*(mK_{X/Y})\big)=f_*(mK_{X/Y})\big)\otimes    \psi_*\oc_W$
		still holds  for  $\psi$ is finite  (see   \cite[Lemma 5.7]{Ara04}).  
		The trace map
		$$
		\psi_*\oc_W\rightarrow \oc_Y
		$$
		splits the natural inclusion $\oc_Y\to \psi_*\oc_W$, and is thus surjective. Hence \eqref{eq:surjective} gives rise to a morphism \begin{align} \label{eq:surjective2}
		\bigoplus_{i=1}^{N}\psi_*\oc_W\otimes A_Y^{m}\to  \psi_*g_*\big(m K_{Z/W} \big)\otimes    A_Y^{-m}\xrightarrow{\Phi}   	 f_*(mK_{X/Y}) \otimes   A_Y^{-m},
		\end{align}
		which is surjective over $Y_0$. 
		By taking $m$ sufficiently large, we may assume that $\psi_*\oc_W\otimes A_Y^{m}$ is   generated by its global sections. 
		Then $f_*(mK_{X/Y}) \otimes A_Y^{-m}$ is globally generated over $Y_0$.   We complete the proof.\qedhere 
	\end{enumerate}
\end{proof} 
Let us prove  the following Bertini-type result, which will be used in  the proof of \cref{cor:main hyper}.
\begin{lem}[A Bertini-type result]\label{Bertini}
Let  $f:X\to Y$  be   the projective family in  \cref{desired global}. Then for any given smooth fiber $X_y$ with $y\in Y_0$,  there is $H\in |\ell K_{X/Y}-\ell f^*A_Y|$ so that $H_{\upharpoonright X_y}$ is smooth. In particular, there is a Zariski open neighborhood $V_0\supset y$ so that $H$ is smooth over $V_0$.  
\end{lem}
\begin{proof}
By Siu's invariance of plurigenera and Grauert-Grothedieck's  \enquote*{cohomology and base change}, we know that 
$f_*(\ell K_{X/Y})\otimes A_Y^{-\ell}$ is locally free on $Y_0$, and the natural map
$$
(f_*(\ell K_{X/Y})\otimes A_Y^{-\ell})_y\to H^0(X_y,\ell K_{X_y})
$$ 
is an isomorphism for any $y\in Y_0$. Since $K_{X_y}$ is assumed to be semi-ample, one can take $\ell\gg 0$ so that $|\ell K_{X_y}|$ is base point free. By the Bertini theorem, one can take a section $s\in H^0(X_y,\ell K_{X_y})$ whose zero locus is a smooth hypersurface on $X_y$. By \cref{desired global}, one has the surjection
$$
H^0\big(Y,f_*(\ell K_{X/Y})\otimes A_Y^{-\ell}\big)\twoheadrightarrow \big(f_*(\ell K_{X/Y})\otimes A_Y^{-\ell}\big)_y\xrightarrow{\simeq} H^0(X_y,\ell K_{X_y}).
$$
Hence there is 
$$\sigma\in H^0(X,\ell K_{X/Y}-\ell f^*A_Y)=H^0\big(Y,f_*(\ell K_{X/Y})\otimes A_Y^{-\ell}\big)$$
which extends the section $s$. In other words, for the zero divisor $H=(\sigma=0)$,  its restriction to $X_y$ is smooth. Hence  there is a Zariski open neighborhood $V_0\supset y$ so that $H$ is smooth over $V_0$.  The lemma is proved.
\end{proof} 
\begin{rem}
Note that  we do not know how to find a hypersurface $H\in |\ell K_{X/Y}-\ell f^*A_Y|$ so that its discriminant locus in $Y$ is normal crossing. We have to blow-up the base $\nu:Y'\to Y$ to achieve this. As $f:X\to Y$ is not flat in general, for the new family 
$$
\begin{tikzcd}
X'\arrow[r]\arrow[d,"f'"]&(X\times_YY')^{\tsim}\arrow[d]\arrow[r]&X\times_YY'\arrow[r]\arrow[d] &X\arrow[d, "f"]\\
Y'\arrow[r, equal] &Y'\arrow[r, equal] &Y'\arrow[r, "\nu"]& Y
\end{tikzcd}
$$
where $X'$ is a desingularization of $(X\times_YY')^{\tsim}$, in general  
$$
\nu^*f_*(\ell K_{X/Y})\not\subset f'_*(\ell K_{X'/Y'}).
$$
In other words,  although the discriminant locus of $\nu^*H\to Y'$ is a simple normal crossing divisor in $Y'$,  $\nu^*H$ might not lie at $|\ell K_{X'/Y'}-\ell f'^*\nu^*A_Y|$. We will overcome this problem in \cref{cor:main hyper} at the cost of the appearance of some $f'$-exceptional divisors. 
\end{rem}

Since the \emph{$\mathbb{Q}$-mild reduction} in \cref{def: good compactification} holds for any smooth surjective projective morphism with connected fibers and smooth base, it follows from our proof in \cref{desired global}  and Kawamata's theorem \cite{Kaw85},  one still has the \emph{generic} global generation as follows.
\begin{thm}\label{thm:Kawamata}
	Let $f_U:U\to V$ be a smooth projective morphism  between quasi-projective varieties  with connected fibers. Assume that the general fiber $F$ of $f_U$  has semi-ample canonical bundle, and $f_U$ is of maximal variation. Then there exists a positive integer $r\gg 0$ and a smooth projective compactification  $f:{X}\to Y$ of $U^r\to V$ so that ${f}_*(m K_{{X}/Y})\otimes \as^{-m}$ is globally generated over some Zariski open subset of $V$. Here $U^r\to V$ is the $r$-fold fiber product of $U\to V$, and $\as$ is some ample line bundle on $Y$. \qed
\end{thm}

\subsection{Sufficiently many  \enquote*{moving}  hypersurfaces}\label{sec:good}
As we have seen in \cref{sec:existence} on the construction of VZ Higgs bundles,  one has to apply branch cover trick to construct a negatively twisted Hodge bundle on the compactification of the base, which is well-defined outside a simple normal crossing divisor. This means that the hypersurface $H\in |\ell K_{X/Y}-\ell f^*A_Y|$  in constructing the cyclic cover is smooth over the complement of an SNC divisor of the base.  As we discussed in \cref{Bertini}, in general we cannot perform a simple blow-up of the base to achieve this. 
In this subsection we will overcome this difficulty in applying the methods in  \cite[Proposition 4.4]{PTW18}. It will be our basic setup in constructing refined VZ Higgs bundles in \cref{construction}.  

\begin{thm}\label{cor:main hyper}
	Let $X_0\to Y_0$ be a smooth family of   minimal projective manifolds of
	general type over a quasi-projective manifold $Y_0$. Suppose that  for any \(y\in Y_0\),  the set of \(z\in Y_0\) with \( X_z \stackrel{{\rm bir}}{\sim} X_y \) is finite.  Let $Y\supset Y_0$ be the smooth compactification   in \cref{def: good compactification}.      Fix  any  $y_0\in Y_0$ and some sufficiently ample line bundle $A_Y$ on $Y$.   Then there exist a   birational morphism $\nu:Y'\to Y$  and a new algebraic fiber space $f':X'\to Y'$ which is smooth over  $\nu^{-1}(Y_0)$, 
	so that  for any sufficiently large and divisible  $\ell$, one can find a hypersurface 
	\begin{align}\label{eq:same hypersurface}
	H\in |\ell K_{X'/Y'}-\ell (\nu\circ f')^*A_Y+\ell E|
	\end{align}
	satisfying that
	\begin{itemize}[leftmargin=0.6cm] 
		\item the divisor $D:=\nu^{-1}(Y\setminus Y_0)$ is   simple normal crossing. 
		\item There exists a reduced divisor $S$ in $ Y'$, so that $D+S$ is simple normal crossing,    
		and    $H\to Y'$ is smooth over $Y'\setminus D\cup S$. 
		\item   The exceptional locus ${\rm Ex}(\nu)\subset {\rm Supp}(D+S)$, and   $y_0\notin\nu (D\cup S )$. 
		\item The divisor $E$ is effective and  $f'$-exceptional with $f'(E)\subset  {\rm Supp}(D+S)$.
	\end{itemize} 
	Moreover, when $X_0\to Y_0$ is \emph{effectively parametrized} over some open set containing $y_0$, so is the new family $X'\to Y'$.  
\end{thm}
\begin{proof}
	The proof is a continuation of that of \cref{desired global}, and we adopt the same notations therein. By \eqref{eq:surjective2} and the isomorphism $$H^0(Z, \ell K_{Z/W}-\ell g^*A_W)\simeq H^0\big(Z'^r,\ell K_{Z'^r/W}-\ell (g'^r)^*A_W\big),$$   the morphism  $\Phi:\psi_*g_*\big(\ell K_{Z/W} \big)\otimes    A_Y^{-\ell}\to    	 f_*(\ell K_{X/Y}) \otimes   A_Y^{-\ell}$ in \eqref{eq:surjective2} gives rise to a natural map
	\begin{align}\label{eq:upsilon}
	\Upsilon:H^0(Z'^r,\ell  K_{Z'^r/W}-\ell (g'^r)^*A_W)\to H^0\big(Y,f_*(\ell K_{X/Y}) \otimes   A_Y^{-\ell}\big)
	\end{align}
	whose image $I$ generates $f_*(\ell K_{X/Y}) \otimes   A_Y^{-\ell }$ over $Y_0$. Note that $\Upsilon$ is fonctorial in the sense that it does not depend on the choice of the birational model $Z\to Z'^r$. By the base point free theorem, for any $y\in Y_0$, $K_{X_y}$ is semi-ample, and we can assume that $\ell \gg 0$ is sufficiently large and divisible so that $\ell K_{X/Y}$ is relatively semi-ample over $Y_0$. Hence we can take a section 
	\begin{align}\label{eq:sigma}
	\sigma\in  H^0\big(Z'^r,\ell K_{Z'^r/W}-\ell (g'^r)^*A_W\big)
	\end{align}
	so that 
	the  zero divisor  of $$\Upsilon(\sigma)\in H^0\big(X, \ell K_{X/Y}-\ell f^*A_Y\big)= H^0\big(Y,f_*(\ell K_{X/Y}) \otimes   A_Y^{-\ell}\big),$$ denoted by 
	$
	H_1\in |\ell K_{X^{(r)}/Y}-\ell (f^{(r)})^* A_Y  |,
	$     is \emph{transverse} to the fiber $X_{y_0}$.  Denote by $T$ the \emph{discriminant locus} of $H_1\to Y$, and $B:=Y\setminus Y_0$. Then $y_0\notin T\cup B$.  Take a log-resolution $\nu:Y'\to Y$ with centers in $T\cup B$ so that  both $D:=\nu^{-1}(B)$ and $D+S:=\nu^{-1}(T\cup B)$ are simple normal crossing. Let  $X'$ be a strong desingularization of $(X\times_YY')^{\tsim}$, and write $f':X'\to Y'$, which is smooth over $Y_0':=\nu^{-1}(Y_0)$. Set $X_0':=f'^{-1}(Y_0')$.   It suffices to show that, there exists a hypersurface $H$ in  \eqref{eq:same hypersurface}  with  $H_{\upharpoonright (\nu\circ f')^{-1}(V)}=H_{1\upharpoonright ( f^{(r)})^{-1}(V)}$, where $V:=Y\setminus  S'\cup B \subset Y_0$. Since the birational morphism $\nu$ is isomorphic at $y_0$, we can write $y_0$ as $\nu^{-1}(y_0)$ abusively.	 
	
	Now we follow the similar  arguments in \cite[Proposition 4.4]{PTW18} to prove the existence of $H$ (in which they apply their methods for \emph{mild morphisms}).  Let $W'$   be a strong desingularization of $W\times_YY'$ which is finite at $y_0\in Y'$.    Write $W'_0:=\nu'^{-1}(W_0)$. By \cref{KSBA birational}, the new family  $ Z'':=Z'^r\times_WW'\to W'$  is  still an slc family,  which  compactifies the smooth family $X_0'\times_{Y'_0}W'\to W_0'$. Let $M'$ be a desingularization of $Z''$ so that it resolves the rational maps  to $X'$ as well as $Z$.
	\begin{eqnarray*} 
		\xymatrix@!0{
			& X    \ar[ddd]|{f } |!{[dd];[rrr]}\hole
			& && Z \ar[rrr] \ar[lll] \ar[ddd]|{g}|!{[dd];[rrr]}\hole
			& && Z'^r\ar[ddd]|{\ g'^r}
			\\
			&&&&\\
			X'\ar|{\!\! \mu}[uur]   \ar[ddd]|{f'}
			& & &   M'\ar[lll]\ar[uur]\ar[ddd]|{h'} \ar[rrr]  
			& && Z''\ar[ddd]|{g''}\ar|{\  \mu'}[uur]
			\\
			& Y  
			& && W   \ar[lll]|!{[lu];[ddl]}\hole|{\ \psi}     \ar@{=}[rrr]|!{[rru];[ddrr]}\hole
			& && W
			\\
			&&&&\\
			Y' \ar[uur]|{\ \nu}
			& &  & W'\ar[lll]|{\ \psi'}       \ar[ruu]|{\ \nu'}\ar@{=}[rrr] &&& W'\ar[uur]|{\ \ \nu'}
		}
	\end{eqnarray*}


	By the properties of  slc families,   $\mu'^*\omega_{Z'^r/W}^{[\ell]}=\omega_{Z''/W'}^{[\ell]}$, which induces a natural map 
	\begin{align}\label{eq:pull}
	\mu^*:	H^0\big(Z'^r,\ell K_{Z'^r/W}-\ell (g'^r)^*A_W\big)\to H^0\big(Z'',\ell K_{Z''/W'}-\ell (\nu'\circ g'')^*A_W\big).
	\end{align}
	Since  both $Z'^r$ and $Z''$  have canonical singularities,   one has  the following natural morphisms
	\begin{align*}	
	g_*(\ell K_{Z/W}) \simeq (g'^r)_*(\ell K_{Z'^r/W} ), \quad
	h'_*(\ell K_{M'/W'})=g''_*(\ell K_{Z''/W'}).
	\end{align*} 
	We can leave out a subvariety  of codimension at least two in  $Y'$ supported on  $D+S$ (which thus avoids $y_0$ by our construction) so that $\psi':W'\to Y'$ becomes a \emph{flat finite} morphism. As discussed at the beginning of the proof, there is also a natural map
	\begin{align}\label{eq:upsilon2}
	\Upsilon':H^0(Z'',\ell  K_{Z''/W'}-\ell (\nu'\circ g'')^*A_W)\to H^0\big(X', \ell K_{X'/Y'}  -\ell ( \nu\circ f')^*A_Y\big)
	\end{align}
	as \eqref{eq:upsilon} by factorizing through $M'$.
	
	Note that  for $V:=Y\setminus  T\cup B $, $\nu:\nu^{-1}(V)\xrightarrow{\simeq} V$ is also an isomorphism, and thus the restriction of $X\to Y$ to $V$ is isomorphic to  that of $X'\to Y'$ to $\nu^{-1}(V)$.    Hence by our construction,
	the restriction of 	$Z'^r\to W$ to $\psi^{-1}(V)$ is isomorphic to that of $Z''\to W'$ to $(\nu\circ \psi')^{-1}(V)=(\nu'\circ \psi)^{-1}(V)$.
	In particular,  under the above isomorphism, 	for  the section $\sigma\in H^0\big(Z'^r, \ell K_{Z'^r/W}-\ell (g'^r)^*A_W\big)$ in \eqref{eq:sigma} with $\Upsilon(\sigma)$ defining  $H_1$, one has
	\begin{align*} 
	\Upsilon(\sigma)_{\upharpoonright f^{-1}(V)}\simeq  \Upsilon'(\mu^*\sigma)_{\upharpoonright (\nu\circ f')^{-1}(V)}.
	\end{align*} 
	where  $\mu^*$ and $\Upsilon'$ are defined in \eqref{eq:pull} and \eqref{eq:upsilon2}. Denote by $\tilde{H}$    the zero divisor  defined by  $$\Upsilon'(\mu^*\sigma)\in H^0\big(X', \ell K_{X'/Y'}  -\ell  ( \nu\circ f')^*A_Y\big).$$
	Recall that $H_1$ is smooth over $V$, then  $\tilde{H}$ is also smooth over $\nu^{-1}(V)$. 
	Note that $\Upsilon'(\mu^*\sigma)\in H^0\big(Y', f'_*(\ell K_{X'/Y'})\otimes \nu^*A_Y^{-\ell}\big)$  is only defined over a big open set of $Y'$ containing $\nu^{-1}(V)$. Hence it extends to a global section   
	$$s\in H^0(X', \ell K_{X'/Y'} -\ell (\nu\circ f')^*A_Y+\ell E),$$  
	where $E$ is an  $f'$-exceptional effective divisor with  $f'(E)\subset  {\rm Supp}(D+S)$.  Denote by $H$  the hypersurface in $X'$ defined by $s$. Hence $H_{\upharpoonright (\nu\circ f')^{-1}(V)}=\tilde{H}_{\upharpoonright (\nu\circ f')^{-1}(V)}$, which is smooth over $\nu^{-1}(V)=Y'\setminus D\cup S\simeq V\ni y_0$.  Note that the property of effective parametrization is invariant under fiber product.  The theorem follows.  
\end{proof}

\subsection{Kobayashi hyperbolicity of the moduli spaces}\label{construction} 
In this subsection, for effectively parametrized  smooth family of minimal projective manifolds of general type, we refine  the Viehweg-Zuo Higgs bundles in  \cref{thm:existence}   
so that we can apply  \cref{construction of Finsler} and the bimeromorphic criteria for Kobayashi hyperbolicity in \cref{bimeromorphic} to prove \cref{main}. 

\begin{thm}\label{Higgs bundle}
	Let $U\to V$ be an   effectively parametrized smooth family of minimal projective manifolds of general type over the quasi-projective manifold $V$.     
	Then for any  given point $y\in V$,  there exists a smooth projective compactification $Y$ for a birational model $\nu:\tilde{V}\to V$, and a VZ Higgs bundle $(\tilde{\es},\tilde{\theta})\supset (\fs,\eta)$  over $Y$ 
	satisfying the following properties:
	\begin{thmlist}
		\item \label{p1} there is a Zariski open set $V_0$ of $V$ containing $y$ so that  $\nu:\nu^{-1}(V_0)\to V_0$  is an isomorphism.
		\item	\label{p2} Both $D:=Y\setminus \tilde{V}$ and $D+S:=Y\setminus \nu^{-1}(V_0)$ are simple normal crossing divisors in $Y$.
		\item \label{p3} The Higgs bundle $(\tilde{E},\tilde{\theta})$ has log poles supported on $D\cup S$, that is, $\tilde{\theta}:\tilde{E}\to \tilde{E}\otimes \big(\log (D+S)\big)$. 
		\item\label{injection}  The   morphism 
		\begin{align}\label{Torelli}
		\tau_{1}: \ts_Y(-\log D)\rightarrow \ls^{-1}\otimes E^{n-1,1}
		\end{align}
		induced by the sub-Higgs sheaf $(\fs,\eta)$
		is  injective over $V_0$.
	\end{thmlist}
\end{thm}
\begin{proof}
	The proof is a continuation of that of \cref{thm:existence}, and we will adopt the same notations. 
	
	We first prove that for any \(y\in V\),  the set of \(z\in V\) with \( X_z \stackrel{{\rm bir}}{\sim} X_y \) is finite.    Take a polarization $\hs$ for $U\to V$ with the Hilbert polynomial $h$.  	Denote by \(\mathscr{P}_h(V) \) the set of such pairs \(( U\to V, \hs)\), up to isomorphisms and up to fiberwise numerical equivalence for \(\hs \). By \cite[Section 7.6]{Vie95}, there exists a coarse quasi-projective moduli scheme \(P_h \) for \(\mathscr{P}_h \), and thus the  family induces a morphism \(V\to P_h\).  By the assumption that the family \(U\to V \)  is effectively parametrized, the induced morphism \(V \to {P}_h \) is quasi-finite, which in turn shows that the set of  \(z\in V\) with \( X_z\) isomorphic to  \(X_y \) is finite. Note that a projective manifold of  general type   has  finitely  many minimal models. Hence   the set of \(z\in V'\) with \( X_z \stackrel{{\rm bir}}{\sim} X_y \) is finite as well.   
	
	Now we will choose the hypersurface in \eqref{eq:cyclic} carefully so that the cyclic cover construction in \cref{thm:existence} can provide the desired refined VZ Higgs bundle. 	Let $  Y'\supset V$ be the smooth  compactification in \cref{def: good compactification}.   By \cref{cor:main hyper}, for any given point $y\in V$ and any sufficiently ample line bundle $\as$ on $Y'$, there exists a   birational morphism $\nu:Y\to Y'$   and a new algebraic fiber space $f:X\to Y$ so that  
	one can find a hypersurface 
	\begin{align}\label{hyper for cyclic}
	H\in |\ell \Omega^n_{X/Y}(\log \Delta)-\ell (\nu\circ f)^*\as+\ell E|,\quad n:=\dim X-\dim Y
	\end{align}
	satisfying that
	\begin{itemize}[leftmargin=0.6cm] 
		\item the inverse image $D:=\nu^{-1}(Y'\setminus V)$ is a simple normal crossing divisor.
		\item There exists a reduced divisor $S$ so that $D+S$ is simple normal crossing, and $H\to Y$ is smooth over $V_0:=Y\setminus (D\cup S)$.
		\item The restriction  $\nu:\nu^{-1}(V_0)\to V_0$ is  an isomorphism. 
		\item The given point $y$ is contained in $V_0$. 
		\item The divisor $E$ is effective and   $f$-exceptional  with $f(E)\subset {\rm Supp}(D+S)$.
		\item For any $z\in V:=\nu^{-1}(V')$, the canonical bundle of the  fiber $X_z:={f}^{-1}(z)$ is big and nef.
		\item The restricted family $f^{-1}(V_0)\to V_0$ is smooth and effectively parametrized.
	\end{itemize}
	Here we set $\Delta:=f^*D$ and $\Sigma:=f^*S$. 	Write 
	$\ls:=\nu^*\as$. Now we take the cyclic cover with respect to  $H$ in \eqref{hyper for cyclic} instead of that in \eqref{eq:cyclic},  and perform the same construction of VZ Higgs $(\tilde{\es},\tilde{\theta})\supset (\fs,\tau)$  bundle as in \cref{thm:existence}.  \cref{p1,p2,p3} can be seen directly from the properties of $H$ and the cyclic construction.

	\cref{injection} has already appeared in \cite[Proposition 2.11]{PTW18} implicitly, and we give a proof here  for the sake of completeness. 
	Recall that both $Z$ and $H$ are   smooth over $V_0$. Denote by    $H_0:=H\cap f^{-1}(V_0)$,   $f_0:X_0=f^{-1}(V_0)\to V_0$, and $g_0:Z_0=g^{-1}(V_0)\to V_0$.  We have
	\begin{align} \nonumber
	F^{n,0}_{\upharpoonright V_0}&=  f_*\big(\Omega^{n}_{X/Y}(\log \Delta)\otimes \lc^{-1}\big)_{\upharpoonright V_0}=\oc_{V_0}\\\nonumber
	E^{n-1,1}_{\upharpoonright V_0}&= R^1 (g_0)_* ( \Omega^{n-1}_{Z_0/V_0})=R^1 (f_0)_* \big(\Omega^{n-1}_{X_0/V_0}\oplus  \bigoplus_{i=1}^{\ell-1} \Omega^{n-1}_{X_0/V_0}(\log H_0)\otimes  (K_{X_0/V_0}\otimes f_0^*\ls^{-1})^{-i}\big)\\\label{eq333}
	F^{n-1,1}_{\upharpoonright V_0}&= R^1 f_*\big(\Omega^{n-1}_{X/Y}(\log \Delta)\otimes \lc^{-1}\big)_{\upharpoonright V_0}=R^1 ({f_0})_*\big( \Omega^{n-1}_{X_0/V_0}\otimes K_{X_0/V_0}^{-1} \big)\simeq R^1 ({f_0})_*(\ts _{X_0/V_0}).
	\end{align}  
	Hence ${\tau_1}_{\upharpoonright V_0}$ factors through 
	\begin{align*}
	{\tau_1}_{\upharpoonright V_0}: &\ts_{V_0}\xrightarrow{\rho}  R^1 ({f_0})_*( \ts _{X_0/V_0})\xrightarrow{\simeq} R^1 ({f_0})_*\big( \Omega^{n-1}_{X_0/V_0}\otimes K_{X_0/V_0}^{-1} \big) \to \\
	& R^1 ({f_0})_*\big( \Omega^{n-1}_{X_0/V_0}(\log H_0)\otimes K_{X_0/V_0}^{-1} \big)  \to R^1 (g_0)_* ( \Omega^{n-1}_{Z_0/V_0})\otimes \ls^{-1},
	\end{align*} 
	where $\rho$ is the Kodaira-Spencer map. Although the intermediate objects in the above factorization might not be  locally free,  the induced $\cb$-linear  map by the sheaf morphism \({\tau_1}_{\upharpoonright V_0} \) at the   \(z\in V_0\)
	\[
	\tau_{1,z}:\ts_{Y, z}\to (\ls^{-1}\otimes E^{n-1,1})_{z}
	\]  
	coincides with the following composition of  \(\cb \)-linear maps between finite dimensional complex vector spaces
	\begin{align}\label{long}
	\tau_{1,z}: \ts_{Y, z}\xrightarrow{\rho_z}H^1(X_z, \ts_{X_z})\xrightarrow{\simeq} H^1(X_z, \Omega^{n-1}_{X_z}\otimes K^{-1}_{X_z})\xrightarrow{j_z}\\\nonumber
	H^1\big(X_z, \Omega^{n-1}_{X_z}(\log H_z)\otimes K^{-1}_{X_z}\big)\to H^1\big(Z_z, \Omega^{n-1}_{Z_z}\big).
	\end{align}
	To prove \cref{injection}, it then suffices to prove that each linear map in \eqref{long} is injective for any $z\in V_0$. 
	
	By the effective parametrization assumption, $\rho_z$ is injective.  
	The  map $j_z$ in \eqref{long} is the same as the \(H^1 \)-cohomology map of the   short exact sequence
	\[0\to K^{-1}_{X_z}\otimes\Omega^{n-1}_{X_z}\to K^{-1}_{X_z}\otimes\Omega^{n-1}_{X_z}(\log H_z)\to K^{-1}_{X_z\upharpoonright H_z}\otimes\Omega^{n-2}_{H_z}\to 0.\]
	Observe that $  K_{X_z \upharpoonright H_z}$ is big. Indeed,  this follows from that 
	\[
	\vol(K_{ {X}_z\upharpoonright H_z})=c_1(K_{ {X}_z\upharpoonright H_z})^{n-1} =c_1(K_{ {X}_z})^{n-1}\cdot H_z=\ell c_1(K_{ {X}_z})^{n}=\ell \vol(K_{ {X}_z})>0.
	\]
	Hence $j_z$ injective by the Bogomolov-Sommese vanishing theorem 
	\[
	H^0\big(H_z, K^{-1}_{X_z\upharpoonright H_z}\otimes\Omega^{d-2}_{H_z}\big)=0,
	\] 
	as  observed in \cite{PTW18}.  
	Since \(\psi_z:Z_z\to X_z \) is the cyclic cover obtained by  taking the  \(\ell \)-th roots out of the smooth hypersurface \(H_z\in |\ell K_{X_z}| \),  the morphism \(\psi \) is   finite. It follows from the degeneration of the Leray spectral sequence that
	\begin{align}\label{Leray}
	H^1(Z_{z}, \Omega^{n-1}_{Z_z})\simeq H^1\big(X_z,(\psi_z)_*\Omega^{n-1}_{Z_z}  \big)=H^1\big(X_z, \Omega^{n-1}_{X_z} \big)\oplus \bigoplus_{i=1}^{\ell -1} H^1\big(X_z, \Omega^{n-1}_{X_z}(\log H_z)\otimes K_{X_z}^{-i}  \big).
	\end{align}
	The last map in \eqref{long}  is therefore injective, for the cohomology group
	$H^1\big(X_z, \Omega^{n-1}_{X_z}(\log H_z)\otimes K_{X_z}^{-1}  \big)$ 
	is  a direct summand of \(H^1(Z_{z}, \Omega^{n-1}_{Z_z})\) by \eqref{Leray}. As a consequence, the  composition $\tau_{1,z}$ in \eqref{long} is injective at each point $z\in V_0$.   \cref{injection} is thus proved. 
\end{proof}

Let us explain how \cref{bimeromorphic,construction of Finsler,Higgs bundle} imply \cref{main}.
\begin{proof}[Proof of \cref{main}]
	We first take    a smooth compactification $Y\supset  V$ as in   \cref{def: good compactification},. 
	By  \cref{Higgs bundle}, for any given point $y\in V$,  there exists a birational morphism $\nu:Y'\to Y$ which is  isomorphic at $y$,   
	so that $D:=Y'\setminus \nu^{-1}(V)$ is a simple normal crossing divisor, and there exists a VZ Higgs bundle $(\tilde{\es},\tilde{\theta})$  whose log pole $D+S$ avoids $y':=\nu^{-1}(y)$. Moreover, by \cref{injection}, $\tau_1$ is injective at $y'$.  Applying \cref{construction of Finsler}, we can associate  $(\tilde{\es},\tilde{\theta})$ a Finsler metric \(F\) on \(\ts_{Y'}(-D)\)  which is positively  definite at \(y'\).  Moreover, if we think of  $F$   as a Finsler metric on $\nu^{-1}(V)$, it is negatively curved in the sense of \cref{negatively curved}. 
	Hence the base $V$ satisfies the conditions in \cref{bimeromorphic}, and we conclude that \(V\) is Kobayashi hyperbolic.
\end{proof}

\appendix
\section{$\mathbb{Q}$-mild reductions (by Dan Abramovich)} \label{appendix}
Let us work over complex number field $\mathbb{C}$. 

The main result in this appendix is the following:
\begin{thm}\label{thm:ab}
	Let $f_0: S_0 \to T_0$ be a projective family of smooth varieties with $T_0$ quasi-projective.
	\begin{thmlist}
		\item There are compactifications $S_0 \subset \cS$ and $T_0 \subset \cT$, with $\cS$ and $\cT$ Deligne-Mumford stacks with projective coarse moduli spaces, and a projective morphism $f:\cS \to \cT$ extending  $f_0$ which is a Koll\'ar family of slc varieties.
		\item Given a finite subset $Z \subset T_0$ there is a projective variety $W$ and finite surjective lci morphism $\rho:W \to \cT$, unramified over $Z$, such that $\rho^{-1} \cT^{sm} = W^{sm}$. 
	\end{thmlist}    
\end{thm} 
Here the notion of Koll\'ar family refers to the condition that the sheaf $\omega_{\cS / \cT}^{[m]}$ is flat and its formation  commutes with arbitrary base change for each $m$. We   refer the readers to \cite[Definition 5.2.1]{AH11} for further details.

Note that the pullback family $\cS \times_\cT W \to W$ is a Koll\'ar family of slc varieties compactifying the pullback $S_0 \times_{T_0} W_0 \to W_0$ of the original family to $W_0 :=W \times_\cT T_0$.

This is applied in the present paper, where some mild regularity assumption on $T_0$ and $W$ is required:

\begin{cor}[$\mathbb{Q}$-mild reduction] \label{def: good compactification} Assume further $T_0$ is smooth. For any given finite subset $Z\subset T_0$,  there exist
	\begin{thmlist}
		\item a compactification $T_0 \subset \underline{\cT}$ with $\underline{\cT}$ a regular projective scheme,  
		\item a simple normal crossings divisor $D \subset \underline{\cT}$ containing $\underline{\cT} \smallsetminus T_0$ and disjoint from $Z$, 
		\item a finite morphism $W \to \underline{\cT}$ unramified outside $D$, and
		\item A Koll\'ar family $S_W \to W$ of slc varieties extending the given family $S_0 \times_{\underline{\cT}}W$. 
	\end{thmlist}	
\end{cor}

The significance of these extended families is through their $\mathbb{Q}$-mildness property. Recall from \cite{AK00} that a family $S \to T$ is \emph{$\mathbb{Q}$-mild} if whenever $T_1 \to T$ is a dominant morphism with $T_1$ having at most Gorenstein canonical singularities, then the total space $S_1 = T_1 \times_S T$ has canonical singularities.
It was shown by Koll\'ar--Shepherd-Barron \cite[Theorem 5.1]{KSB88} and Karu  \cite[Theorem  2.5]{Kar00} that Koll\'ar families of slc varieties whose generic fiber has at most Gorenstein canonical singularities are $\mathbb{Q}$-mild.

The main result is proved using moduli of Alexeev stable maps.

Let $V$ be a projective variety. A morphism $\phi:U \to V$ is a \emph{stable map} if $U$ is slc and $K_U$ is $\phi$-ample. More generally, given  $\pi: U \to T$,  a morphism $\phi:U \to V$ is a \emph{stable map over $T$} or a \emph{family of stable maps parametrized by $T$} if $\pi$ is a Koll\'ar family of slc varieties and  $K_{U/T}$ is $\phi\times \pi$-ample. Note that this condition is very flexible and does not require the fibers to be of general type, although key applications in \cref{desired global,cor:main hyper} require some positivity of the fibers.

\begin{thm}[\!\protect{\cite[Theorem 1.5]{DR18}}] Stable maps form an algebraic stack $M(V)$ locally of finite type over $\mathbb{C}$, each of whose connected components is a proper global quotient stack with projective coarse moduli space.
\end{thm}

The existence of an algebraic stack satisfying the valuative criterion for properness was known to Alexeev, and can also be deduced directly from the results of \cite{AH11}, which presents it as a global quotient stack. The work \cite{DR18}  shows that the stack has bounded, hence proper components, admitting projective course moduli spaces. An algebraic approach for these statements is provided in \cite[Corollary 1.2]{Kar00}.


\begin{proof}[Proof of \cref{thm:ab}]
	\begin{enumerate}[leftmargin=0cm,itemindent=0.7cm,labelwidth=\itemindent,labelsep=0cm, align=left,label= {\rm (\roman*)},itemsep=0.07cm]
		\item 	Let $T_0 \subset T$  and $S_0 \subset S$ be  projective compactifications with $\pi: S \to T$ extending $f_0$. The family $S_0 \to T_0$ with the injective morphism $\phi: S_0 \to S$ is a family of stable maps into $S$, providing a morphism $T_0 \to M(S)$ which is in fact injective. Let $\cT$ be the closure of $T_0$.  Since $M(S)$ is proper, $\cT$ is proper. Let $\cS$ be the pullback of the universal family along $\cT \to M(S/T)$. Then $\cS\supset S_0$ is a compactification as needed.
		\item The existence of $W$ follows from the main result of \cite{KV04}. 	\qedhere	
	\end{enumerate} 
\end{proof}
\begin{proof}[Proof of \cref{def: good compactification}]
	Consider the coarse moduli space $\underline{\cT}$ of the stack $\cT$ provided by the first part of the main result. This might be singular, but by Hironaka's theorem  we may replace it by a resolution of singularities such that $D_\infty:=\underline{\cT}\smallsetminus T_0$ is a simple normal crossings divisor. Thus condition (\lowerromannumeral{1})  is satisfied.
	
	For each component $D_i\subset D_\infty$ denote by $m_i$ the ramification index of $\cT \to  \underline{\cT}$. In particular any covering $W \to \underline{\cT}$ whose ramification indices over $D_i$ are divisible by $m_i$ lifts along the generic point of $D_i$ to $\cT$. 
	
	Choosing a Kawamata covering package \cite{AK00} disjoint from $Z$ we obtain a simple normal crossings divisor $D$ as required by (\lowerromannumeral{2}), and finite covering $W \to \underline{\cT}$ as required by (\lowerromannumeral{3}), such that $W \to \underline{\cT}$ factors through $\cT$ at every generic point of $D_i$. 
	
	By the Purity Lemma \cite[Lemma 2.4.1]{AV02} the morphism $W \to \cT$ extends over all of $W$, hence we obtain a family $S_W \to W$ as required by (\lowerromannumeral{4}).
\end{proof}


\begin{thebibliography}{{Den}18b} 
	\providecommand{\url}[1]{\texttt{#1}}
	\providecommand{\urlprefix}{URL }
	
	\bibitem[AH11]{AH11}
	Dan Abramovich and Brendan Hassett.
	\newblock \enquote{Stable varieties with a twist.}
	\newblock In \enquote{Classification of algebraic varieties,} EMS Ser. Congr.
	Rep., 1--38,  (Eur. Math. Soc., Z\"urich2011).
	\newblock \urlprefix\url{http://dx.doi.org/10.4171/007-1/1}.
	
	\bibitem[AK00]{AK00}
	Dan Abramovich and Kalle Karu.
	\newblock \enquote{Weak semistable reduction in characteristic 0.}
	\newblock \emph{Invent. Math.} (2000) vol. 139~(2): 241--273.
	\newblock \urlprefix\url{https://doi.org/10.1007/s002229900024}.
	
	\bibitem[Ara04]{Ara04}
	Donu Arapura.
	\newblock \enquote{Frobenius amplitude and strong vanishing theorems for vector
		bundles.}
	\newblock \emph{Duke Math. J.} (2004) vol. 121~(2): 231--267.
	\newblock \urlprefix\url{http://dx.doi.org/10.1215/S0012-7094-04-12122-0}.
	\newblock With an appendix by Dennis S. Keeler.
	
	\bibitem[AV02]{AV02}
	Dan Abramovich and Angelo Vistoli.
	\newblock \enquote{Compactifying the space of stable maps.}
	\newblock \emph{J. Amer. Math. Soc.} (2002) vol.~15~(1): 27--75.
	\newblock \urlprefix\url{http://dx.doi.org/10.1090/S0894-0347-01-00380-0}.
	
	\bibitem[BD18]{DB18}
	S{\'e}bastien {Boucksom} and Simone {Diverio}.
	\newblock \enquote{{A note on Lang's conjecture for quotients of bounded
			domains}.}
	\newblock \emph{arXiv e-prints} (2018) arXiv:1809.02398.
	
	\bibitem[BD19]{BD19}
	Damian Brotbek and Ya~Deng.
	\newblock \enquote{Kobayashi hyperbolicity of the complements of general
		hypersurfaces of high degree.}
	\newblock \emph{Geometric and Functional Analysis} (2019)
	\urlprefix\url{http://dx.doi.org/10.1007/s00039-019-00496-2}.
	
	\bibitem[BHPS13]{BHPS13}
	Bhargav Bhatt, Wei Ho, Zsolt Patakfalvi, and Christian Schnell.
	\newblock \enquote{Moduli of products of stable varieties.}
	\newblock \emph{Compos. Math.} (2013) vol. 149~(12): 2036--2070.
	\newblock \urlprefix\url{http://dx.doi.org/10.1112/S0010437X13007288}.
	
	\bibitem[BP08]{BP08}
	Bo~Berndtsson and Mihai P{\u a}un.
	\newblock \enquote{Bergman kernels and the pseudoeffectivity of relative
		canonical bundles.}
	\newblock \emph{Duke Math. J.} (2008) vol. 145~(2): 341--378.
	\newblock \urlprefix\url{https://doi.org/10.1215/00127094-2008-054}.
	
	\bibitem[BP10]{BP10}
	Bo~{Berndtsson} and Mihai {P{\u{a}}un}.
	\newblock \enquote{{Bergman kernels and subadjunction}.}
	\newblock \emph{arXiv e-prints} (2010) arXiv:1002.4145.
	
	\bibitem[BPW17]{BPW17}
	Bo~{Berndtsson}, Mihai {Paun}, and Xu~{Wang}.
	\newblock \enquote{{Algebraic fiber spaces and curvature of higher direct
			images}.}
	\newblock \emph{arXiv e-prints} (2017) arXiv:1704.02279.
	
	\bibitem[Bro17]{Bro17}
	Damian Brotbek.
	\newblock \enquote{On the hyperbolicity of general hypersurfaces.}
	\newblock \emph{Publ. Math. Inst. Hautes \'Etudes Sci.} (2017) vol. 126: 1--34.
	\newblock \urlprefix\url{https://doi.org/10.1007/s10240-017-0090-3}.
	
	\bibitem[Cao19]{Cao16}
	Junyan Cao.
	\newblock \enquote{Albanese maps of projective manifolds with nef anticanonical
		bundles.}
	\newblock \emph{{Ann. Sci. \'Ec. Norm. Sup\'er. (4)}} (2019) vol.~52~(4):
	1137--1154.
	
	\bibitem[CDG19]{CDG19}
	Beno{\^\i}t {Cadorel}, Simone {Diverio}, and Henri {Guenancia}.
	\newblock \enquote{{On subvarieties of singular quotients of bounded domains}.}
	\newblock \emph{arXiv e-prints} (2019) arXiv:1905.04212.
	
	\bibitem[CK17]{CK17}
	Fabrizio {Catanese} and Yujiro {Kawamata}.
	\newblock \enquote{{Fujita decomposition over higher dimensional base}.}
	\newblock \emph{arXiv e-prints} (2017) arXiv:1709.08065.
	
	\bibitem[CKS86]{CKS86}
	Eduardo Cattani, Aroldo Kaplan, and Wilfried Schmid.
	\newblock \enquote{Degeneration of {H}odge structures.}
	\newblock \emph{Ann. of Math. (2)} (1986) vol. 123~(3): 457--535.
	\newblock \urlprefix\url{https://doi.org/10.2307/1971333}.
	
	\bibitem[CP15]{CP15b}
	Fr{\'e}d{\'e}ric {Campana} and Mihai {P{\u a}un}.
	\newblock \enquote{Orbifold generic semi-positivity: an application to families
		of canonically polarized manifolds.}
	\newblock \emph{Ann. Inst. Fourier (Grenoble)} (2015) vol.~65~(2): 835--861.
	\newblock \urlprefix\url{http://aif.cedram.org/item?id=AIF_2015__65_2_835_0}.
	
	\bibitem[CP16]{CP16}
	---{}---{}---.
	\newblock \enquote{Positivity properties of the bundle of logarithmic tensors
		on compact {K}{\" a}hler manifolds.}
	\newblock \emph{Compos. Math.} (2016) vol. 152~(11): 2350--2370.
	\newblock \urlprefix\url{http://dx.doi.org/10.1112/S0010437X16007442}.
		
	\bibitem[CP19]{CP15}
	---{}---{}---.
	\newblock \enquote{{Foliations with positive slopes and birational stability of
			orbifold cotangent bundles.}}
	\newblock \emph{{Publ. Math., Inst. Hautes \'Etud. Sci.}} (2019) vol. 129:
	1--49.
	
	\bibitem[CaoP17]{CP17}
	Junyan Cao and Mihai P\u{a}un.
	\newblock \enquote{Kodaira dimension of algebraic fiber spaces over abelian
		varieties.}
	\newblock \emph{Invent. Math.} (2017) vol. 207~(1): 345--387.
	\newblock \urlprefix\url{https://doi.org/10.1007/s00222-016-0672-6}.

	
	\bibitem[CRT19]{CRT19}
	Beno\^{\i}t Cadorel, Erwan Rousseau, and Behrouz Taji.
	\newblock \enquote{Hyperbolicity of singular spaces.}
	\newblock \emph{J. \'{E}c. polytech. Math.} (2019) vol.~6: 1--18.
	\newblock \urlprefix\url{http://dx.doi.org/10.5802/jep.85}.
	
	\bibitem[{Dem}97]{Dem97}
	Jean-Pierre {Demailly}.
	\newblock \enquote{Algebraic criteria for {K}obayashi hyperbolic projective
		varieties and jet differentials.}
	\newblock In \enquote{Algebraic geometry---{S}anta {C}ruz 1995,} vol.~62 of
	\emph{Proc. Sympos. Pure Math.}, 285--360,  (Amer. Math. Soc., Providence,
	RI1997).
	
	\bibitem[{Dem}20]{Dem18}
	---{}---{}---.
	\newblock \enquote{{Recent results on the Kobayashi and Green-Griffiths-Lang
			conjectures.}}
	\newblock \emph{{Jpn. J. Math. (3)}} (2020) vol.~15~(1): 1--120.
	
	\bibitem[{Den}17]{Den17}
	Ya~{Deng}.
	\newblock \enquote{{Applications of the Ohsawa-Takegoshi Extension Theorem to
			Direct Image Problems}.}
	\newblock \emph{arXiv e-prints} (2017) arXiv:1703.07279. To appear in \emph{
		Int. Math. Res. Not. IMRN}.
	
	\bibitem[{Den}18a]{Den18}
	---{}---{}---.
	\newblock \enquote{{Kobayashi hyperbolicity of moduli spaces of minimal
			projective manifolds of general type (with the appendix by Dan Abramovich)}.}
	\newblock \emph{arXiv e-prints} (2018) arXiv:1806.01666.
	
	\bibitem[{Den}18b]{Den18b}
	---{}---{}---.
	\newblock \enquote{{Pseudo Kobayashi hyperbolicity of base spaces of families
			of minimal projective manifolds with maximal variation}.}
	\newblock \emph{arXiv e-prints} (2018) arXiv:1809.05891.
	
	\bibitem[{Den}19]{Den19}
	---{}---{}---.
	\newblock \enquote{{Hyperbolicity of bases of log Calabi-Yau families}.}
	\newblock \emph{arXiv e-prints} (2019) arXiv:1901.04423.
	
	\bibitem[{Den}20]{Den20}
	---{}---{}---.
	\newblock \enquote{{Big Picard theorem and algebraic hyperbolicity for
			varieties admitting a variation of Hodge structures}.}
	\newblock \emph{arXiv e-prints} (2020) arXiv:2001.04426.
	
	\bibitem[DLSZ19]{DLSZ}
	Ya~{Deng}, Steven {Lu}, Ruiran {Sun}, and Kang {Zuo}.
	\newblock \enquote{{Picard theorems for moduli spaces of polarized varieties}.}
	\newblock \emph{arXiv e-prints} (2019) arXiv:1911.02973.
	
	\bibitem[DR18]{DR18}
	Ruadha{\'i} Dervan and Julius Ross.
	\newblock \enquote{Stable maps in higher dimensions.}
	\newblock \emph{Mathematische Annalen} (2018)
	\urlprefix\url{http://dx.doi.org/10.1007/s00208-018-1706-8}.
	
	\bibitem[GT84]{GT84}
	Phillip Griffiths and Loring Tu.
	\newblock \enquote{Curvature properties of the {H}odge bundles.}
	\newblock In \enquote{Topics in transcendental algebraic geometry ({P}rinceton,
		{N}.{J}., 1981/1982),} vol. 106 of \emph{Ann. of Math. Stud.}, 29--49,
	(Princeton Univ. Press, Princeton, NJ1984).
	
	\bibitem[Kar00]{Kar00}
	Kalle Karu.
	\newblock \enquote{Minimal models and boundedness of stable varieties.}
	\newblock \emph{J. Algebraic Geom.} (2000) vol.~9~(1): 93--109.
	
	\bibitem[Kas85]{Kas85}
	Masaki Kashiwara.
	\newblock \enquote{The asymptotic behavior of a variation of polarized {H}odge
		structure.}
	\newblock \emph{Publ. Res. Inst. Math. Sci.} (1985) vol.~21~(4): 853--875.
	\newblock \urlprefix\url{http://dx.doi.org/10.2977/prims/1195178935}.
	
	\bibitem[Kaw85]{Kaw85}
	Yujiro Kawamata.
	\newblock \enquote{Minimal models and the {K}odaira dimension of algebraic
		fiber spaces.}
	\newblock \emph{J. Reine Angew. Math.} (1985) vol. 363: 1--46.
	\newblock \urlprefix\url{http://dx.doi.org/10.1515/crll.1985.363.1}.
	
	\bibitem[KK08a]{KK08a}
	Stefan Kebekus and S\'{a}ndor~J. Kov\'{a}cs.
	\newblock \enquote{Families of canonically polarized varieties over surfaces.}
	\newblock \emph{Invent. Math.} (2008) vol. 172~(3): 657--682.
	\newblock \urlprefix\url{http://dx.doi.org/10.1007/s00222-008-0128-8}.
	
	\bibitem[KK08b]{KK08b}
	---{}---{}---.
	\newblock \enquote{Families of varieties of general type over compact bases.}
	\newblock \emph{Adv. Math.} (2008) vol. 218~(3): 649--652.
	\newblock \urlprefix\url{http://dx.doi.org/10.1016/j.aim.2008.01.005}.
	
	\bibitem[KK10]{KK10}
	---{}---{}---.
	\newblock \enquote{The structure of surfaces and threefolds mapping to the
		moduli stack of canonically polarized varieties.}
	\newblock \emph{Duke Math. J.} (2010) vol. 155~(1): 1--33.
	\newblock \urlprefix\url{http://dx.doi.org/10.1215/00127094-2010-049}.
	
	\bibitem[Kob98]{Kob98}
	Shoshichi Kobayashi.
	\newblock \emph{Hyperbolic complex spaces}, vol. 318 of \emph{Grundlehren der
		Mathematischen Wissenschaften [Fundamental Principles of Mathematical
		Sciences]},  (Springer-Verlag, Berlin1998).
	\newblock \urlprefix\url{https://doi.org/10.1007/978-3-662-03582-5}.
	
	\bibitem[KP17]{KP17}
	S\'{a}ndor~J. Kov\'{a}cs and Zsolt Patakfalvi.
	\newblock \enquote{Projectivity of the moduli space of stable log-varieties and
		subadditivity of log-{K}odaira dimension.}
	\newblock \emph{J. Amer. Math. Soc.} (2017) vol.~30~(4): 959--1021.
	\newblock \urlprefix\url{http://dx.doi.org/10.1090/jams/871}.
	
	\bibitem[KSB88]{KSB88}
	J{\'a}nos Koll{\'a}r and Nicholas~I. Shepherd-Barron.
	\newblock \enquote{Threefolds and deformations of surface singularities.}
	\newblock \emph{Invent. Math.} (1988) vol.~91~(2): 299--338.
	\newblock \urlprefix\url{http://dx.doi.org/10.1007/BF01389370}.
	
	\bibitem[KV04]{KV04}
	Andrew Kresch and Angelo Vistoli.
	\newblock \enquote{On coverings of {D}eligne-{M}umford stacks and surjectivity
		of the {B}rauer map.}
	\newblock \emph{Bull. London Math. Soc.} (2004) vol.~36~(2): 188--192.
	\newblock \urlprefix\url{http://dx.doi.org/10.1112/S0024609303002728}.
	
	\bibitem[Lan91]{Lan91}
	Serge Lang.
	\newblock \emph{Number theory. {III}}, vol.~60 of \emph{Encyclopaedia of
		Mathematical Sciences},  (Springer-Verlag, Berlin1991).
	\newblock \urlprefix\url{http://dx.doi.org/10.1007/978-3-642-58227-1}.
	\newblock Diophantine geometry.
	
	\bibitem[Mor87]{Mor85}
	Shigefumi Mori.
	\newblock \enquote{Classification of higher-dimensional varieties.}
	\newblock In \enquote{Algebraic geometry, {B}owdoin, 1985 ({B}runswick,
		{M}aine, 1985),} vol.~46 of \emph{Proc. Sympos. Pure Math.}, 269--331,
	(Amer. Math. Soc., Providence, RI1987).
	
	\bibitem[Pat12]{Pat12}
	Zsolt Patakfalvi.
	\newblock \enquote{Viehweg's hyperbolicity conjecture is true over compact
		bases.}
	\newblock \emph{Adv. Math.} (2012) vol. 229~(3): 1640--1642.
	\newblock \urlprefix\url{http://dx.doi.org/10.1016/j.aim.2011.12.013}.
	
	\bibitem[Pet84]{Pet84}
	C.~A.~M. Peters.
	\newblock \enquote{A criterion for flatness of {H}odge bundles over curves and
		geometric applications.}
	\newblock \emph{Math. Ann.} (1984) vol. 268~(1): 1--19.
	\newblock \urlprefix\url{http://dx.doi.org/10.1007/BF01463870}.
	
	\bibitem[PS17]{PS17}
	Mihnea Popa and Christian Schnell.
	\newblock \enquote{Viehweg's hyperbolicity conjecture for families with maximal
		variation.}
	\newblock \emph{Invent. Math.} (2017) vol. 208~(3): 677--713.
	\newblock \urlprefix\url{https://doi.org/10.1007/s00222-016-0698-9}.
	
	\bibitem[PT18]{PT14}
	Mihai P{\u{a}}un and Shigeharu Takayama.
	\newblock \enquote{Positivity of twisted relative pluricanonical bundles and
		their direct images.}
	\newblock \emph{J. Algebraic Geom.} (2018) vol.~27~(2): 211--272.
	\newblock \urlprefix\url{http://dx.doi.org/10.1090/jag/702}.
	
	\bibitem[PTW19]{PTW18}
	Mihnea {Popa}, Behrouz {Taji}, and Lei {Wu}.
	\newblock \enquote{Brody hyperbolicity of base spaces of certain families of
		varieties.}
	\newblock \emph{Algebra and Number Theory} (2019) vol.~13~(9): 2205--2242.
	\newblock \urlprefix\url{http://dx.doi.org/10.2140/ant.2019.13.2205}.
	
	\bibitem[Rou16]{Rou16}
	Erwan Rousseau.
	\newblock \enquote{Hyperbolicity, automorphic forms and {S}iegel modular
		varieties.}
	\newblock \emph{Ann. Sci. \'{E}c. Norm. Sup\'{e}r. (4)} (2016) vol.~49~(1):
	249--255.
	\newblock \urlprefix\url{http://dx.doi.org/10.24033/asens.2281}.
	
	\bibitem[Roy75]{Roy74}
	H.~L. Royden.
	\newblock \enquote{Intrinsic metrics on {T}eichm\"uller space.}
	\newblock \emph{Proceedings of the {I}nternational {C}ongress of
		{M}athematicians ({V}ancouver, {B}. {C}., 1974), {V}ol. 2} (1975) 217--221.
	
	\bibitem[Sch73]{Sch73}
	Wilfried Schmid.
	\newblock \enquote{Variation of {H}odge structure: the singularities of the
		period mapping.}
	\newblock \emph{Invent. Math.} (1973) vol.~22: 211--319.
	\newblock \urlprefix\url{https://doi.org/10.1007/BF01389674}.
	
	\bibitem[{Sch}12]{Sch12}
	Georg {Schumacher}.
	\newblock \enquote{Positivity of relative canonical bundles and applications.}
	\newblock \emph{Invent. Math.} (2012) vol. 190~(1): 1--56.
	\newblock \urlprefix\url{https://doi.org/10.1007/s00222-012-0374-7}.
	
		\bibitem[{Sch}18]{Sch17}
		---{}---{}---.
	\newblock \enquote{{Moduli of canonically polarized manifolds, higher order
			Kodaira-Spencer maps, and an analogy to Calabi-Yau manifolds.}}
	\newblock In \enquote{{Uniformization, Riemann-Hilbert correspondence,
			Calabi-Yau manifolds and Picard-Fuchs equations.}}, 369--399,  (Somerville,
	MA: International Press; Beijing: Higher Education Press2018).
	
	\bibitem[Schn17]{Schn17}
	Christian Schnell.
	\newblock \enquote{{On a theorem of Campana and P\u{a}un}.}
	\newblock \emph{{Épijournal de Géométrie Algébrique}} (2017) vol. {Volume
		1}.
	\newblock \urlprefix\url{https://epiga.episciences.org/3871}.
	

	
	\bibitem[Siu86]{Siu86}
	Yum-Tong Siu.
	\newblock \enquote{Curvature of the {W}eil-{P}etersson metric in the moduli
		space of compact {K}\"ahler-{E}instein manifolds of negative first {C}hern
		class.}
	\newblock In \enquote{Contributions to several complex variables,} Aspects
	Math., E9, 261--298,  (Friedr. Vieweg, Braunschweig1986).
	
	\bibitem[Siu15]{Siu15}
	---{}---{}---.
	\newblock \enquote{Hyperbolicity of generic high-degree hypersurfaces in
		complex projective space.}
	\newblock \emph{Invent. Math.} (2015) vol. 202~(3): 1069--1166.
	\newblock \urlprefix\url{http://dx.doi.org/10.1007/s00222-015-0584-x}.
	
	\bibitem[Ste77]{Ste77}
	Joseph Steenbrink.
	\newblock \enquote{Mixed {H}odge structure on the vanishing cohomology.}
	\newblock In \enquote{Real and complex singularities ({P}roc. {N}inth {N}ordic
		{S}ummer {S}chool/{NAVF} {S}ympos. {M}ath., {O}slo, 1976),} 525--563,
	(Sijthoff and Noordhoff, Alphen aan den Rijn1977).
	
	\bibitem[TY15]{TY14}
	Wing-Keung To and Sai-Kee Yeung.
	\newblock \enquote{Finsler metrics and {K}obayashi hyperbolicity of the moduli
		spaces of canonically polarized manifolds.}
	\newblock \emph{Ann. of Math. (2)} (2015) vol. 181~(2): 547--586.
	\newblock \urlprefix\url{https://doi.org/10.4007/annals.2015.181.2.3}.
	
	\bibitem[TY18]{TY18}
	---{}---{}---.
	\newblock \enquote{Augmented {W}eil-{P}etersson metrics on moduli spaces of
		polarized {R}icci-flat {K}\"{a}hler manifolds and orbifolds.}
	\newblock \emph{Asian J. Math.} (2018) vol.~22~(4): 705--727.
	\newblock \urlprefix\url{http://dx.doi.org/10.4310/AJM.2018.v22.n4.a6}.
	
	\bibitem[Vie90]{Vie90}
	Eckart Viehweg.
	\newblock \enquote{Weak positivity and the stability of certain {H}ilbert
		points. {II}.}
	\newblock \emph{Invent. Math.} (1990) vol. 101~(1): 191--223.
	\newblock \urlprefix\url{https://doi.org/10.1007/BF01231501}.
	
	\bibitem[Vie95]{Vie95}
	---{}---{}---.
	\newblock \emph{Quasi-projective moduli for polarized manifolds}, vol.~30 of
	\emph{Ergebnisse der Mathematik und ihrer Grenzgebiete (3) [Results in
		Mathematics and Related Areas (3)]},  (Springer-Verlag, Berlin1995).
	\newblock \urlprefix\url{https://doi.org/10.1007/978-3-642-79745-3}.
	
	\bibitem[VZ01]{VZ01}
	Eckart Viehweg and Kang Zuo.
	\newblock \enquote{On the isotriviality of families of projective manifolds
		over curves.}
	\newblock \emph{J. Algebraic Geom.} (2001) vol.~10~(4): 781--799.
	
	\bibitem[VZ02]{VZ02}
	---{}---{}---.
	\newblock \enquote{Base spaces of non-isotrivial families of smooth minimal
		models.}
	\newblock In \enquote{Complex geometry ({G}\"ottingen, 2000),} 279--328,
	(Springer, Berlin2002).
	
	\bibitem[VZ03]{VZ03}
	---{}---{}---.
	\newblock \enquote{On the {B}rody hyperbolicity of moduli spaces for
		canonically polarized manifolds.}
	\newblock \emph{Duke Math. J.} (2003) vol. 118~(1): 103--150.
	\newblock \urlprefix\url{https://doi.org/10.1215/S0012-7094-03-11815-3}.
	
	\bibitem[Wol86]{Wol86}
	Scott~A. Wolpert.
	\newblock \enquote{Chern forms and the {R}iemann tensor for the moduli space of
		curves.}
	\newblock \emph{Invent. Math.} (1986) vol.~85~(1): 119--145.
	\newblock \urlprefix\url{http://dx.doi.org/10.1007/BF01388794}.
	
	\bibitem[Yam19]{Yam18}
	Katsutoshi Yamanoi.
	\newblock \enquote{Pseudo {K}obayashi hyperbolicity of subvarieties of general
		type on abelian varieties.}
	\newblock \emph{J. Math. Soc. Japan} (2019) vol.~71~(1): 259--298.
	\newblock \urlprefix\url{http://dx.doi.org/10.2969/jmsj/75817581}.
	
	\bibitem[Zuc84]{Zuc84}
	Steven Zucker.
	\newblock \enquote{Degeneration of {H}odge bundles (after {S}teenbrink).}
	\newblock In \enquote{Topics in transcendental algebraic geometry ({P}rinceton,
		{N}.{J}., 1981/1982),} vol. 106 of \emph{Ann. of Math. Stud.}, 121--141,
	(Princeton Univ. Press, Princeton, NJ1984).
	
	\bibitem[Zuo00]{Zuo00}
	Kang Zuo.
	\newblock \enquote{On the negativity of kernels of {K}odaira-{S}pencer maps on
		{H}odge bundles and applications.}
	\newblock \emph{Asian J. Math.} (2000) vol.~4~(1): 279--301.
	\newblock \urlprefix\url{https://doi.org/10.4310/AJM.2000.v4.n1.a17}.
	\newblock Kodaira's issue.
	
\end{thebibliography}

\end{document}